\documentclass[
reprint,
onecolumn
]{revtex4-2}

\makeatletter
\renewcommand{\p@subsection}{}  
\usepackage[utf8]{inputenc}
\usepackage{hyperref}
\usepackage{xcolor}
\usepackage[margin=1in]{geometry}

\usepackage{algorithm}
\usepackage{algpseudocode}
\usepackage[subrefformat=parens]{subcaption}

\usepackage{amsmath,amstext,amssymb} 
\usepackage{amsthm}
\usepackage[noabbrev]{cleveref}  

\newtheorem{theorem}{Theorem}[section]

\newtheorem{prop}[theorem]{Proposition}

\newtheorem{example}[theorem]{Example}
\Crefname{proposition}{Proposition}{Propositions}
\Crefname{prop}{Proposition}{Propositions}

\usepackage{makecell}

\numberwithin{table}{section}
\numberwithin{figure}{section}
\numberwithin{algorithm}{section}
\numberwithin{equation}{section}
\numberwithin{subsection}{section}

\usepackage{booktabs, multirow, graphicx, bbm,
  tabularx, url, graphicx, siunitx, ulem, booktabs}

\DeclareMathOperator*{\Argmin}{Argmin}
\DeclareMathOperator*{\trace}{trace}
\DeclareMathOperator*{\kvec}{vec}
\DeclareMathOperator*{\supp}{supp}

\DeclareMathOperator*{\relgapMin}{\text{rel\_gap}_{\text{min}}}
\DeclareMathOperator*{\relgap}{\text{rel\_gap}}
\newcommand{\cI}{\mathcal{I}}
\newcommand{\cT}{\mathcal{T}}
\newcommand{\cN}{\mathcal{N}}
\newcommand{\cR}{\mathcal{R}}
\newcommand{\cC}{\mathcal{C}}
\newcommand{\cU}{\mathcal{U}}

\usepackage{array}
\newcommand{\PreserveBackslash}[1]{\let\temp=\\#1\let\\=\temp}
\newcolumntype{R}[1]{>{\PreserveBackslash\raggedleft}p{#1}}

\begin{document}

\title{Hardware-Compatible Single-Shot Feasible-Space Heuristics for Solving the Quadratic Assignment Problem}

\author{%
	Haesol Im,\textsuperscript{1} 
    Chan-Woo Yang,\textsuperscript{1} 
    Moslem Noori,\textsuperscript{1}
    Dmitrii Dobrynin,\textsuperscript{2,3}
    Elisabetta Valiante,\textsuperscript{1}
    Giacomo Pedretti,\textsuperscript{4}
    Arne Heittmann,\textsuperscript{2}
    Thomas Van Vaerenbergh,\textsuperscript{5}
    Masoud Mohseni,\textsuperscript{4}
    John Paul Strachan,\textsuperscript{2,3}
    Dmitri Strukov,\textsuperscript{6}
    Raymond Beausoleil,\textsuperscript{4}
    and Ignacio Rozada\textsuperscript{1,}
}
\thanks{Corresponding author: \href{mailto:ignacio.rozada@1qbit.com}{ignacio.rozada@1qbit.com}}

\affiliation{
\textsuperscript{1}1QB Information Technologies (1QBit), Vancouver, BC, Canada\\
\textsuperscript{2}Institute for Neuromorphic Compute Nodes (PGI-14), Peter Gr\"unberg Institute, Forschungszentrum J\"ulich GmbH, J\"ulich,  Germany\\
\textsuperscript{3}RWTH Aachen University, Aachen, Germany \\
\textsuperscript{4}HPE Labs, Hewlett Packard Enterprise, Milpitas, CA, USA\\
\textsuperscript{5}HPE Labs, Hewlett Packard Enterprise, Brussels, Belgium\\
\textsuperscript{6}University of California, Santa Barbara, CA, USA 
}

\date{\today}

\begin{abstract}
Research into the development of special-purpose computing architectures designed to solve quadratic unconstrained binary optimization (QUBO) problems has flourished in recent years. It has been demonstrated in the literature that such special-purpose solvers can outperform traditional complementary metal--oxide--semiconductor architectures by orders of magnitude with  respect to timing metrics on synthetic problems. 
However, they face challenges with constrained problems such as the quadratic
assignment problem (QAP), where mapping to binary formulations such as QUBO introduces overhead and limits parallelism. In-memory computing (IMC) devices, such as memristor-based analog Ising machines, offer significant speed-ups and efficiency gains over traditional
CPU-based solvers, particularly for solving combinatorial optimization problems.
In this work, we present a novel  hardware-aware QAP optimization framework designed for IMC hardware. 
By co-designing the local search heuristic with the underlying hardware, we exploit the intrinsic massive parallelism that allows
for computing of full neighbourhoods simultaneously to make update
decisions. We ensure binary solutions remain feasible by selecting local moves
that lead to neighbouring feasible solutions, leveraging feasible-space search
heuristics and the underlying structure of a given problem. 
Our approach is compatible with both digital computers and analog hardware.
We demonstrate its effectiveness in CPU
implementations by comparing it with state-of-the-art heuristics for solving the QAP.
\end{abstract}

\maketitle

\section{Introduction}

The quadratic assignment problem (QAP) was originally introduced in 1957 by Koopmans and Beckmann~\cite{Koopmans57}. In the QAP, the goal is to assign a number of facilities to an equal number of locations such that the cost associated with flows (e.g., energy or goods) between the facilities is minimized. This problem has remained a popular research topic in the literature due to its ubiquity in various domains, such as facility location, scheduling, process communications, warehouse management, and other fields; see, for example, Refs.~\cite{finke1987quadratic, li2016dynamic, carvalho2006microarray, burkard1998quadratic}. 
In its  native formulation, the solution space of the QAP is represented as comprising all possible permutations, where each permutation represents distinct one-to-one assignments of facilities to locations. An instance of the QAP may be denoted as $\mathrm{QAP}(F, D)$, where  $F = [F_{ij}]_{n \times n}$ and $D = [D_{kl}]_{n \times n}$ represent the flow and the distance matrices between $n$ facilities and $n$ locations, respectively. The goal is to find the assignment of facilities to locations such that the total flow cost is minimized. An assignment decision is feasible if and only if it is a permutation. Given a set $\Pi_n$ of all permutations of $n$ facilities, the aim of $\mathrm{QAP}(F, D)$ is to find an optimal permutation $\pi^* \in \Pi_n$ such that
\begin{equation} 
\label{eq:KoopBeckFormulation}
    \pi^* \in  \Argmin_{\substack{\pi \in \Pi_n}} \left\{ \sum_{i=1}^{n} \sum_{j=1}^{n} F_{ij} D_{\pi(i)\pi(j)} \right\} .
\end{equation}

Research on combinatorial optimization has focused primarily on benchmarking heuristics and algorithms running on traditional complementary metal--oxide--semiconductor (CMOS) hardware: mainly CPUs, but increasingly also GPUs, field-programmable gate arrays (FPGA), and even some custom application-specific integrated chip (ASIC) implementations of various algorithms. The performance of CPUs has doubled approximately every two years since the early 1970s, an empirical observation that has become known as Moore's law. This steady exponential increase in available computing power is expected to slow down in the current decade because the density of transistors in integrated circuits cannot grow indefinitely beyond fundamental physical limits at the nanometre scale. Relatedly, Dennard's scaling law states that, as transistors shrink, their power density remains constant, allowing for increased computing efficiency per chip area.  Combined with Moore's law, it has resulted in a doubling in performance per joule every 18 months. However, Dennard's scaling law began to break down in the mid-2010s. For this reason, attention is shifting towards alternative computing paradigms, such as GPUs, FPGAs, and ASICs, as well as to devices that exploit quantum  mechanical effects, and devices that exploit analog effects to perform very energy-efficient computations~\cite{cai2020power, fahimi2021combinatorial}. D-Wave Systems's quantum annealer~\cite{Johnson11}, Fujitsu's Digital Annealer~\cite{Matsubara17,Aramon19}, NTT's coherent Ising machine~\cite{Hamerly18}, Hitachi's FPGA-based Ising computers~\cite{Okuyama16,Yoshimura17,Okuyama19}, LightSolver's light processing units~\cite{meirzada2022lightsolver}, and Toshiba's simulated bifurcation machine~\cite{Goto19} are all examples of new technologies that are expected to benefit the field of computationally intensive discrete optimization problems, and can be generally referred to as in-memory computing (IMC) hardware accelerators. Application-driven analog IMC hardware devices that embed problem data within a crossbar structure are another promising approach~\cite{doi:10.1126/science.adi9405, bhattacharya2024computing}. In-memory computing hardware accelerators are known for their energy-efficient computation~\cite{scaleIMC_TSP_Lu_2023,pedretti2025solvingSAT,zhang2024distributedbinaryoptimizationinmemory} and, when combined with application-specific algorithms, can fully exploit parallelism, leading to significantly improved computation times and reduced computational energy consumption. In-memory computing hardware typically benefits heuristics that can tolerate limited precision, but has limitations, for example, implementing algorithms with complex decision heuristics, such as backtracking.

The majority of the novel optimization hardware solutions that have been developed in recent years aim to solve instances of quadratic unconstrained binary optimization (QUBO) problems, or using the equivalent Ising formulation. Such devices are promoted as fast general-purpose  optimization solvers because a wide variety of hard problems can be mapped to QUBO form in polynomial time~\cite{10.3389/fphy.2014.00005}. The literature that has showcased the progress of these  new devices has not, however, demonstrated their efficiency in handling various practical industrial optimization problems. New computational devices are still maturing and the technological limitations, such as limited numbers of bits, qubits, connectivity, low precision, and noise will likely be addressed in the future. Mapping real-world problems into QUBO form presents other challenges that cannot be fixed by hardware advances. The reformulation into QUBO, particularly for constrained problems, often results in significantly harder problems by creating a much larger, more rugged solution landscape that includes both feasible and infeasible solutions~\cite{Dobrynin:2024iyi}. These infeasible solutions are typically cast as a form of penalty in the objective function~\cite{Glover2022}, which results in an objective function with many local optima. Searching in a mixed landscape of feasible and infeasible solutions has the key disadvantage that intermediary infeasible solutions must be traversed in order to move from one feasible solution to another, and even sophisticated algorithms struggle to find good solutions~\cite{MOHSENI2021}. Furthermore, these devices face challenges with constrained problems like the QAP, where mapping to binary formulations such as QUBO introduces overhead and limits parallelism. Recent work has shown that state-of-the-art heuristics for solving the QAP can be mapped to FPGA and custom ASIC architectures, making use of massive parallelization to evaluate all possible local updates and guide a stochastic local search heuristic~\cite{MAQOBagherbeik_2022,scaleIMC_TSP_Lu_2023}.

In this paper, we present a hardware-aware co-design framework, tailored for IMC hardware, that explores the feasible space in the binary formulation of the QAP, and performs single-shot evaluations of the full-neighbourhood solution. Our approach enables massive parallelism by enabling a one-step, single-shot evaluation of the full neighbourhood or local gradient, enabling optimal local decisions, which can be made efficiently on a single computing core. When parallel search is discussed in the context of QAP heuristics, it is typically associated with leveraging multiple parallel processors to run independent heuristic searches~\cite{10.1007/978-3-319-39636-1_4}. In our approach, the full gradient is evaluated in parallel within a single iteration of the algorithm; see, for example, Ref.~\cite{pedretti2025solvingSAT}. Binary solutions remain feasible by selecting local moves that lead to neighbouring feasible solutions, enabling the use of feasible-space search heuristics.  As a result, our approach achieves an exponential reduction in search space compared with QUBO implementations, along with a much simpler energy landscape. Additionally, our method is compatible with both digital computers and IMC hardware accelerators. To evaluate the performance of our co-designed algorithm, we benchmark its effectiveness, comparing it to CPU-based implementations of well-known heuristics for solving the QAP.

\subsection{Applications of the QAP}

The QAP model serves as a robust combinatorial optimization model capable of resolving complex logistical and design challenges that naturally involve the concepts of flow and distance, providing decision support on resource allocation, scheduling, and infrastructure design. One of the more well-known applications of the QAP is in logistics and manufacturing, primarily targeting the minimization of transportation and handling costs. This includes facility layout design, where elements such as work centres and departments are arranged within a facility to minimize the total travel distance between them~\cite{Dickey_1972Cbau, FU_KAKU_1997Mwam, BenjaafarSaifallah2002MaAo, CHIANG1998457}. Insights derived from solving the QAP typically suggest that work centres with high flow can be placed close to each other. Similarly, in warehouse optimization problems, such as storage location assignment, the QAP is applied to assign products to specific storage locations in order to minimize the total handling costs and maximize the utilization of storage space~\cite{Silva2022Optimization, HausmanWarrenH1976OSAi}.

The QAP also has applications in the hardware sector for chip and circuit board design, which focuses on minimizing communication overhead and wiring length. For instance,  when designing printed circuit boards or computer backboards, the QAP can be used to minimize the cumulative length of interconnecting wires, reducing manufacturing costs~\cite{53bbb073-5adc-3787-97ba-b0c9dae99c60, EMANUEL2012525, Krarup1978}. A contemporary application is in the physical design of quantum processing units, where the QAP framework is utilized to efficiently determine the optimal placement and connectivity of computational qubits~\cite{dury2020quboformulationqubitallocation, chiew2025optimalfermionqubitmappingsquadratic}.

Another real-world application of the QAP is to assign entities---such as employees, departments, and physical units---to locations arising in service industries. For instance, optimizing the placement of medical departments to specific areas in a hospital~\cite{CUBUKCUOGLU2021102952} or minimizing the total cost incurred by  travel distance and flow of patients~\cite{Elshafei_1977_hospital} can be desired. The QAP can be leveraged to suggest zone placements that mitigate the risk of cross-infections by minimizing unnecessary mixing of patients~\cite{Seva2025Bed}. In transportation operations, the QAP model is applied to assign connecting flights to airport gates, minimizing travel distance of passengers  or luggage transport~\cite{Bouras2014}.

Beyond these industrial applications, the mathematical structure of the QAP extends to various combinatorial problems~\cite{10.5120/16825-6584, Cela98, Commander2005QAPThesis, 0b2b06e1-2775-327a-843f-af24a0afab1a,LOIOLA2007657,heffley1977QAPrunner}. For example, the QAP can be used to find stable molecular conformations by assigning molecular fragments to structural positions~\cite{Phillips1994}. The QAP also exhibits structural similarities with numerous other classical discrete optimization problems, such as the travelling salesperson problem, the graph partitioning problem, and the maximum clique problem.

\subsection{Notation}
\label{sec:notation}

Throughout this paper, we use the following notation. 
We let $\mathbb{R}^{n}$ and $\mathbb{R}^{m\times n}$ be  sets of real-valued vectors of length $n$ and $m \times n$ matrices, respectively.
The set of $n \times n$ symmetric matrices is denoted by $\mathbb{S}^n$. 
Given a matrix $X\in \mathbb{R}^{m\times n}$,
let $\kvec(X)\in \mathbb{R}^{m  n}$ denote the vector formed by sequentially stacking columns of $X$.
The $n \times n$ identity matrix is denoted by $I_n$. We use $e_i$ to  denote the standard unit vector in an arbitrary dimension, omitting the dimension specification when it is clear from the context. 
For the all-ones vector of length $n$, the notation $\mathbf{1}_n$ is used.
Given a vector $x\in \mathbb{R}^n$, $x_i$ represents the $i$-th component of the vector $x$.
Given $A\in \mathbb{R}^{m\times n}$,  
$A(i,j)$ and $A_{i,j}$ refer to the $(i,j)$-th element of $A$ and $A^T \in \mathbb{R}^{n\times m}$ refers to the transpose of $A$.
The $i$-th row and $i$-th column of $A$ are represented by $A(i,:)$ and $A(:,i)$, respectively.
Given a vector $x\in \mathbb{R}^n$, $\supp(x)$ denotes the support of $x$, that is, $\supp(x)=\{ i: x_i \ne 0\} \subseteq \{0,1,\ldots, n-1\}$.
The notation $\langle \cdot , \cdot \rangle$ 
is used to denote the standard inner product and
 $\otimes$ is used to represent the Kronecker product.
The symbol $\circ$ is used for the Hadamard product (i.e., the element-wise product) between two vectors or matrices. 
For the set of permutations of length $n$, and the set of $n \times n$ permutation matrices, $\Pi_n$ and $P_n$ are used, respectively.
We summarize the notation we have defined and provide examples of its use in  \Cref{sec:SymbolExample}.

\subsection{Contributions}

The contributions of this paper are as follows. 
\begin{enumerate}
\item 
We provide an efficient construction of the full neighbourhood at a given feasible binary solution. 
\item 
We design a parallel evaluation framework that computes objective differences across all neighbouring solutions in one iteration.
\item 
We introduce a hardware-aware co-designed framework tailored to the architectural requirements of application-driven IMC hardware.
\end{enumerate}

\subsection{Outline}

\Cref{sec:background} provides a background on the QAP and a discussion on the motivation for using the feasible-space method.
In \Cref{sec:localUpdate}, we show formulas that leverage problems' underlying structure to optimize performance. 
\Cref{sec:FNSearch} introduces the novel full-neighbourhood search that explores an entire neighbourhood in parallel, potentially benefiting from IMC hardware. 
The benchmarking heuristics that we developed and tested, along with our experimental results, are outlined in \Cref{sec:Numerics}. We summarize our work in \Cref{sec:conclusion}.

\section{Background}
\label{sec:background}

In this section, we review existing approaches for solving the QAP and explain the rationale behind pursuing the feasible-space method. 
Being an NP-hard problem, the QAP is commonly tackled using local-search heuristic algorithms, such as those presented in Refs.~\cite{DreznerZvi2003ANGA,BashiriMahdi2012Eham,AntColoniesQAP99}. They are CPU-based solvers that offer fast runtimes,  and have been shown to produce near-optimal solutions. Recently, hardware accelerators have been proposed as a means to solve challenging combinatorial optimization problems; see, for example, Ref.~\cite{Dobrynin:2024iyi}. 
Problems are often formulated in  QUBO form in order to use these devices.
Some examples of problems that have been solved by such devices using a QUBO formulation include the max-cut problem, graph colouring problem~\cite{10.3389/fphy.2014.00005}, $k$-satisfiability problem~\cite{Sat2QUBO16, bhattacharya2024computing}, and the travelling salesperson problem~\cite{Warren2020TST_QUBO}. 
It also possible to map the QAP to a binary formulation, which has been used to solve it with QUBO-based IMC hardware accelerators. 
In this section, we provide an overview of the binary formulation of the QAP, and discuss the motivation for designing the feasible-space method within the binary space.

\subsection{Binary Formulation of the QAP}

To solve the QAP using an Ising hardware solver to solve a QUBO problem, it needs to be cast into a binary formulation. Let $x_{ik}$ be a binary decision variable that is equal to $1$ if facility $i$ is assigned to location $k$, and $0$ otherwise. The QUBO formulation of the problem \eqref{eq:KoopBeckFormulation} can be written using the following objective function $E$:
\begin{equation}
\label{qubo_formulation}
\min_{x \in\{0,1\}^{n^2}}  
\sum_{i=1}^{n} \sum_{k=1}^{n} \sum_{j=1}^{n} \sum_{l=1}^{n} 
D_{ij} F_{kl} x_{ik} x_{jl} + \lambda\left[\sum_{k=1}^{n}\left(\sum_{i=1}^{n} x_{ik}-1\right)^2 + \sum_{i=1}^{n}\left(\sum_{k=1}^{n} x_{ik}-1\right)^2\right] .
\end{equation}
A penalty multiplier $\lambda$ is needed in order to ensure that only one facility is assigned to each location and vice versa. The parameter $\lambda$ should be chosen sufficiently large such that for every infeasible assignment $x_{\mathrm{inf}}$, there is at least one feasible assignment $x_{\mathrm{f}}$, where $E(x_{\mathrm{inf}}) > E(x_{\mathrm{f}})$. This yields $E(x_{\mathrm{inf}}) > \lambda$ because the flow and the distance matrices are non-negative and the assignment feasibility violation value is at least two. Further, let $u$ be an upper bound such that $E(x_{\mathrm{f}}) < u$ for every feasible solution (i.e., permutation) $x_{\mathrm{f}}$. Setting $\lambda > u$ would ensure every non-permutation solution evaluates to a value greater than the upper bound $u$, and, as a result, the objective values of all permutation solutions are smaller than all non-permutation solutions' objective values. The parameter $u$ is determined by sending the largest outgoing and incoming flows of each facility over the longest distances. This results in a loose bound on $\lambda$, and it might be possible to find a tighter bound. Thus, the penalty multiplier can be set to \(\lambda = \alpha u\), and the parameter \(\alpha \in [0,1]\) is tuned for optimal performance of the algorithm. Of note, with a tight bound on $\lambda$, all feasible solutions will be local optima in the extended binary solution space that also considers infeasible solutions.

\subsubsection{Solution Space}

There are multiple issues that contribute to making the binary formulation of the QAP much more challenging to solve than the native formulation \eqref{eq:KoopBeckFormulation}, at least when using a standard single bit-flip QUBO solver.
One of the main issues with the binary formulation is the exponential increase of the solution space. In the native QAP formulation, the solution space is given by the set of all permutations, \(\Pi_n\). For a given value of \(n\), there are \(n!\) possible solutions, whereas in the binary formulation, the solutions are expressed as binary vectors of size \(n^2\); therefore, the total solution space has a size of \(2^{n^2}\). \Cref{tab:qubo_qap_solution_space} shows a comparison of the size of the solution space between the native QAP and the binary formulations for several problem sizes.

\begin{table}[ht]
    \centering
    \begin{tabular} {llll}
\toprule
\multicolumn{1}{l}{Number of Facilities \ \ \ } & 
        \multicolumn{1}{l}{Number of} & 
        \multicolumn{1}{l}{Solution Space:} \ \ \ &
        \multicolumn{1}{l}{Solution Space:} \\
        \multicolumn{1}{l}{and Locations ($n$)} & 
        \multicolumn{1}{l}{Binary Variables ($n^2$) \ \  \ } & 
        \multicolumn{1}{l}{Native ($n!$)} & 
        \multicolumn{1}{l}{Binary ($2^{n^2}$)} \\
\midrule
         20  & 400   & \(\sim10^{18}\)   & \(\sim10^{120}\)   \\ 	
         40  & 1600  & \(\sim10^{48}\)  & \(\sim10^{482}\)  \\ 
         60  & 3600  & \(\sim10^{82}\)  & \(\sim10^{1084}\)  \\ 
         80  & 6400  & \(\sim10^{119}\)  & \(\sim10^{1927}\)  \\ 
         100 & 10,000 & \(\sim10^{158}\) & \(\sim 10^{3010}\) \\ \bottomrule  
    \end{tabular}
    \caption{Number of variables and size of the solution space in the native and binary formulations for various problem sizes.}
    \label{tab:qubo_qap_solution_space}
\end{table}

The exponentially larger solution space of the QUBO formulation comprises mostly infeasible solutions, for example, solutions where a single location has two separate facilities assigned to it. In standard QUBO local-search heuristics, solutions are updated through single bit-flips, and since there are $n^2$ binary variables, there are exactly $n^2$ configurations accessible from any given solution via single bit-flips. Given a feasible starting configuration, none of the $n^2$ single bit-flip moves lead to another feasible configuration; in fact, they result in a solution penalized by the term $\lambda \left[\sum_{k=1}^{n}\left(\sum_{i=1}^n x_{ik} - 1\right)^2 +\sum_{i=1}^{n}\left(\sum_{k=1}^n x_{ik} - 1\right)^2 \right]$  in the QUBO objective function \eqref{qubo_formulation}. It can easily be shown that to move via single bit-flips between feasible solutions in the binary space, it is necessary to visit at least three infeasible solutions before reaching another feasible solution. A feasible binary solution can be represented as an $n \times n$ permutation matrix by rearranging the $n^2$ solution vector, and a local move to another feasible solution corresponds to swapping two columns of the permutation matrix. The minimal way to do this is via four individual bit flips.

In the native space of QAP problems of size $n=3$, the basic move involves simply swapping two elements of a solution as follows:
\begin{equation*}
    \begin{pmatrix}
    {\color{red} 3} \\ 
    {\color{orange} 1} \\ 
    {\color{blue} 2} 
  \end{pmatrix}
    \quad \rightarrow \quad
    \begin{pmatrix}
    {\color{orange} 1} \\  
    {\color{red} 3} \\ 
    {\color{blue} 2} 
  \end{pmatrix}
\end{equation*}
In single bit-flip QUBO heuristics, the equivalent move between two feasible QAP solutions must traverse at least three infeasible solutions (infeasible solutions are shown in grey and black):
\begin{equation*}
    \begin{pmatrix}
    {\color{red}0} \\ {\color{red}0} \\ {\color{red}1} \\ {\color{orange}1} \\ {\color{orange}0} \\ {\color{orange}0} \\ {\color{blue}0} \\ {\color{blue}1} \\ {\color{blue}0} 
  \end{pmatrix}
 \quad \rightarrow \quad
\begin{pmatrix}
    {\color{black}1} \\ {\color{lightgray}0} \\ {\color{lightgray}1} \\ {\color{lightgray}1} \\ {\color{lightgray}0} \\ {\color{lightgray}0} \\ {\color{lightgray}0} \\ {\color{lightgray}1} \\ {\color{lightgray}0} 
  \end{pmatrix}
 \quad \rightarrow \quad
\begin{pmatrix}
    {\color{black}1} \\ {\color{lightgray}0} \\ {\color{lightgray}1} \\ {\color{lightgray}1} \\ {\color{lightgray}0} \\ {\color{black}1} \\ {\color{lightgray}0} \\ {\color{lightgray}1} \\ {\color{lightgray}0} 
  \end{pmatrix}
 \quad \rightarrow \quad
\begin{pmatrix}
    {\color{black}1} \\ {\color{lightgray}0} \\ {\color{lightgray}1} \\ {\color{black}0} \\ {\color{lightgray}0} \\ {\color{black}1} \\ {\color{lightgray}0} \\ {\color{lightgray}1} \\ {\color{lightgray}0} 
  \end{pmatrix}
 \quad \rightarrow \quad
\begin{pmatrix} {\color{orange}1} \\ {\color{orange}0} \\ {\color{orange}0} \\ 
    {\color{red}0} \\ {\color{red}0} \\ {\color{red}1} \\{\color{blue}0} \\ {\color{blue}1} \\ {\color{blue}0} 
  \end{pmatrix}
\end{equation*}
In practice, however, it is often the case that  significantly more than three infeasible configurations are required before a new feasible solution is found, as now all feasible solutions exist as local optima, surrounded by high-energy (i.e., penalized by $\lambda$) infeasible solutions. \Cref{QAPResiduals} illustrates this phenomenon. In \Cref{fig:gap_QUBO_tabu}, the tabu-based QUBO heuristic is able to see only fewer than 30 feasible solutions, 
indicated by the red markers, over 1000 iterations. Annealing-based QUBO solvers often have even worse performance, particularly as the temperature decreases and they are forced to accept only moves that result in an improvement.
\begin{figure}
    \centering
    \includegraphics[width=0.75\textwidth]{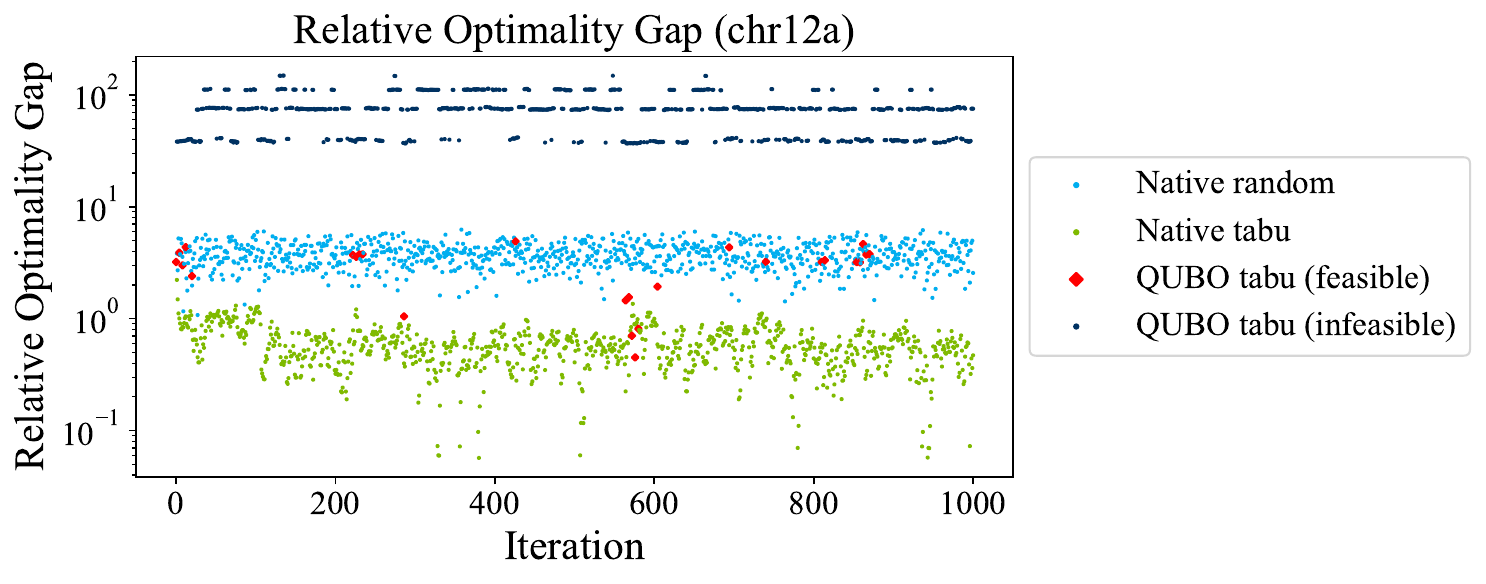}
    \caption{Relative optimality gaps between objective values observed by three heuristic solvers and optimal value for the instance ``chr12a'' of QAPLIB, each having 12 facilities and locations (144 binary variables in total). The search heuristics used are native-space tabu search, a random feasible-solution generator, and a QUBO-space tabu search.}
    \label{fig:gap_QUBO_tabu}
\end{figure}

\begin{figure}[ht]
    \centering
    \includegraphics[width=.7\textwidth]{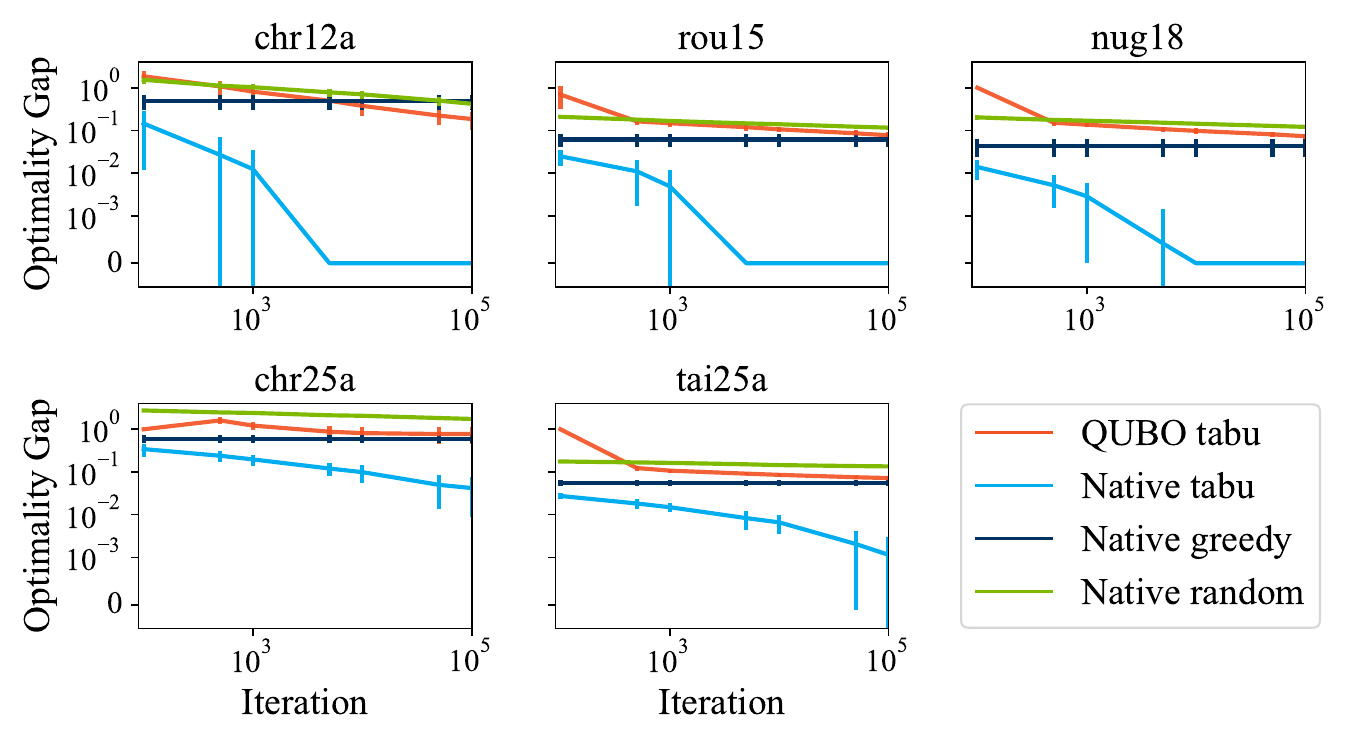}
    \caption{Relative optimality gap as a function of the number of iterations for the QAPLIB \cite{QAPLIBurl} instances between 12 and 25 facilities and locations (or between 144 and 625 binary variables). Four heuristics, each repeated $30$  times, are tested: a tabu implementation on the QUBO formulation, a tabu implementation on the native formulation, a greedy formulation on the native formulation, and a baseline of generated random feasible solutions.}
    \label{QAPResiduals}
\end{figure}

Solvers attempting to navigate this landscape tend to spend most of the time on infeasible solutions, and are thus able to find only a few feasible solutions. Any notion of performing gradient descent on the feasible space is lost as there is no direct connection between feasible solutions. \Cref{QAPResiduals} shows that, in comparing optimized solvers in the QUBO space to the generation of random feasible solutions (which is trivial, as feasible solutions are simply permutations), the QUBO heuristics do not improve on the random heuristic.
The native-space tabu search heuristic, on the other hand, drastically outperforms the tabu heuristic applied to the QUBO formulation and is often able to find the optimal solution.

\section{An Improved Local Search Heuristic for Hardware}
\label{sec:localUpdate}

The QAP problem QAP$(F,D)$ can be formulated as a series of matrix multiplications by representing solutions as permutation matrices. 
Recall that $P_n$ is the set of $n \times n$ permutation matrices. 
Then, the objective function representing an optimal solution $\pi^*$ in the problem \eqref{eq:KoopBeckFormulation} is equivalent to the formulation
\begin{equation}
\label{eq:traceFormulation}
    X^* \in   \Argmin_{X\in P_n} \left\{ \trace(  DXFX^T ) \right\} .
\end{equation}
It is well-known that, for any $n \times n$ matrices $D,F$, and $X$, the following relationship holds: 
\begin{equation}
    \label{eq:trace_to_qaud}
\begin{array}{rcl}
\trace (FXD^TX^T) &=&  \kvec(F^TXD)^T \kvec(X) \\ 
   & = & \left[ (D^T \otimes F^T)\kvec(X) \right]^T \kvec(X) \\
   &= & \kvec(X)^T (D\otimes F)  \kvec(X) .
\end{array}
\end{equation}
Thus, the formulation \eqref{eq:traceFormulation} is equivalent to 
\begin{equation}
    \label{eq:QUBOFeasFormulation}
        X^* \in  \Argmin_{X\in P_n} \left\{  \kvec(X)^T (D\otimes F)  \kvec(X)  \right\} ,
\end{equation}
and this expression represents the non-penalized part of the binary formulation in the QUBO objective function~\eqref{qubo_formulation}.

In this section, we propose a local update mechanism that adheres to the formulation~\eqref{eq:QUBOFeasFormulation}. 
We emphasize that working within the constraints 
\[
\left\{\kvec(X) \in \{0,1\}^{n^2} : X\in P_n \right\}
\]
removes the need to employ the penalty terms in the QUBO formulation \eqref{qubo_formulation}; from this point forward, we refer to this formulation as the ``binary formulation''.
For the balance of this paper, we assume that $D$ and $F$ each have the zero-diagonal property ($D_{i,i}=0, F_{i,i}=0,$ for each $i$), as is observed in the vast majority of QAP instances.

\subsection{Quadruple Bit-Flip}
\label{sec:QaudFlipImproved}

The main differences in solving a QAP problem in the native space compared to solving it in the QUBO space are that the QUBO solution space is much larger and that it is composed of mostly infeasible solutions. In addition, the basic move on a single bit-flip QUBO solver would force the solver to spend most of its time in infeasible space and impede any gradient decisions based only on feasible solutions. 
It is possible to avoid the infeasible solutions and regain the performance of a native QAP solver by defining each move as a quadruple bit-flip move within the constrained feasible binary space. 
This specialized move is equivalent to swapping two elements in a permutation (2-opt)~\cite{Taillard91}, and to swapping two columns of a permutation matrix (4-opt)~\cite{WANG2023109220}.
The key idea is that flipping a single zero variable in a feasible solution uniquely determines the other three variables that will need to be flipped in order to guarantee another feasible solution.

We illustrate how a quadruple bit-flip is determined with an initially feasible solution with $n=4$ native variables (i.e., 16 variables in the binary space). Suppose that the variable at index seven is flipped from zero to one:
\[     x= 
    \begin{pmatrix}
    1 \\ 0 \\ 0 \\ 0 \\
    0 \\  0 \\  {\color{red}1} \\  1 \\ 
    0 \\  0 \\  1 \\  0 \\
    0 \\  1 \\  0 \\  0 \\ 
    \end{pmatrix}
    \quad \rightarrow \quad
    \begin{bmatrix}
    1 & 0 & 0 & 0 \\
    0 & 0 & 0 & 1\\
    0 & {\color{red}1} & 1 & 0\\
    0 & 1 & 0 & 0\\
    \end{bmatrix} .
\]
Visualizing the binary solution as a matrix immediately exposes the violations of the one-hot encoding constraints (i.e., rows and columns must contain exactly one ``1''), and pinpoints the bit-flips that are required to arrive at a new feasible solution:
\begin{equation*}
     x= 
    \begin{pmatrix}
    1 \\ 0 \\ 0 \\ 0 \\
    0 \\  0 \\  {\color{red}1} \\  {\color{red}0} \\ 
    0 \\  0 \\  {\color{red}0} \\  {\color{red}1} \\
    0 \\  1 \\  0 \\  0 \\ 
    \end{pmatrix}
    \quad \rightarrow \quad
    \begin{bmatrix}
    1 & 0 & 0 & 0 \\
    0 & 0 & 0 & 1 \\
    0 & {\color{red}1} & {\color{red}0} & 0\\
    0 & {\color{red}0} & {\color{red}1} & 0\\
    \end{bmatrix} .
\end{equation*}
For an arbitrary $n$, there are $n^2-n = n(n-1)$ variables assigned a value of ``0'' from which to choose, and a permutation solution in native space has $\binom{n}{2} = n(n-1)/2$ possible local moves. The 2$\times$ disparity comes from the fact that there are two possible choices of ``0'' elements in the binary solution that result in the identical new feasible solution.
Transitioning from one feasible solution to another involves selecting a pair of variables assigned a value of ``1''.  This choice then uniquely determines the two other variables assigned a value of ``0'' that need to be flipped. 
We can systematically perform these operations directly within the feasible binary space. 
Given two indices $z_1, z_2 \in  \{0,1,\ldots, n^2-1\}$ to flip from ``1'' to ``0'', we identify indices $z_3, z_4\in \{0,1,\ldots, n^2-1\}$ to flip from ``0'' to ``1'' following the rule summarized in \Cref{algo:quadBitflip}.
\begin{algorithm}[H]
\caption{~~Quadruple Bit-Flips in Feasible Binary Space}
\label{algo:quadBitflip}
\begin{algorithmic}[1]
\Function{Identify\_Bit\_to\_Flip}{$z_1, z_2$}
    \State $z_3 = \left\lfloor {z_2}/{n} \right\rfloor \cdot n + \left( z_1 \bmod  n \right) $
    \State $z_4 =  \left\lfloor {z_1}/{n} \right\rfloor \cdot n + \left( z_2 \bmod  n \right)$
    \State \Return $z_3,z_4$
\EndFunction
\end{algorithmic}
\end{algorithm}
Let $x$ be a feasible binary solution and let $z_3,z_4=$ \texttt{{Identify\_Bit\_to\_Flip}}($z_1,z_2)$ for two distinct indices \mbox{$z_1,z_2 \in \supp(x)$.}
Then, a new feasible solution in the binary space is obtained by the following assignment:
\[
x(z_1) \leftarrow 0, \ 
x(z_2) \leftarrow 0, \
x(z_3) \leftarrow 1, \
x(z_4) \leftarrow 1.
\]
There are $n$ elements assigned a value of ``1'' in each feasible binary solution $x$ that correspond to $\supp(x)$. Therefore, the local neighbourhood of possible feasible moves comprises $\binom{n}{2}$ quadruple bit-flips.

\subsection{Feasible-Space Updates to the Binary Formulation}

In this section, we derive a formula to compute the difference in the objective values between two permutation assignments, expressed through a sequence of inner products. The formula facilitates natural translations to IMC hardware.
We use the quadruple bit-flips outlined in \Cref{sec:QaudFlipImproved} and vectorized operations. 
Starting from a feasible solution $x\in \{0,1\}^{n^2}$,  let $y\in \{0,1\}^{n^2}$ be a feasible solution derived from $x$ via a quadruple bit-flip. 
For simplicity, let $Q := D\otimes F$ in the formulation \eqref{eq:QUBOFeasFormulation}.
The goal is to show how to compute the gradient, that is, the difference in the objective value for a local update, 
\begin{equation}
\label{eq:orig_diff_energy}
y^TQ y - x^TQ x,
\end{equation}
using simple matrix--vector products. 

Let $z_1$ and $z_2$ be the indices where ``$1$'' in $x$ is flipped to ``$0$'', while $z_3$ and $z_4$ are the indices where ``$0$'' in $x$ is flipped to ``$1$'', following the procedure outlined in \Cref{algo:quadBitflip}.
Since $y$ is the result of the quadruple bit-flip originating from $x$, 
$\supp(x) \cap  \supp(y)^c = \{z_1,z_2\}$
and $  \supp(y) \cap \supp(x)^c = \{z_3,z_4\}$.
We define the binary vectors $\phi_x ,\phi_y\in \{0,1\}^{n^2}$, where
\begin{equation}
    \label{eq:diffphi_xy}
(\phi_x)_i = 
\begin{cases}
1  & \text{if }  i \in \{z_1,z_2\}, \\
0 & \text{otherwise,}
\end{cases}
\quad \text{ and } \quad
(\phi_y)_i = 
\begin{cases}
1  & \text{if }  i \in  \{z_3,z_4\}, \\
0 & \text{otherwise.}
\end{cases}
\end{equation}
In other words, 
$\phi_x =e_{z_1}+e_{z_2}$
and 
$\phi_y=e_{z_3}+e_{z_4}$, where $e_i$ denotes the $i$-th standard unit vector.
The support of the vectors $\phi_x$ and $\phi_y$ indicates the positions in the incumbent variable assignment where bit flips have occurred.  Thus, the two solutions $x$ and $y$ have the following relationship:
\[
y = x-e_{z_1}-e_{z_2}+e_{z_3}+e_{z_4} = x- \phi_x + \phi_y.
\]
Then, the difference \eqref{eq:orig_diff_energy} can be computed explicitly as 
\begin{equation}
\label{eq:naive_diff_computation}
\begin{array}{ccl}
y^TQy - x^TQx 
&=& \left[ Q(z_3, : )  + Q(z_4, : )  \right] y
- \left[  Q(z_1, : ) + Q(z_2 , : )  \right] x \\
&& +\left[ Q(:,z_3 )  + Q(:,z_4 ) \right]^T y
- \left[  Q(:,z_1 ) + Q(:,z_2 )  \right]^Tx  \\
&& +~Q(z_1,:)\phi_x + Q(z_2,:) \phi_x - Q(z_3,:)\phi_y -Q(z_4,:)\phi_y .
\end{array}
\end{equation}
The derivation of the equality \eqref{eq:naive_diff_computation} is included in Proposition \ref{prop:nativeDiff}.

The remainder of this section focuses on reducing the number of algebraic operations outlined in the equality~\eqref{eq:naive_diff_computation}.
We employ symmetrization as a strategy to achieve efficient local updates 
and  present an improved local update 
that leverages the structure of semi-symmetric instances.

\subsubsection{Feasible-Space Updates via Symmetrization}
\label{sec:symmetrization}

For any $w\in \mathbb{R}^n$ and $A\in \mathbb{R}^{n\times n}$,
recall the identity 
\[ 
w^TAw   =  w^T \left( \frac{1}{2} A + \frac{1}{2}A^T\right)w.
\]
Hence, we may symmetrize $Q$ and assume the symmetry of $Q$ without loss of generality.
Since we have $Q(:,i) = Q(:,i)^T$ holds for all $i$,
this symmetrization step significantly reduces the number of inner products in the  na\"ive computation \eqref{eq:naive_diff_computation}:
\begin{equation}
\label{eq:symm_diff_computation_0}
\begin{array}{ccl}
y^TQy - x^TQx  &=&  2 \cdot \left[ Q(z_3, : )  + Q(z_4, : ) \right] y - 2 \cdot \left[  Q(z_1, : ) + Q(z_2 , : ) \right] x  \\
&& +~Q(z_1,:)\phi_x + Q(z_2,:) \phi_x - Q(z_3,:)\phi_y -Q(z_4,:)\phi_y .
\end{array}
\end{equation}
Furthermore, since $Q$ has the zero-diagonal property, note that 
\[
Q(z_1,:) \phi_x 
= Q(z_1,:) e_{z_1} + Q(z_1,:)e_{z_2}
= Q(z_1,z_1) + Q(z_1,z_2)
= Q(z_1,z_2).
\]
Similarly, we have  
\[
Q(z_2,:) \phi_x = Q(z_2,z_1), \ 
Q(z_3,:) \phi_y = Q(z_3,z_4), \ \text{ and } \
Q(z_4,:) \phi_y = Q(z_4,z_3). 
\]
Therefore, the last four terms of the right-hand side of the equality \eqref{eq:symm_diff_computation_0} can be simplified owing to the symmetry of $Q$, resulting in a more compact gradient evaluation:
\begin{equation}
\label{eq:symm_diff_computation_1}
\begin{array}{ccl}
&& y^TQy - x^TQx
= 2 \cdot \left[ Q(z_3, : )  + Q(z_4, : ) \right] y + 2 \cdot \left[ - Q(z_1, : ) - Q(z_2 , : ) \right] x  + 2 \cdot \left[ Q(z_1,:)\phi_x  - Q(z_4,:)\phi_y \right].
\end{array}
\end{equation}

\subsubsection{Feasible-Space Updates for Semi-symmetric Instances}
\label{sec:Update_SemiSym}

A QAP instance is said to be ``symmetric'' if both the flow matrix $F$ and the distance matrix $D$ are symmetric, 
``semi-symmetric'' if either $D$ or $F$ is symmetric,
and ``asymmetric'' if neither of the matrices is symmetric.
The QAPLIB \cite{QAPLIBurl} is a popular benchmarking library widely used by researchers. Out of 136 instances in the QAPLIB, only eight QAP instances formulated by Burkard and Offermann (the ``bur'' class) are asymmetric.
We further simplify the formula \eqref{eq:symm_diff_computation_1} for symmetric or semi-symmetric instances in Proposition~\ref{prop:SemisymFormula} below.

\begin{prop}
\label{prop:SemisymFormula}
Let $D \in \mathbb{R}^{n\times n}$, $F\in \mathbb{S}^n$ and $Q = \frac{1}{2}(D\otimes F) + \frac{1}{2}(D\otimes F)^T$. 
Let $x=\kvec(X)$, where $X \in P_n$, and $y=x-e_{z_1}-e_{z_2}+e_{z_3}+e_{z_4}$, with $(z_1,z_2,z_3,z_4)$ being the tuple that realizes the quadruple bit-flip.
Then,
\begin{equation}
\label{eq:diff_semi_symm_update}
\begin{array}{c}
y^TQy - x^TQx = 2 \cdot \left[ Q(z_3, : )y  + Q(z_4, : )y - Q(z_1, : )x - Q(z_2 , : )x   \right]  .
\end{array}
\end{equation}
\end{prop}

\begin{proof}
Let     $\phi_x = e_{z_1}+e_{z_2}$ and  $\phi_y = e_{z_3}+e_{z_4}$.
We show that 
\begin{equation}
\label{eq:semi_identical_char}
Q(z_1,:)\phi_x  = Q(z_4,:)\phi_y 
\end{equation}
holds, where 
\[
Q = \frac{1}{2}(D\otimes F) + \frac{1}{2}(D\otimes F)^T 
= \left( \frac{1}{2}D + \frac{1}{2}D^T \right) \otimes F. 
\]
Let $\hat{D}= \frac{1}{2}D + \frac{1}{2}D^T$.
Since  $F$ is symmetric, each $(i,j)$-th $n$-by-$n$ block submatrix within $\hat{D}\otimes F$ is symmetric:
\[
\hat{D}_{i,j} F_{k,\ell} 
= \hat{D}_{i,j} F_{\ell, k} , \  \forall i,j,k,\ell \in \{0,\ldots, n-1\}.
\] 
Hence, we have 
\begin{equation}
    \label{eq:semi_sym_Qubo}
Q(i n  + k , j n+\ell ) = Q(i n  + \ell , j n+ k ),  \  i,j,k,\ell \in \{0,\ldots, n-1\}.
\end{equation}
Let $i=\lfloor z_1/n \rfloor$ and $j=\lfloor z_2/n \rfloor$.
Recall \Cref{algo:quadBitflip}, and note that
\begin{equation}
    \label{eq:z1-z4AlgReltation}
\begin{array}{lcl}
z_1 = in+ (z_1\bmod n), & & 
z_2 = jn+ (z_2\bmod n),  \\
z_3 = jn+ (z_1\bmod n), & & 
z_4 = in+ (z_2\bmod n)  .\\
\end{array}
\end{equation}
Thus, $Q(z_1,z_2)=Q(z_4,z_3)$ holds by the identity \eqref{eq:semi_sym_Qubo}.
Now, recall that 
\[
Q(z_1,:)\phi_x = Q(z_1,z_2) \ \text{ and } \
Q(z_4,:)\phi_y = Q(z_4,z_3) .
\]
Thus, the equality~\eqref{eq:semi_identical_char} holds and the equality~\eqref{eq:diff_semi_symm_update} follows.
\end{proof}
We note that the simplification from the equality \eqref{eq:symm_diff_computation_1} to the equality \eqref{eq:diff_semi_symm_update} is a direct consequence of the matrix $Q$ being constructed from at least one symmetric matrix. This structure, stemming from the semi-symmetry of the problem data, results in the characterization \eqref{eq:semi_identical_char}.
The reduced gradient evaluation in the equality~\eqref{eq:diff_semi_symm_update}
 is equally applicable to semi-symmetric instances with a symmetric matrix $D$, rather than $F$ being symmetric.
The cyclicity of the trace operator provides
\[
\trace (FXD^TX^T) = \trace (D^TX^TFX) . 
\]
Then, by a change of variables, $Z \leftarrow X^T$, we obtain 
\[
\trace (FXD^TX^T) = \trace (D^TZFZ^T).
\]
Thus, by the equality \eqref{eq:trace_to_qaud}, we have
\[
\kvec(X)^T (D\otimes F)\kvec(X) =   \kvec(Z)^T ( F^T  \otimes D^T) \kvec(Z)
=   \kvec(Z)^T ( F  \otimes D) \kvec(Z).
\]
\Cref{algo:quboMatFormulation} shows how to formulate the data matrix $Q$ used in the binary space representation of the QAP (i.e., the QUBO matrix formulation) used throughout this paper. 
\begin{algorithm}[H]
\caption{~~QUBO Matrix Formulation}
\label{algo:quboMatFormulation}
\begin{algorithmic}[1]
\Function{\texttt{Formulate\_QUBO\_Matrix}}{$F$ $\leftarrow$ flow matrix, $D$ $\leftarrow$ distance matrix}
        \If{$F,D$ are symmetric} \Comment{symmetric instance}
           \State $Q=F\otimes D$
        \ElsIf{exactly one of $F,D$ is symmetric} \Comment{semi-symmetric instance}
            \If{$F$ is symmetric}  {$Q=D\otimes F$}
            \ElsIf{$D$ is symmetric} {$Q=F\otimes D$} \EndIf
        \State  $Q = \frac{1}{2}Q + \frac{1}{2}Q^T$
    \Else{ $\ F,D$ are not symmetric} \Comment{asymmetric instance}
        \State $Q=F\otimes D$
        \State  $Q = \frac{1}{2}Q + \frac{1}{2}Q^T$
    \EndIf
    \State \Return $Q$
\EndFunction
\end{algorithmic}
\end{algorithm}

\section{Full-Neighbourhood Search}
\label{sec:FNSearch}

In the native QAP space, given a solution represented as a permutation matrix $X\in P_n$, we call $Y$ a ``neighbour'' of $X$ if $Y$ is obtained by swapping two distinct columns of $X$. 
A ``full neighbourhood'' of $X$ is the set of all such neighbours of $X$. Hence, the full neighbourhood of $X$ consists of $\binom{n}{2}=\frac{n(n-1)}{2}$ members. 
In this section, we present an analogous interpretation of the full neighbourhood in the binary space, followed by a vectorized evaluation  mechanism for computing the gradients associated with the full neighbourhood in the binary space. 
Our approach focuses on parallelizing the evaluation of the objective function within a single iteration of the algorithm, rather than utilizing multiple parallel threads to evaluate each individual objective function. 

\subsection{Definition of ``Full Neighbourhood''}

Let $x=\kvec(X)$, where $X \in P_n$ is a permutation matrix. 
The full neighbourhood of $x$, denoted by $\cN(x)$, 
is the set of $4$-tuples 
\begin{equation}
    \label{eq:def_of_FN_QUBO}
\cN(x) := \left\{\left(z_1,z_2,
\left\lfloor \frac{z_2}{n} \right\rfloor \cdot n + \bmod(z_1, n)
,
\left\lfloor \frac{z_1}{n}\right\rfloor \cdot n + \bmod(z_2, n) 
\right) :
z_1,z_2 \in \supp(x), z_1<z_2
\right\}.
\end{equation}
Note that the last two elements in each tuple originate from the quadruple bit-flip defined in \Cref{algo:quadBitflip}.
We use the term ``neighbour'' interchangeably to refer both to 
an element \mbox{$(z_1,z_2,z_3,z_4) \in \cN(x)$} and to the vector
$y = x-e_{z_1}-e_{z_2}+e_{z_3}+e_{z_4}$ when the meaning is clear.
Since $\cN(x)$ has an equivalent interpretation in the native space $P_n$, there are exactly $\binom{n}{2}$ members in $\cN(x)$.
For a given neighbour~\mbox{$(z_1,z_2,z_3,z_4) \in \cN(x)$}, it is important to note that 
\begin{equation}
\label{eq:neighrPairRule}    
\left\lfloor \frac{z_1}{n} \right\rfloor  = \left\lfloor \frac{z_4}{n} \right\rfloor  
\text{ and }
\left\lfloor \frac{z_2}{n} \right\rfloor  = \left\lfloor \frac{z_3}{n} \right\rfloor  . 
\end{equation}
Hence, in the context of permutation matrices, 
each neighbour in  $\cN(x)$ corresponds to 
swapping the \mbox{$\left\lfloor \frac{z_1}{n}\right\rfloor$-th} column with the $\left\lfloor \frac{z_2}{n} \right\rfloor$-th column of $X$.
Lastly,
according to the definition of $\cN(x)$ in the set~\eqref{eq:def_of_FN_QUBO},
the first two elements in each $4$-tuple 
 represent the indices of $x$ assigned a value of ``1''.
Since $\cN(x)$ collects all such pairs $z_1,z_2\in \supp(x)$, taking their union collectively reconstructs $\supp(x)$.
Recall that the last two elements of each tuple designate the indices of $x$ that need to be flipped from ``0'' to ``1''. Note that these elements are uniquely determined by the first two elements.  
This observation leads to the following 
algebraic set relationships:
\begin{equation}
\label{eq:suppx_suppxCompl}
\begin{array}{lcl}
\text{supp}(x) &=&\cup_{(z_1,z_2,z_3,z_4)\in \cN(x)} \{z_1,z_2\}  , \\
\text{supp}(x)^c  &=& \cup_{(z_1,z_2,z_3,z_4)\in \cN(x)} \{z_3,z_4\}  .
\end{array}
\end{equation}
Example~\ref{example:FNillustration} 
illustrates the conformations of the full neighbourhood in three different solution spaces. 

\begin{example}(Equivalence of three configurations)
\label{example:FNillustration}
Consider the permutation \mbox{$\pi = [2, 0, 1, 3, 4] \in   \Pi_5$} and its counterpart 
$X  = $
\begin{scriptsize}
$\begin{bmatrix}
0 & 1 & 0 & 0 & 0 \\
 0 & 0 & 1 & 0 & 0 \\
 1 & 0 & 0 & 0 & 0 \\
 0 & 0 & 0 & 1 & 0 \\
 0 & 0 & 0 & 0 & 1 \\
 \end{bmatrix}$\end{scriptsize}
realized in $P_5$.
Then, swapping two elements of $\pi$ is equivalent to swapping two columns of $X$, resulting in the following $10$ permutation matrices accounting for the full neighbourhood of $X$:
\begin{scriptsize}
\[
\begin{bmatrix}
1 & 0 & 0 & 0 & 0 \\
 0 & 0 & 1 & 0 & 0 \\
 0 & 1 & 0 & 0 & 0 \\
 0 & 0 & 0 & 1 & 0 \\
 0 & 0 & 0 & 0 & 1 \\
 \end{bmatrix},
 \begin{bmatrix}
      0 & 1 & 0 & 0 & 0 \\
 1 & 0 & 0 & 0 & 0 \\
 0 & 0 & 1 & 0 & 0 \\
 0 & 0 & 0 & 1 & 0 \\
 0 & 0 & 0 & 0 & 1 \\
 \end{bmatrix},
 \begin{bmatrix}
      0 & 1 & 0 & 0 & 0 \\
 0 & 0 & 1 & 0 & 0 \\
 0 & 0 & 0 & 1 & 0 \\
 1 & 0 & 0 & 0 & 0 \\
 0 & 0 & 0 & 0 & 1 \\
 \end{bmatrix},
 \begin{bmatrix}
      0 & 1 & 0 & 0 & 0 \\
 0 & 0 & 1 & 0 & 0 \\
 0 & 0 & 0 & 0 & 1 \\
 0 & 0 & 0 & 1 & 0 \\
 1 & 0 & 0 & 0 & 0 \\
 \end{bmatrix}, 
 \begin{bmatrix}
      0 & 0 & 1 & 0 & 0 \\
 0 & 1 & 0 & 0 & 0 \\
 1 & 0 & 0 & 0 & 0 \\
 0 & 0 & 0 & 1 & 0 \\
 0 & 0 & 0 & 0 & 1 \\
 \end{bmatrix},
\]
\[
\begin{bmatrix}
     0 & 0 & 0 & 1 & 0 \\
 0 & 0 & 1 & 0 & 0 \\
 1 & 0 & 0 & 0 & 0 \\
 0 & 1 & 0 & 0 & 0 \\
 0 & 0 & 0 & 0 & 1 \\
\end{bmatrix},
\begin{bmatrix}
   0 & 0 & 0 & 0 & 1 \\
 0 & 0 & 1 & 0 & 0 \\
 1 & 0 & 0 & 0 & 0 \\
 0 & 0 & 0 & 1 & 0 \\
 0 & 1 & 0 & 0 & 0 \\ 
\end{bmatrix},
\begin{bmatrix}
     0 & 1 & 0 & 0 & 0 \\
 0 & 0 & 0 & 1 & 0 \\
 1 & 0 & 0 & 0 & 0 \\
 0 & 0 & 1 & 0 & 0 \\
 0 & 0 & 0 & 0 & 1 \\
\end{bmatrix},
\begin{bmatrix}
    0 & 1 & 0 & 0 & 0 \\
 0 & 0 & 0 & 0 & 1 \\
 1 & 0 & 0 & 0 & 0 \\
 0 & 0 & 0 & 1 & 0 \\
 0 & 0 & 1 & 0 & 0 \\
\end{bmatrix},
\begin{bmatrix}
     0 & 1 & 0 & 0 & 0 \\
 0 & 0 & 1 & 0 & 0 \\
 1 & 0 & 0 & 0 & 0 \\
 0 & 0 & 0 & 0 & 1 \\
 0 & 0 & 0 & 1 & 0 \\
\end{bmatrix}.
\]
\end{scriptsize}
\hspace{-0.9em}
Swapping two elements of $\pi$ has an analogous meaning in the feasible binary space.
Let $x=\kvec(X)$. 
Swapping two elements of $\pi$ at indices $0$ and $1$ yields the permutation $[0,2,1,3,4]$. Realizing this swap in the binary space is equivalent to the binary solution $x-e_{2}-e_{5}+e_{7}+e_{0}$.
The remaining neighbours in $\cN(x)$ can be identified similarly using Table \ref{table:FN(x)}. Note that 
$\cup_{(z_1,z_2,z_3,z_4)\in \cN(x)} \{z_1,z_2\} $ collects $\mathrm{supp}(x) = \{2,5,11,18,24\}$ and   $\cup_{(z_1,z_2,z_3,z_4)\in \cN(x)} \{z_3,z_4\}$ collects the complement $\mathrm{supp}(x)^c = \{0,1,\ldots,24\} \setminus \{2,5,11,18,24\}$.
\begin{table}[ht]
\centering
\begin{tabular}{c|c|rrrr}
\hline
   Swap Indices of $\pi$ & Label for Members in $\cN(x)$&   $z_1$ &   $z_2$ &   $z_3$ &   $z_4$ \\
\hline
  0 $\leftrightarrow$ 1 & $N_0$ &    2 &     5 &     7 &     0 \\
  0 $\leftrightarrow$ 2 & $N_1$ &   2 &    11 &    12 &     1 \\
  0 $\leftrightarrow$ 3 & $N_2$ &   2 &    18 &    17 &     3 \\
  0 $\leftrightarrow$ 4 & $N_3$ &    2 &    24 &    22 &     4 \\
  1 $\leftrightarrow$ 2 & $N_4$ &    5 &    11 &    10 &     6 \\
  1 $\leftrightarrow$ 3 & $N_5$ &   5 &    18 &    15 &     8 \\
  1 $\leftrightarrow$ 4 & $N_6$ &    5 &    24 &    20 &     9 \\
  2 $\leftrightarrow$ 3 & $N_7$ &   11 &    18 &    16 &    13 \\
  2 $\leftrightarrow$ 4 & $N_8$ &   11 &    24 &    21 &    14 \\
  3 $\leftrightarrow$ 4 & $N_9$ &   18 &    24 &    23 &    19 \\
\hline
\end{tabular}
\caption{Full neighbourhood at $x=(0, 0, 1, 0, 0, 1, 0, 0, 0, 0, 0, 1, 0, 0, 0, 0, 0, 0, 1, 0, 0, 0, 0, 0, 1)^T$.}
\label{table:FN(x)}
\end{table}
\end{example}

\subsection{Partial Update of the Full Neighbourhood}
\label{sec:PartialFNUp}

Let $\cN(x)$ be the full neighbourhood of a feasible solution $x$.
We fix a neighbour $(\bar{z}_1,\bar{z}_2,z_2^*,z_1^*) \in \cN(x)$.
Note that by the definition~\eqref{eq:def_of_FN_QUBO}, the elements in this neighbour are related by the equalities
\begin{equation}
    \label{eq:z12star_z12bar}
z_1^* = 
\left\lfloor \frac{\bar{z}_1}{n}\right\rfloor \cdot n + \bmod(\bar{z}_2, n) \ \text{ and } \
z_2^* = \left\lfloor \frac{\bar{z}_2}{n} \right\rfloor \cdot n + \bmod(\bar{z}_1, n).
\end{equation}
Now, consider the neighbour that results from the quadruple bit-flip via $(\bar{z}_1,\bar{z}_2,z_2^*,z_1^*)$,
\[
y=x-e_{\bar{z}_1}-e_{\bar{z}_2}+e_{z_1^*}+e_{z_2^*},
\]
and its full neighbourhood $\cN(y)$.
In this section, we derive $\cN(y)$ through a partial update to $\cN(x)$.

Recall, from the definition~\eqref{eq:def_of_FN_QUBO} and the equalities~\eqref{eq:suppx_suppxCompl}, that the first two elements of all members of the full neighbourhoods $\cN(x)$ and $\cN(y)$ are characterized by all pairs of elements in $\supp(x)$ and $\supp(y)$, respectively, while the last two elements are induced by the first two. Since $x$ and $y$ are related by a quadruple bit-flip, 
we have $|\supp(x)\cap \supp(y)| = n-2$. 
This implies that, among the members in $\cN(x)$ and $\cN(y)$, 
those originating from $\supp(x)\cap \supp(y)$ are identical; see \Cref{fig:suppx_suppt_intersection}(a).
Hence, among $\binom{n}{2}$ members in each of $\cN(x)$ and $\cN(y)$, exactly $\binom{n-2}{2}$ members are shared, that is, 
\[
|\cN(x) \cap \cN(y)| = \binom{n-2}{2} .
\]
Consequently, 
$\binom{n}{2}-\binom{n-2}{2}=2n-3$ members within $\cN(x)$ require updates to construct $\cN(y)$. 
\begin{figure}[ht]
    \centering
    \includegraphics[width=0.7\linewidth]{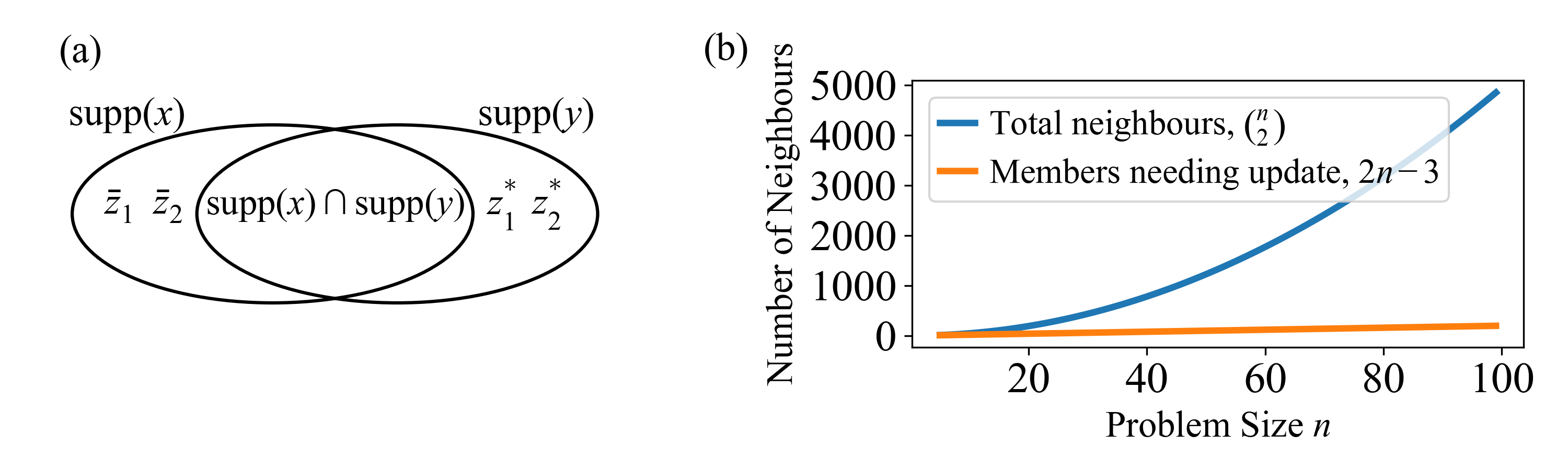}
    \caption{(a) Relationship between two feasible binary solutions $x, y \in \{0,1\}^{n^2}$. Consequently, there are $\binom{n-2}{2}$ elements in $\cN(x)\cap \cN(y)$. (b) Disparity between the total number of members in the full neighbourhood and those requiring updates.}
    \label{fig:suppx_suppt_intersection}
\end{figure}

Now, we define the set
\[  
\cU_{\cN(x)}
:=
\left\{ (z_1,z_2,z_3,z_4)\in \mathcal{N}(x) : z_1 \in \{\bar{z}_1,\bar{z}_2 \} \text{ or }  z_2 \in \{ \bar{z}_1,\bar{z}_2 \} 
\right\}.
\]
Note that the $2n-3$ members within $\cN(x)$ requiring updates are characterized by $\cU_{\cN(x)}$.
Notably, the cardinality of $\cU_{\cN(x)}$ is linear in $n$, while the cardinality of $\cN(x)$ is quadratic in $n$; see \Cref{fig:suppx_suppt_intersection}(b).

\subsubsection{Updates Realized Using the Tuples in $\cN(x)$}

Let $\dot{\cup}$ denote the disjoint union. 
We partition $\cU_{\cN(x)}$ based on the membership of $\bar{z}_1$ and $\bar{z}_2$ within each tuple: 
\[
\cU_{\cN(x)} = 
\cU_{\cN(x)}^{\bar{z}_1}
\dot{\cup}  \ 
\cU_{\cN(x)}^{\bar{z}_2}  
\dot{\cup}  \ 
\cU_{\cN(x)}^{\bar{z}_1,\bar{z}_2} ,
\]
where
\[
\begin{array}{rcl}
\cU_{\cN(x)}^{\bar{z}_1}
&=&
\left\{ (z_1,z_2,z_3,z_4)\in \cN(x) : 
\bar{z}_1 \in \{z_1,z_2 \} , 
\bar{z}_2 \not\in \{z_1,z_2 \} 
\right\} , \\ 
\cU_{\cN(x)}^{\bar{z}_2}
&=&
\left\{ 
(z_1,z_2,z_3,z_4)\in \mathcal{N}(x) : \bar{z}_2 \in \{z_1,z_2 \} , \bar{z}_1 \not\in \{z_1,z_2 \} 
\right\} ,  \text{ and }\\
\mathcal{U}_{\mathcal{N}(x)}^{\bar{z}_1,\bar{z}_2}
&=&
\left\{ (z_1,z_2,z_3,z_4)\in \mathcal{N}(x) :  
\{\bar{z}_1,\bar{z}_2 \} = \{z_1,z_2\}    
\right\} .
\end{array}
\]
We note that 
\[
\left|  \cU_{\mathcal{N}(x)} \right| = 
\left|  \cU_{\cN(x)}^{\bar{z}_1} \right|+
\left|  \cU_{\cN(x)}^{\bar{z}_2}\right|+
\left|  \cU_{\cN(x)}^{\bar{z}_1,\bar{z}_2} \right|=(n-2)+(n-2)+1
=2n-3.
\]
Let $(z_1,z_2,z_3,z_4)\in \mathcal{U}_{\mathcal{N}(x)}^{\bar{z}_1} \subseteq \mathcal{N}(x)$ and assume that $z_1=\bar{z}_1$. 
We observe the changes made to $(z_1,z_2,z_3,z_4)$  when $z_1$ is updated to $z_1^*$, while $z_2$ remains unchanged.  
From the definition~\eqref{eq:def_of_FN_QUBO}, the third element, $z_3$, in this tuple is transformed according to
\begin{equation}
    \label{eq:UFN_z1_thrid}
\left\lfloor \frac{z_2}{n} \right\rfloor \cdot n + \bmod(z_1, n) 
\rightarrow
\left\lfloor \frac{z_2}{n} \right\rfloor \cdot n + \bmod(z_1^*, n) .
\end{equation}
The fourth element remains unchanged due to the relationship~\eqref{eq:neighrPairRule}; hence, we have the equality
\begin{equation}
    \label{eq:UFN_z1_fourth}
\left\lfloor \frac{z_1}{n} \right\rfloor \cdot n + \bmod(z_2, n)
=
\left\lfloor \frac{z_1^*}{n} \right\rfloor \cdot n + \bmod(z_2, n)  .
\end{equation}
Thus, it follows from the assignment \eqref{eq:UFN_z1_thrid} and the equality
\eqref{eq:UFN_z1_fourth} that the neighbour $(z_1,z_2,z_3,z_4)$ in $\mathcal{U}_{\mathcal{N}(x)}^{\bar{z}_1}$ is updated as follows:
\begin{equation}
    \label{eq:z1_nbr_formula}
(z_1,z_2,z_3,z_4)
\to \left( z_1^*,z_2, \left\lfloor \frac{z_2}{n} \right\rfloor \cdot n + \bmod(z_1^*,n) , z_4 \right) .
\end{equation}
Similarly, for $(z_1,z_2,z_3,z_4)\in \cU_{\cN(x)}^{\bar{z}_2}$ with $z_2 = \bar{z}_2$,
the neighbour is updated as follows:
\begin{equation}
    \label{eq:z2_nbr_formula}
(z_1,z_2,z_3,z_4)
\to \left( 
z_1,z_2^*, z_3, \left\lfloor \frac{z_1}{n} \right\rfloor \cdot n  + \bmod(z_2^*,n)
\right) .
\end{equation}
Lastly, for $(z_1,z_2,z_3,z_4)\in \cU_{\cN(x)}^{\bar{z}_1,\bar{z}_2}$,
the neighbour is updated as follows:
\[
(z_1,z_2,z_3,z_4)
\to \left(
z_1^*, z_2^*,
\left\lfloor \frac{z_2}{n} \right\rfloor \cdot n +\bmod(z_1^*,n) , 
\left\lfloor \frac{z_1}{n} \right\rfloor \cdot n +\bmod(z_2^*,n)
\right) =
(z_1^*,z_2^*,\bar{z}_2,\bar{z}_1)
.
\]

\begin{example}
\label{example:FNiupdatellustration}
Recall Example \ref{example:FNillustration} and $\cN(x)$ in Table \ref{table:FN(x)}.
Let ``$N_7$'' from the full neighbourhood $\cN(x)$  corresponding to $(z_1,z_2,z_3,z_4)=(11,18,16,13)$ be chosen for the quadruple bit-flip. 
Then, the corresponding new solution is $y=x-e_{11}-e_{18}+e_{16}+e_{13}$.
Then, the full neighbourhood at $y$, $\cN(y)$, is shown in Table \ref{table:FN(y)}. Note that 
$N_1,N_2,N_4,N_5,N_7,N_8,N_9$ from $\cN(x)$ are updated to produce $\cN(y)$ in Table \ref{table:FN(y)}, while $N_0,N_3,N_6$ from $\cN(x)$ and $\cN(y)$ are identical.
The changed elements are highlighted in bold in Table~\ref{table:FN(y)}.
Of note, exactly two elements in  $N_1,N_2,N_4,N_5,N_8,N_9$ are modified, whereas all elements of $N_7$ are affected. 
\begin{table}[ht]
\centering

\begin{tabular}{|c|rrrr||c|rrrr|}
\hline
\multicolumn{5}{|c||}{$\mathcal{N}(x)$} & \multicolumn{5}{c|}{$\mathcal{N}(y)$} \\
\hline
label &   $z_1$ &   $z_2$ &   $z_3$ &   $z_4$ & label &   $z_1$ &   $z_2$ &   $z_3$ &   $z_4$ \\
\hline
$N_0$ &    2 &     5 &     7 &     0 &$N_0$ &     2 &     5 &     7 &     0 \\
$N_1$ &   2 &    11 &    12 &     1  &$N_1$ &     2 &    \textbf{13} &    12 &     \textbf{3} \\
$N_2$ &   2 &    18 &    17 &     3 &$N_2$ &     2 &    \textbf{16} &    17 &     \textbf{1} \\
$N_3$ &    2 &    24 &    22 &     4 &$N_3$&     2 &    24 &    22 &     4\\
$N_4$ &    5 &    11 &    10 &     6 &$N_4$&     5 &    \textbf{13} &    10 &     \textbf{8}\\
$N_5$ &   5 &    18 &    15 &     8 & $N_5$&     5 &   \textbf{16} &    15 &     \textbf{6}\\
$N_6$ &    5 &    24 &    20 &     9&  $N_6$&     5 &    24 &    20 &     9  \\
$N_7$ &   11 &    18 &    16 &    13& $N_7$ &    \textbf{13} &    \textbf{16} &    \textbf{18} &    \textbf{11}\\
$N_8$ &   11 &    24 &    21 &    14 & $N_8$ &    \textbf{13} &    24 &    \textbf{23} &    14\\
$N_9$ &   18 &    24 &    23 &    19 &$N_9$ &    \textbf{16} &    24 &    \textbf{21} &    19 \\
\hline
\end{tabular}
    \caption{Full neighbourhood $\cN(x)$ at the solution $x$  from Example \ref{example:FNillustration} and the full neighbourhood $\cN(y)$ at \mbox{$y=(0, 0, 1, 0, 0, 1, 0, 0, 0, 0, 0, \textbf{0}, 0, \textbf{1}, 0, 0, \textbf{1}, 0, \textbf{0}, 0, 0, 0, 0, 0, 1)^T$.}}
    \label{table:FN(y)}
\end{table}

\end{example}

\subsubsection{Updates Executed Using Matrices}

In principle, it is possible to construct $\cN(y)$ from scratch by enumerating all neighbours of $x$ directly. 
However, instead of performing this exhaustive enumeration, we use a more efficient approach by constructing two sparse matrices. 
These matrices are designed to systematically update the set of  neighbours from $\cN(x)$ to $\cN(y)$. 
By leveraging these matrices, we can perform targeted neighbour updates, avoiding the need to rebuild the full neighbourhood from the ground up.

Let the members in the set $\cN(x)$ be initialized to follow the \textit{lexicographic} order. The neighbours~$N_0,N_1, \ldots, N_9$ in Tables \ref{table:FN(x)} and \ref{table:FN(y)} show examples of the lexicographic ordering, which is determined by the first two elements~$z_1$ and~$z_2$.
Let the matrices $S_{12},S_{34}\in \{0,1\}^{\binom{n}{2}\times n^2}$, 
where the supports of the $i$-th rows are
\begin{equation}
    \label{eq:S12S34Construction}
\text{supp}(S_{12}[i,:]) =\{z_1, z_2\}, \ \ \text{supp}(S_{34}[i,:]) =\{z_3, z_4\} ,   \ \ \text{where }  (z_1,z_2,z_3,z_4) \text{ is the $i$-th member of } \mathcal{N} (x).
\end{equation}
Note that $\supp(x)=\supp( \mathbf{1}^TS_{12})$ and 
$\{0,\ldots, n^2-1\} \setminus \supp(x)=\supp( \mathbf{1}^TS_{34})$.

In order to detect the rows of $S_{12}$ and $S_{34}$ that correspond to the neighbours in $\cU_{\cN(x)}$, let $b\in \{0,1,2\}^{n^2}$, where
\begin{equation}
    \label{eq:blockPickerVecotr}
b_i = \begin{cases} 1 & \text{if } i =\bar{z}_1, \\
2 & \text{if } i =\bar{z}_2, \\
0 & \text{otherwise. }
\end{cases} 
\end{equation}
We then use $S_{12}b$ to detect the rows of $S_{12}$ that contain the members of
$\mathcal{U}_{\mathcal{N}(x)}$.
Consider the following sets:
\begin{equation}
    \label{eq:RowSets}
\cR_1 := \{i : (S_{12}b)_i = 1\},  \ 
\cR_2 := \{i : (S_{12}b)_i = 2\}, \ \text{ and } \
\cR_3 := \{i : (S_{12}b)_i = 3\} .
\end{equation}
The set 
$\cR_1$ contains the row indices of $S_{12}$ and $S_{34}$ that correspond to  the neighbours $ \mathcal{U}_{\mathcal{N}(x)}^{\bar{z}_1}$.
Similarly, $\cR_2$ and $\cR_3$ contain the row indices of  $S_{12}$ and $S_{34}$ corresponding to the neighbours 
$\mathcal{U}_{\mathcal{N}(x)}^{\bar{z}_2}$ and $\mathcal{U}_{\mathcal{N}(x)}^{\bar{z}_1,\bar{z}_2}$, respectively.
Of note, 
$|\cR_1| = |\cR_2| 
= |\mathcal{U}_{\mathcal{N}(x)}^{\bar{z}_1} |
= | \mathcal{U}_{\mathcal{N}(x)}^{\bar{z}_2} |
=n-2$
and $|\cR_3| = |\mathcal{U}_{\mathcal{N}(x)}^{\bar{z}_1,\bar{z}_2}|=1$.

Now we update $S_{12}$ to represent the first two elements in each $4$-tuple of  $\cN(y)$ as follows:
\begin{enumerate}
\item 
For all $i\in \cR_1$, let
$S_{12}(i,\bar{z}_1) = 0$  and $S_{12}(i,z_1^*) = 1$;
\item
For all $i\in \cR_{2}$,  let
$S_{12}(i,\bar{z}_2) = 0$  and $S_{12}(i,z_2^*) = 1$; and 
\item 
For $i\in \cR_3$, let
$S_{12}(i,[\bar{z}_1,\bar{z}_2])=0$ and 
$S_{12}(i,[z_1^*,z_2^*])=1$.
\end{enumerate}
Note that this update involves changing $2n-2$ elements of $S_{12}$ from 0 to 1, and another $2n-2$ changes from 1 to 0. 
See \Cref{fig:ExampleSupdate}(a), (b), and (c) for an illustration of the transition of this update for an instance of size $n=12$.
The update procedure for $S_{12}$ is outlined in \Cref{algo:S_update_12}.
\begin{figure}[ht]
    \centering
    \includegraphics[width=0.9\linewidth]{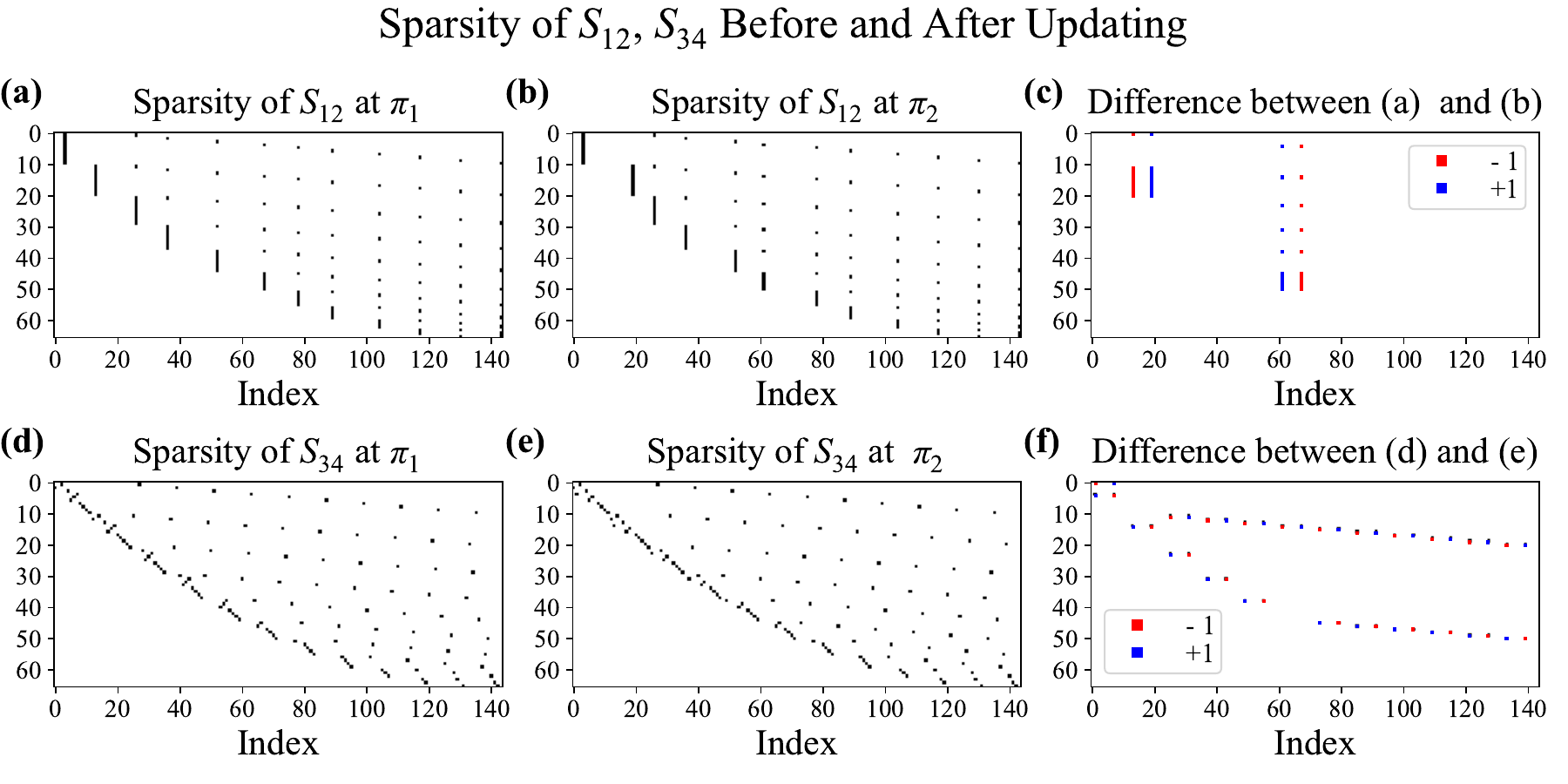}
    \caption{
    The transition of the full-neighbourhood lookup matrices $S_{12}, S_{34}$ at the permutation assignments $\pi_1 =[3,  \textbf{1},  2,  0,  4,  7 , 6,  \textbf{5} , 8 , 9 ,10 ,11]\in \Pi_{12}$ and $\pi_2=[ 3,  \textbf{5},  2,  0,  4,  7,  6,  \textbf{1},  8,  9, 10, 11]\in \Pi_{12}$.
    The assignment $\pi_2$ is 
     obtained after swapping the elements of $\pi_1$ at indices $1$ and $7$. (a) and (d) show the full neighbourhood of $\pi_1$;  (b) and (e) show the full neighbourhood of $\pi_2$. (c) highlights the difference between (a) and (b), and (f) highlights the between (d) and (e). Among $66 = n(n-1)/2$ members in the full neighbourhood of $\pi_1$, $21=2n-3$ members need updates, resulting in $44=4(n-1)$ element changes in both  $S_{12}$ and $S_{34}$.
    }
    \label{fig:ExampleSupdate}
\end{figure}
\begin{algorithm}[H]
\caption{~~Partial Neighbourhood Update for $S_{12}$}
\label{algo:S_update_12}
\begin{algorithmic}[1]
\Function{\texttt{Partial\_Neighbourhood\_Update\_S$_{12}$}}{$z_1,z_2,z_1^*,z_2^*$, $S_{12}$, $\cR_1, \cR_2, \cR_3$}
        \State $S_{12}(i,z_1) = 0, S_{12}(i,z_1^*) = 1 $, $i\in \cR_1$
        \State $S_{12}(i,z_2) = 0, S_{12}(i,z_2^*) = 1 $, $i\in \cR_2$
        \State $S_{12}(i,[z_1,z_2]) = 0, S_{12}(i,[z_1^*,z_2^*]) = 1 $, $i\in \cR_3$
    \State \Return $S_{12}$
\EndFunction
\end{algorithmic}
\end{algorithm}

Recall that a neighbour has an analogous meaning in the space of permutation matrices.
Once a neighbour is chosen from $\cN(\kvec(X))$, it analogously means that two columns of $X\in P_n$ are swapped while the other $n-2$ columns remain invariant. 
The indices of these invariant column $X$ are characterized by the set
\[
\cC_0 = \{0,1,2,...,n-1\}\setminus \left\{ \left\lfloor \frac{\bar{z}_1}{n}\right\rfloor,  \left\lfloor \frac{\bar{z}_2}{n}\right\rfloor \right\} .
\] 
Given a set $\mathcal{D}$ of numbers, 
we define the function $\texttt{List}_{\texttt{sorted}}(\mathcal{D})$ by
\[
\left( \texttt{List}_{\texttt{sorted}}(\mathcal{D})\right)_i = d_i,  \ 
\text{where } d \text{ is the vector of the elements of $\mathcal{D}$ sorted in non-decreasing order}.
\]
Let $c=\texttt{List}_{\texttt{sorted}}(\cC_0) \in \mathbb{R}^{n-2}$.
To update the rows in $\cR_1$ of $S_{34}$,
recall the assignments \eqref{eq:UFN_z1_thrid} and \eqref{eq:z1_nbr_formula}. 
The vector 
\[
c_1= n\cdot c+ (\bar{z}_1 \bmod n)
\]
indicates the columns of $S_{34}$  that need to be updated from a value of ``1'' to ``0''.
Similarly, the column indices of $S_{34}$ associated with $\cR_2$ that require updating from ``1'' to ``0'' are identified using  
\[
c_2= n \cdot c+ (\bar{z}_2 \bmod n).
\]
As for the coordinates of $S_{34}$ that require updating from ``0'' to ``1'',
recall \eqref{eq:z12star_z12bar} and note that
\[
(\bar{z}_2 \bmod n) - (z^*_1 \bmod n) + (\bar{z}_1 \bmod n)
= (\bar{z}_1 \bmod n).
\]
Thus, the elements of $S_{34}$ with columns corresponding to $n\cdot c + (z^*_1 \bmod n)$ and rows associated with $\cR_1$ can be identified using $c_2$ and $\cR_1$. 
Similarly, $c_1$ and $\cR_2$ can be used to identify the coordinates of $S_{34}$ that require updating from ``0'' to ``1''.
As with $S_{12}$,  updating $S_{34}$ involves changing $4n-4$ elements:  $2n-2$ elements changing from ``0'' to ``1'', and another $2n-2$ changing from ``1'' to ``0''.

Let $r_1=\texttt{List}_{\texttt{sorted}}(\cR_1)$
and 
$r_2=\texttt{List}_{\texttt{sorted}}(\cR_2)$.
Flipping the values of $S_{34}$ at the following coordinates provides the required update:
\begin{subequations} 
\begin{flalign}
  \{( r_1(i),c_1(i))\}_{i\in \{0,1,\ldots,n-3\}} , &  \quad \{( r_1(i),c_2(i))\}_{i\in \{0,1,\ldots,n-3\}} , 
      \label{eq:swapIndS34_R1} \\
   \{( r_2(i),c_1(i))\}_{i\in \{0,1,\ldots,n-3\}} ,  & \quad \{( r_2(i),c_2(i))\}_{i\in \{0,1,\ldots,n-3\}},
   \label{eq:swapIndS34_R2} \\
     \{( i ,\bar{z}_1),( i,\bar{z}_2)\}_{i\in \cR_3}  ,
     &  \quad  \{( i,z_1^* ),( i, z_2^*) \}_{i\in \cR_3}   . \label{eq:swapIndS34_R3}
\end{flalign}
\end{subequations}
The procedure for updating $S_{34}$ is outlined
in \Cref{algo:S_update_34}, and the transition of $S_{34}$ is illustrated in  \Cref{fig:ExampleSupdate}(d), (e), and (f) for an instance of size $n=12$.
The full-neighbourhood update executed using arrays is summarized in \Cref{algo:S_update_forall}.
The notation ``$\leftrightarrow$'' in \Cref{algo:S_update_34} stands for the swap operator.
We note that following the update governed by \Cref{algo:S_update_forall} preserves the lexicographic ordering of the subsequent full neighbourhoods.
\begin{algorithm}[H]
\caption{~~Partial Neighbourhood Update for $S_{34}$}
\label{algo:S_update_34}
\begin{algorithmic}[1]
\Function{\texttt{Partial\_Neighbourhood\_Update\_S$_{34}$}}{$\bar{z}_1,\bar{z}_2,z_1^*,z_2^*$, $S_{34}$, $r_1, r_2, \cR_3$}
        \State $\cC_0 \gets \{0,1,2,...,n-1\}\setminus \left\{ \left\lfloor \frac{\bar{z}_1}{n}\right\rfloor,  \left\lfloor \frac{\bar{z}_2}{n}\right\rfloor \right\}$
        \State $c \gets \texttt{List}_{\texttt{sorted}}(\cC_0)$
        \State $c_1 \gets n\cdot c +(\bar{z}_1 \bmod n)$
        \State $c_2 \gets n\cdot c +(\bar{z}_2 \bmod n)$
        \State $S_{34}(r_1[i],c_1[i]) \leftrightarrow S_{34}(r_1[i],c_2[i])$, \ 
        $S_{34}(r_1[i],c_2[i]) \leftrightarrow S_{34}(r_1[i],c_1[i])$, $i =0,1,\ldots, n-3$ \Comment{\eqref{eq:swapIndS34_R1}}
        \State $S_{34}(r_2[i],c_1[i]) \leftrightarrow S_{34}(r_2[i],c_2[i])$, \ 
        $S_{34}(r_2[i],c_2[i]) \leftrightarrow S_{34}(r_2[i],c_1[i])$, $i =0,1,\ldots, n-3$   \Comment{\eqref{eq:swapIndS34_R2}}
        \State $S_{34}(i,[\bar{z}_1,\bar{z}_2]) = 1, S_{34}(i,[z_1^*,z_2^*]) = 0 $, $i\in \cR_3$   \Comment{\eqref{eq:swapIndS34_R3}}
    \State \Return $S_{34}$
\EndFunction
\end{algorithmic}
\end{algorithm}
\begin{algorithm}[H]
\caption{~~Update of Full Neighbourhood}
\label{algo:S_update_forall}
\begin{algorithmic}[1]
\Function{\texttt{Update\_Full\_Neighbourhood}}{$S_{12}$, $S_{34}$, $z_1,z_2$}
\State $z_2^* \gets \left\lfloor \frac{z_2}{n}\right\rfloor \cdot n + \bmod(z_1, n) $, $z_1^* \gets \left\lfloor \frac{z_1}{n}\right\rfloor \cdot n + \bmod(z_2, n)$
\State $b \gets e_{z_1}+2e_{z_2}$ \Comment{The vector \eqref{eq:blockPickerVecotr}}
\State 
$\cR_1 \gets \{i : (S_{12}b)_i = 1\}, \
\cR_2 \gets \{i : (S_{12}b)_i = 2\}, \ 
\cR_3 \gets \{i : (S_{12}b)_i = 3\} $
\Comment{The sets \eqref{eq:RowSets}}
\State  $S_{12} \gets $ \texttt{Partial\_Neighbourhood\_Update\_S$_{12}$}($z_1,z_2,z_1^*,z_2^*$, $S_{12}$, $\cR_1, \cR_2, \cR_3$) \Comment{\Cref{algo:S_update_12}}
\State 
$r_1 \gets \texttt{List}_{\texttt{sorted}}(\cR_1)$,
$r_2 \gets \texttt{List}_{\texttt{sorted}}(\cR_2)$
\State  $S_{34} \gets $ \texttt{Partial\_Neighbourhood\_Update\_S$_{34}$}($z_1,z_2,z_1^*,z_2^*$, $S_{34}$, $r_1, r_2, \cR_3$) \Comment{\Cref{algo:S_update_34}}
\State \Return $S_{12},S_{34}$
\EndFunction
\end{algorithmic}
\end{algorithm}

\subsection{Vectorized Gradient Evaluation of the Full Neighbourhood for \\Symmetric and Semi-symmetric Instances}
\label{sec:vectoFNEval}

Given a feasible solution $x$ and its neighbour $y$, 
recall, from Proposition~\ref{prop:SemisymFormula}, that semi-symmetric instances allow for the simplification 
\[
\langle y,  Qy \rangle - \langle x, Qx \rangle 
= 
2 \cdot \left[ Q(z_3,:)y+Q(z_4,:)y-Q(z_1,:)x - Q(z_2,:)x \right] .
\]
In this section, we focus on symmetric and semi-symmetric instances. 
Let $y^i$ be the feasible solution realized by the $i$-th element of $\cN(x)$.
We aim to compute the vector $\Delta \in \mathbb{R}^{\binom{n}{2}}$, where the $i$-th component of
$\Delta$ is 
\[
\Delta_i =
Q(z_3,:)y^i+Q(z_4,:)y^i-Q(z_1,:)x - Q(z_2,:)x ,  \text{ where }
(z_1,z_2,z_3,z_4) \text{ is the $i$-th member of $\cN(x)$}.
\]
Each $\Delta_i$ represents the gradient of the $i$-th neighbour of $\cN(x)$. In our discussion, we omit the scalar multiplier ``$2$'' from Proposition~\ref{prop:SemisymFormula} since it does not affect the relative ordering of the elements in $\Delta$.

Given a vector $x\in \{0,1\}^{n^2}$, we use a matrix $P_{\supp(x)}:=I_{n^2}(\supp(x),:)\in \{0,1\}^{|\supp(x)|\times n^2}$ that is capable of pulling the elements of $x$ associated with the support $\supp(x)$.
Let $K=I_n \otimes \mathbf{1}_n \in \{0,1\}^{n^2\times n}$.
Then, for any vector $x$ with $\supp(x) = \{i_1,i_2,\ldots, i_n\}$, $K P_{\supp(x)} w$ copies 
the elements of $w$ corresponding to $\supp(x)$ exactly $n$ times, that is,
\begin{equation}
    \label{eq:breakComputation}
K (P_{\supp(x)} w)
=
( \underbrace{ w_{i_1}, w_{i_1}, \ldots, w_{i_1} }_{n \text{ times}}, 
\underbrace{w_{i_2}, w_{i_2}, \ldots, w_{i_2} }_{n \text{ times}},
\ldots ,
\underbrace{w_{i_n}, w_{i_n}, \ldots, w_{i_n} }_{n \text{ times}}
)^T .
\end{equation}
We decompose $\Delta_i$ into two parts:
\begin{equation}
    \label{eq:DeltaFormula}
\Delta_i = 
\tilde{\Delta}_i + \mathcal{E}_i ,
\end{equation}
where
\begin{equation}
    \label{eq:defofIthDeltaError}
\begin{array}{rcl}
\tilde{\Delta}_i
&=&
Q(z_3,:)x+Q(z_4,:)x-Q(z_1,:)x - Q(z_2,:)x  ,  \text{ and } \\
\mathcal{E}_i &=& Q(z_3,:) (y^i-x)  + Q(z_4,:) (y^i-x) .
\end{array}
\end{equation}
Propositions \ref{prop:gainComputation} and \ref{prop:errorComputation} show how the elements $\tilde{\Delta}_i$ and $\mathcal{E}_i$ associated with each neighbour in $\cN(x)$ can be computed using simple matrix--vector operations.
\begin{prop}
\label{prop:gainComputation}
Let $\cN(x)$ be the full neighbourhood of  
$x=\kvec(X)$, where $X \in P_n$.
Let $S_{34}$ be the array satisfying
the equalities~\eqref{eq:S12S34Construction}.
Then, the $i$-th element of the vector 
\begin{equation}
    \label{eq:approxGrad}
S_{34}g
\end{equation}
yields $\tilde{\Delta}_i$ defined in the equality~\eqref{eq:defofIthDeltaError}, where
\begin{equation}
   \label{eq:gainComputation}
\begin{array}{rcl}
g = -K  ( P_{\supp(x)} ((Qx )\circ x))
+ (Q x)\circ (\mathbf{1}_{n^2}-x) .  \\
 \end{array}
\end{equation}
\end{prop}
\begin{proof}
Let $\bar{x}$ be the negation of the current variable assignment $x$, that is, 
\[
\bar{x} = \mathbf{1}_{n^2}-x .
\]
We fix a row index $i$ of $S_{34}$.  
Observe that
\[
e_i^TS_{34} \left( (Qx)\circ \bar{x}\right)
=(e_{z_3}^T + e_{z_4}^T) \left( (Qx)\circ \bar{x}\right)
= Q(z_3,:)x+Q(z_4,:)x.
\]
Let $\supp(x) = \{i_1,i_2, \ldots, i_n\}$. Then,
\[
K  ( P_{\supp(x)} ((Qx )\circ x)) = ( \underbrace{ (Qx )_{i_1}, (Qx )_{i_1}, \cdots, (Qx )_{i_1} }_{n \text{ times}}, 
\ldots ,
\underbrace{(Qx )_{i_n}, (Qx )_{i_n}, \ldots, (Qx )_{i_n} }_{n \text{ times}}
)^T .
\]
Let
$\{z_1,z_2\} = \supp(e_i^TS_{12})$ be the corresponding counterpart of $\{z_3,z_4\}$.
The relationship  \eqref{eq:neighrPairRule} implies 
\[
e_i^T S_{34} \left( K  ( P_{\supp(x)} ((Qx )\circ x)) \right) = Q(z_1,:)x + Q(z_2,:)x.
\]
Then, with the vector 
\[g = - K  ( P_{\supp(x)} ((Qx )\circ x))
+ (Qx)\circ \bar{x} , 
\]
one can verify that 
$\tilde{\Delta} = S_{34}g$.
\end{proof}
We note that the matrix $I_n \otimes \mathbf{1}_n^T $  in the proof above can be used instead of the projection $P_{\supp(x)}$. Thus, a fixed matrix   $(I_n \otimes \mathbf{1}_n) (I_n \otimes \mathbf{1}_n^T) = I_n \otimes (\mathbf{1}_n \mathbf{1}_n^T)$ can be used instead of $K  P_{\supp(x)}$.
We illustrate the computation of the vector $g$ from the equality~\eqref{eq:gainComputation} in~\Cref{fig:gainComputationPic}.
\begin{figure}[H]
    \centering
    \includegraphics[width=0.75\linewidth]{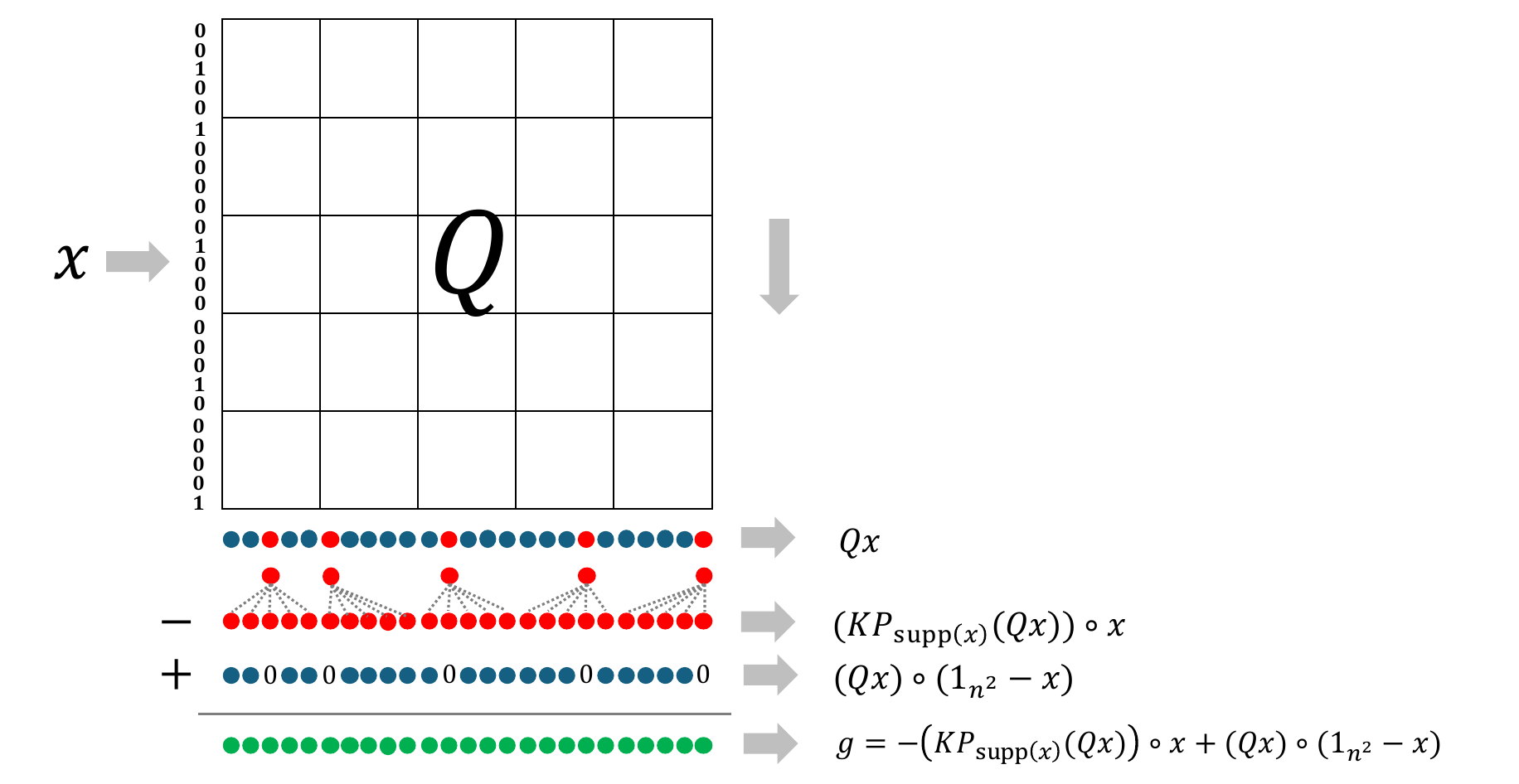}
    \caption{Diagram of the computation of the equality~\eqref{eq:gainComputation} at variable assignment \mbox{$x=\kvec \left( \begin{bmatrix}  e_2 & e_0 & e_1 & e_3 & e_4    \end{bmatrix} \right)$.}}
    \label{fig:gainComputationPic}
\end{figure}
\begin{prop}
\label{prop:errorComputation}
Let $\cN(x)$ be the full neighbourhood of 
$x=\kvec(X)$, where $X \in P_n$.
Let $S_{34}$ be the array satisfying the equality~\eqref{eq:S12S34Construction}.
Then, the $i$-th element of the vector
\begin{equation}
\label{eq:errorCorrectorVector}
\begin{array}{rcl}
\mathcal{E}
=\left((S_{34}Q)\circ S_{34}  \right) \mathbf{1}_{n^2} 
 \end{array}
\end{equation}
yields the equality~\eqref{eq:defofIthDeltaError}.
\end{prop}
\begin{proof} 
For a fixed $i$-th neighbour $(z_1,z_2,z_3,z_4) \in \cN(x)$, we observe that
\begin{equation}
    \label{eq:errorProofDecomp}
\begin{array}{rcl}
\mathcal{E}_i &=& Q(z_3,:) (y^i-x)  + Q(z_4,:) (y^i-x) \\
&=&  Q(z_3,:) (\phi_y-\phi_x)  +  Q(z_4,:) (\phi_y-\phi_x)  \\
&=&
(Q(z_3,:)+Q(z_4,:)) \phi_y -(Q(z_3,:)+Q(z_4,:)) \phi_x ,
\end{array}
\end{equation}
where $\phi_x$ and $\phi_y$ are as defined in the equalities~\eqref{eq:diffphi_xy}.

We first consider the terms $(Q(z_3,:)+Q(z_4,:)) \phi_y$ in the equality~\eqref{eq:errorProofDecomp}.
For each $i$, 
 observe that
\[
e_{i}^TS_{34}Q e_j = Q(z_3,j)+Q(z_4,j), \text{ where }
\{z_3,z_4 \} = \supp(e_i^T S_{34}). 
\]
Then, the $i$-th element of the column sum of $((S_{34}Q)\circ S_{34})$ is
\[
e_i^T\left((S_{34}Q)\circ S_{34}  \right) \mathbf{1}_{n^2}
=
Q(z_3,z_3)+Q(z_3,z_4)+Q(z_4,z_3)+Q(z_4,z_4)\\
=
Q(z_3,z_4)+Q(z_4,z_3) .
\]
Thus, 
$e_i^T\left((S_{34}Q)\circ S_{34}  \right) \mathbf{1}_{n^2} = (Q(z_3,:)+Q(z_4,:)) \phi_y$.

We now consider $(Q(z_3,:)+Q(z_4,:)) \phi_x$ in the equality~\eqref{eq:errorProofDecomp}. We aim to show that $Q(z_3,:) \phi_x=0$ and $Q(z_4,:)\phi_x=0$.
Recall the algebraic relationship \eqref{eq:z1-z4AlgReltation} among $z_1,z_2,z_3$, and $z_4$.
For simplicity, we let 
\[
\begin{array}{lcl}
z_1 = kn + t_1, & & 
z_2 = jn+  t_2,  \\
z_3 = jn+  t_1, & & 
z_4 = kn+ t_2  .\\
\end{array}
\]
Then, we may represent $e_{z_1}=e_{kn+ t_1} \in \{0,1\}^{n^2}$ using the Kronecker product
\[
e_{z_1}  = e_k \otimes e_{t_1} , \text{ where } e_k, e_{t_1} \in \{0,1\}^{n}. 
\]
Similarly, with $e_k, e_j, e_{t_1},e_{t_2}\in \{0,1\}^n$, we have
\[
\begin{array}{lcl}
e_{z_2} = e_j \otimes e_{t_2}, \ 
e_{z_3} = e_j \otimes e_{t_1},  \text{ and } 
e_{z_4} =  e_k \otimes e_{t_2} .  
\end{array}
\]
Note that
$Q(z_3,:)\phi_x = e_{z_3}^TQ e_{z_1} + e_{z_3}^TQ e_{z_2}$.
We let $Q = D\otimes F$ without loss of generality.
Now observe the expansion
\[
\begin{array}{rcl}
e_{z_3}^TQe_{z_1} 
&= & (e_j \otimes e_{t_1})^T (D\otimes F) ( e_k \otimes e_{t_1} ) \\
&= & (e_j \otimes e_{t_1})^T (De_k \otimes Fe_{t_1})   \\
&= & (e_j^T  De_k ) \otimes (e_{t_1}^T F  e_{t_1})  \\
&= & 0,
\end{array}
\]
where the last equality holds, since $F$ is assumed to have the zero-diagonal property.
Similarly, since $D$ is also assumed to have the zero-diagonal property, we have 
\[
\begin{array}{rcl}
e_{z_3}^TQe_{z_2} 
&= & (e_j \otimes e_{t_1})^T (D\otimes F) ( e_j \otimes e_{t_2}) \\
&= & (e_j \otimes e_{t_1})^T (De_j \otimes Fe_{t_2})   \\
&= & (e_j^T  De_j ) \otimes (e_{t_1}^T F  e_{t_2})  \\
&= & 0 .
\end{array}
\]
Using a similar expansion, one can obtain the equality $Q(z_4,:)\phi_x=0$.
Therefore, the element 
$e_i^T\mathcal{E}$ is equal to~$Q(z_3,:) (y^i-x)  + Q(z_4,:) (y^i-x)$.
\end{proof}
We note that the assumption that  $D$ and $F$ both having the zero-diagonal property plays a crucial role in deriving the formula~\eqref{eq:errorCorrectorVector}. 
As shown in equality \eqref{eq:errorProofDecomp}, the value $\mathcal{E}_i$ comprises up to eight elements of the matrix $Q$. 
However, because the matrices $D$ and $F$ both have the zero-diagonal property, the value $\mathcal{E}_i$ comprises at most two elements of $Q$, as shown in the proof.

The steps for computing $\Delta$ using  Propositions~\ref{prop:gainComputation} and~\ref{prop:errorComputation} is summarized in \Cref{algo:FN_Local_Eval}. 
We view $\tilde\Delta$ as an approximate gradient and $\Delta$ as the exact gradient. 
\Cref{algo:FN_Local_Eval} is illustrated in Example~\ref{example:parallelEvel}, with a problem size of $n = 5$.
\begin{algorithm}[H]
\caption{~~Full-Neighbourhood Evaluation}
\label{algo:FN_Local_Eval}
\begin{algorithmic}[1]
\Function{\texttt{Evaluate\_Full\_Neighbourhood}}{$x$, $S_{12},S_{34}$}
        \State $g \gets - K  ( P_{\supp(x)} ((Qx )\circ x))
+ (Qx)\circ ( \mathbf{1}_{n^2}-x) $ \Comment{The equality~\eqref{eq:gainComputation}}
        \State $\Delta \gets S_{34}g$  \Comment{The approximate gradient~\eqref{eq:approxGrad}}
        \If{exact gradient is desired}
        \State $\Delta \gets \Delta +  \left((S_{34}Q)\circ(S_{34})  \right) \mathbf{1}_{n^2}$ \Comment{The exact gradient using the error corrector~\eqref{eq:errorCorrectorVector}} 
        \EndIf
    \State \Return $\Delta$
\EndFunction
\end{algorithmic}
\end{algorithm}
\begin{example}(Illustration of \Cref{algo:FN_Local_Eval})
\label{example:parallelEvel}
We continue with the permutation assignment $\pi=[2,0,1,3,4]$ from Example \ref{example:FNillustration}.
Consider the following distance and flow matrices:
\[
D=\begin{bmatrix}
     0 & 1 & 0 & 1 & 4 \\
 1 & 0 & 4 & 2 & 0 \\
 0 & 0 & 0 & 3 & 5 \\
 5 & 3 & 1 & 0 & 1 \\
 4 & 5 & 0 & 3 & 0 \\
\end{bmatrix} \in \mathbb{R}^{5\times 5}, \quad
F = \begin{bmatrix}
     0 & 2 & 1 & 3 & 4 \\
 2 & 0 & 3 & 2 & 2 \\
 1 & 3 & 0 & 7 & 0 \\
 3 & 2 & 7 & 0 & 3 \\
 4 & 2 & 0 & 3 & 0 \\
\end{bmatrix} \in \mathbb{S}^5 .
\]
Let $Q=\frac{1}{2}(D\otimes F)+\frac{1}{2}(D\otimes F)^T$ be formed via \Cref{algo:quboMatFormulation}.
Then
\[ 
K  ( P_{\supp(x)} ((Qx )\circ x)) = 
(I_5 \otimes \mathbf{1}_5)  
\begin{pmatrix}
    11.5 & 15 & 13 & 20.5 & 21
\end{pmatrix}^T \in \mathbb{R}^{25},
\]
and
\[
\begin{array}{ccrrrrrrrrrrrrr}
(Qx)\circ (\mathbf{1}_{25}-x) &=& (  16 &  13&   0&  12&   8.5& 0&  13&  23.5& 18.5& 11.5&  10 &   0   \ldots \\ 
& &16&  7.5 &
 14&  13.5& 13.5 & 8.5 & 0&  14&   9&  21&  24&  33&   0 )^T \in \mathbb{R}^{25}.
\end{array}
\]
Hence, the vector $g$ in the equality~\eqref{eq:gainComputation} is 
\[
\begin{array}{ccrrrrrrrrrrrrrr}
g &=& ( 4.5 &   1.5 &-11.5&   0.5&  -3&  -15&   -2&    8.5&   3.5&  -3.5&  -3 & -13  \ldots\\
 &&   3&   -5.5&   1&   -7&   -7&  -12&  -20.5&  -6.5& -12&    0&    3&   12& -21 )^T \in \mathbb{R}^{25}.
\end{array} 
\]
The computations of the vector~\eqref{eq:approxGrad}, representing the approximate gradient $S_{34}g$, and the equality~\eqref{eq:errorCorrectorVector}, representing the error correction vector $\mathcal{E}$, are 
\[
\begin{array}{lclrrrrrrrrrr}
 \tilde{\Delta} &=& 
(    13& 4.5& -11.5&   0&   -5&   -3.5& -15.5& -12.5&   1&  5.5
 )^T  ,\\
 \mathcal{E} &=&
(  \,\,\, 2&  0& 21&  0&  8&  0& 20&  8& 10& 12
   )^T . 
 \end{array}
\]
The exact gradient is computed by
$\Delta= \tilde\Delta + \mathcal{E} = \begin{pmatrix}
    15  &4.5& 9.5&  0&  3& -3.5 & 4.5& -4.5 &11 &  17.5
\end{pmatrix}^T$.
\end{example}

\subsection{Parallel Full-Neighbourhood Search Algorithm}

We summarize all of the routines introduced in
 \Cref{sec:localUpdate}, \Cref{sec:PartialFNUp}, and \Cref{sec:vectoFNEval} in
 \Cref{algo:FN_All_routine}.

\begin{algorithm}[H]
\caption{~~Parallel Full-Neighbourhood Local Search Algorithm}
\label{algo:FN_All_routine}
\begin{algorithmic}[1]
\Function{\texttt{Full\_Neighbourhood\_Local\_Search}}{$F \gets$ flow matrix, $D\gets $ distance matrix,  $X \gets $ permutation matrix, $i_\text{max}$ $\gets$ maximum number of iterations} 
        \State $Q \gets \texttt{Formulate\_QUBO\_Matrix}(F,D)$ \Comment{\Cref{algo:quboMatFormulation}} 
        \State $x\gets \kvec(X)$ 
        \State $(f_{\text{best}}, x_{\text{best}} ) \gets ( \langle x, Qx \rangle , x) $ \State Initialize $S_{12},S_{34}$ \Comment{The equality~\eqref{eq:S12S34Construction}}
        \State $i \gets 0$
        \While{$i <$ $i_\text{max}$ }
            \State $\Delta \gets$ \texttt{Evaluate\_Full\_Neighbourhood}($x$, $S_{12}$, $S_{34}$)  \label{line:algoFNevalLine}
            \Comment{\Cref{algo:FN_Local_Eval}}
            \State $(z_1,z_2,z_3,z_4) \gets $  \texttt{Choose\_Move}($S_{34}, \Delta$) \label{line:pickHeuristic} \Comment{A heuristic of choice}
            \State $x(z_1) \gets 0, \  x(z_2) \gets 0, \ x(z_3) \gets 1, \ x(z_4)\gets 1  $ \Comment{A quadruple bit-flip} 
        \If{$\langle x, Qx \rangle <f_{\text{best}}$}
            \State $(f_{\text{best}}, x_{\text{best}} ) \gets ( \langle x, Qx \rangle , x) $ 
        \EndIf
        \State $S_{12} ,S_{34} \gets$  \texttt{Update\_Full\_Neighbourhood}($S_{12} ,S_{34},z_1,z_2$)    \label{line:algoUpdateS}
        \Comment{\Cref{algo:S_update_forall}}
        \State $i \gets i + 1$
        \EndWhile
    \State \Return $(f_{\text{best}}, x_{\text{best}} )$
\EndFunction
\end{algorithmic}
\end{algorithm}

We make a few remarks pertaining to \Cref{algo:FN_All_routine}.
The explicit heuristic is not specified the function \texttt{Choose\_Move} (line \ref{line:pickHeuristic} of \Cref{algo:FN_All_routine}).
This flexibility is a strength of our approach and allows for the choosing of the most appropriate local search heuristic. This adaptability ensures that the approach can be easily integrated with a wide range of heuristics available in the literature. 
The array $S_{34}$ is passed as an input to \texttt{Choose\_Move} because, once a value $\Delta_i$ provides the gradient, the corresponding neighbour must be identified for the quadruple bit-flip. 
The computational cost associated with the error correction term 
$\left((S_{34}Q)\circ S_{34})  \right) \mathbf{1}_{n^2}$ is significant in the function
\texttt{Evaluate\_Full\_Neighbourhood} (line \ref{line:algoFNevalLine} of \Cref{algo:FN_All_routine}) primarily due to the presence of the matrix--matrix product $S_{34}Q$.
As such, we may consider omitting the error correction routine and rely on the vector~\eqref{eq:approxGrad} which is the approximation. 
There are two rationales behind this.
Each $\mathcal{E}_i$ involves two elements of $Q$. 
Meanwhile, 
given $y^i$ derived from the $i$-th member $(z_1,z_2,z_3,z_4)$ 
of $\cN(x)$, 
the exact gradient is computed via
\[
\Delta_i =
Q(z_3,:)y^i+Q(z_4,:)y^i-Q(z_1,:)x - Q(z_2,:)x .
\]
Note that $\Delta_i$ involves $4n$ elements of $Q$.
Therefore, by excluding $\mathcal{E}_i$,
only two out of $4n$ terms are miscalculated.
As the problem order $n$ increases, the impact on these two error terms diminishes.
At the same time, if the algorithm always chooses the move that leads to the best immediate reduction in the value of the objective function, it many become trapped in a local minimum. 
Hence, a heuristic mechanism is needed to help the algorithm escape from local minima. 
Moreover, selecting a good candidate from a full-neighbourhood search requires relative ordering of values in $\Delta$, rather than the values themselves.
We perform benchmarking with respect to the impact of omitting the error correction steps in~\Cref{sec:ExactApproxNumeric} below.
In \Cref{sec:IMCimplementation}, we provide a blueprint of the circuit implementation of \Cref{algo:FN_All_routine}, intended to bridge the gap between our algorithmic framework and its physical hardware realization. We also discuss operational challenges involved in realizing the algorithm using analog IMC hardware, as well as potential mitigation strategies.

\section{Numerical Experiments}
\label{sec:Numerics}

We used instances from the QAPLIB~\cite{QAPLIBurl} of various sizes ranging from $n=12$ to $n=100$ to evaluate the performance of \Cref{algo:FN_All_routine}.
All instances from the QAPLIB follow a size-based naming convention, ``identifier\_$n$\_version'', where $n$ indicates the problem size.
The relative optimality gap was chosen as the performance metric.
Let 
\begin{equation}
    \label{eq:fbest_notation}
f_\text{best} (i_\text{max} )
:= \min_x \left\{ { \langle x,Qx \rangle : x \text{ visited by \Cref{algo:FN_All_routine} up to $i_\text{max}$}   } \right\}.
\end{equation}
We let $f^*$ be the best known objective value obtained from the QAPLIB. We define the relative optimality gap as a function of the maximum allowed iterations:
\begin{equation}
    \label{eq:relgap_formula}
\relgap(i_\text{max}) :=
\frac{f_\text{best}(i_\text{max}) - f^*}{f^*}.
\end{equation}
We let $\mathcal{T}$ represent the set of trials, or repeats, that is, the number of times an instance is called by \Cref{algo:FN_All_routine}, and ${\relgap}^t$ be the relative gap obtained at the $t$-th trial.
The best optimality gap is then quantifiable by the minimum of the relative optimality gap:
\[
\relgapMin(i_\text{max}) := \min_{t \in \mathcal{T}} \left\{ {\relgap}^t(i_\text{max}) \right\}, \text{ for } i_\text{max} \in \mathbb{N}.
\]
We evaluate our approach in two settings of gradient computations. 
Let $x$ be a feasible binary solution with its neighbour $y=x-e_{z_1}-e_{z_2}+e_{z_3}+e_{z_4}$, where $(z_1,z_2,z_3,z_4) \in \cN(x)$.
One approach uses the exact gradient $\Delta$ using the equality~\eqref{eq:DeltaFormula},
while the other employs the approximation $\tilde\Delta$ (see Proposition~\ref{prop:gainComputation}).
In the context of  \Cref{sec:vectoFNEval},
we evaluate gradients in $\tilde\Delta + \mathcal{E}$ or $\tilde\Delta$.
We use the symbol $\Theta$ to mean the following:
\[
\Theta =
\begin{cases}
    \tilde\Delta + \mathcal{E} & \text{if the exact gradient is used,} \\
    \tilde\Delta & \text{if the approximate gradient is used.}
\end{cases}
\]
\subsection{Heuristics}
\label{sec:Heuristics}

We evaluate five heuristics that can be employed in the \texttt{Choose\_Move} function (line \ref{line:pickHeuristic} in Algorithm \ref{algo:FN_All_routine}).
The first heuristic, labelled ``\texttt{Greedy}'',  is the greedy approach.
Given a set of candidates $\cN(x)$ and their corresponding  value of $\Theta$, 
we select the one that results in the smallest value, that is, 
\begin{equation}
    \label{eq:greedyHeu}
N^* =  \texttt{Greedy}(x)  \in \Argmin_{N} \{ \Theta :  N  \in \cN(x) ,\text{ with } \Theta \text{ computed from  \Cref{algo:FN_Local_Eval}} \} .
\end{equation}
The second heuristic is a pseudo-greedy approach, labelled ``\texttt{Top10}''.
This heuristic randomly chooses one neighbour from the set of neighbours that yields the top 10 best elements in $\Theta$. 
Let $\widehat{\Theta}=\texttt{List}_{\texttt{sorted}}(\Theta)$, with $ \Theta$  computed using  $\cN(x)$. Let 
$\pi\in \Pi\left(\binom{n}{2} \right)$
be the permutation that yields the ordered vector  $\widehat{\Theta} = \Theta(\pi)$.
We then select a neighbour at random:
\begin{equation}
    \label{eq:top10Heu}
N^* =  \texttt{Top10}(x) 
 \in \left\{ \pi^{-1}(0) , \pi^{-1}(1), \ldots,   \pi^{-1}(9)      \right\} .
\end{equation}
The third heuristic we use is a random walk heuristic inspired by the well-known WalkSAT heuristic~\cite{walkSATBreak_KautzSelmanCohen} in the area of satisfiability problems.
Given a probability $p\in (0,1)$, this heuristic chooses a neighbour based on the following rule:
\begin{equation}
    \label{eq:walkSATHeu}
N^* =
\texttt{WalkQAP}(x)
:=\begin{cases}
\texttt{Top10}(x) & \text{with the probability $p$},    \\
\text{randomly selected neighbour from $\cN(x)$} & \text{with the probability $1-p$.}
\end{cases}
\end{equation}

The fourth heuristic is the tabu search heuristic~\cite{Taillard91}, labelled ``\texttt{Tabu}''.
In brief, a tabu search constructs a so-called tabu list, which records  the feasible points visited up to a given limit. 
Once $\widehat{\Theta}=\texttt{List}_{\texttt{sorted}}(\Theta)$ is computed, this heuristic sequentially iterates over the neighbours within the full neighbourhood.
If a neighbour $i$ results in the value of $f(x) + \widehat{\Theta}_i$ being better than the best objective value found so far, we select that neighbour for the next iteration. 
Otherwise, we choose the incumbent neighbour, provided it is not in the tabu list. 
Once a neighbour is selected, it is appended to the tabu list, and if the tabu list is at its size limit, the oldest item in the list is removed. 

The final heuristic we use is simulated annealing, labelled ``\texttt{SA}''.
The initial temperature $T_\text{high}$ is estimated as 
$\delta_h/ \log(p_\text{initial})$, where $\delta_h$ is the $\gamma_\text{initial}$-th percentile of sampled nonzero absolute objective differences and $p_\text{initial}$ is a given acceptance probability.
Similarly, the final temperature is estimated as 
$T_\text{low} = \delta_l / \log(p_\text{final})$, where $\delta_l$ is the $\gamma_\text{final}$-th percentile of the same distribution and $p_\text{final}$ is a given acceptance probability.
These energy differences are obtained by evaluating configurations drawn randomly from the solution space. 
We then form the temperature schedule following an exponential cooling scheme, where the temperature at the $i$-th annealing step is defined as 
$T_i = T_\text{high} \left(\frac{T_\text{low}}{T_\text{high}}\right)^{\frac{i}{i_{\max} -1}}$.
Given the temperature sequence $\{T_i\}_{i=1}^{i_{\max}}$ and a current solution $x_c$,  a candidate solution $x_i \in \cN(x_c)$ is accepted at iteration $k$ with probability
$\min \left\{ 1, \ \exp\left[ - \frac{ (f(x_c) + \Theta_i )-f(x_c)}{T_k} \right] \right\}$.

The parameters used for the heuristics is as follows. 
For the \texttt{WalkQAP} heuristic, we set the probability parameter $p=0.95$.
For the \texttt{Tabu} heuristic, we set the maximum length of the tabu list to $20$, with the oldest entry being removed when a new neighbour is added whenever the list exceeds this limit.
For the \texttt{SA} heuristic, objective value differences are sampled from 10 random configurations, each with 10 neighbours.  
The parameters settings for $T_\text{high}$ and $T_\text{low}$ are as follows:
$p_\text{initial}= 0.8$,
$p_\text{final} = 0.1$,
$\gamma_\text{initial}= 50$, and
$\gamma_\text{final} = 5$.

\subsection{Comparison of Exact and Approximate Gradients}
\label{sec:ExactApproxNumeric}
In this section, we discuss the results of our experiments using the heuristics listed in \Cref{sec:Heuristics} for two settings of $\Theta$. 
The number of repeats was set to $100$, that is, $|\cT|=100$.
The Intel Xeon E7-8890 v3 running at 2.50 GHz with 144 cores  was used for our experiments. 
Our aim is to assess the impact of omitting the error correction term $\mathcal{E}$ presented in Proposition~\ref{prop:errorComputation}.
The primary motivation for omitting this step stems from its computational cost.
Computing $\mathcal{E}$
involves a matrix--matrix multiplication $S_{34}Q$, which requires more-substantial processing resources than a matrix--vector multiplication. 
These operations significantly increase the overall computation time, especially as the problem size grows.
The second reason is that when it comes to choosing a neighbour in a given full neighbourhood, the relative ordering of the values in $\Theta$ is important, not the values themselves.
For example, given $\Theta^1 = (-70,-40,-100,3,5)$ and $\Theta^2 = (-80,-50,-90,3,5)$ representing two distinct gradients,  we select the neighbour associated with 
the smallest value, which is the third element of each vector. 
Finally, IMC hardware devices typically have limited bit-precision capabilities. 
As a result, quantization of the nonzero elements in $Q$ becomes necessary. 
This introduces another potential source of error due to the loss of precision. 
Consequently, it is crucial to design algorithms that are resilient to multiple sources of error.
Despite these challenges, incorporating the error correction step may still be feasible in certain scenarios, particularly when matrix--matrix multiplications can be efficiently performed \cite{Wang2021Scalable}. 

\Cref{tab:heuristic_gaps} shows the minimum relative optimality gaps, $\relgapMin(10^5)$, achieved by each heuristic under both approximate and exact gradient methods across 20 QAPLIB instances. Each value reflects the best relative gap obtained over $10^5$ iterations over 100 trials. For most instances, the \texttt{Top10} and \texttt{WalkQAP} heuristics consistently achieve near-optimal values. However, \texttt{Greedy} and \texttt{Tabu} heuristics exhibit higher gaps when the approximate gradient is used. We analyze the behaviour of the \texttt{Greedy}, \texttt{Tabu}, and \texttt{Top10} heuristics below.
\begin{table}[ht]
\centering
\begin{tabular}{l|rrrrr|rrrrr}
\toprule
 & \multicolumn{5}{c|}{Approximate Gradient} & \multicolumn{5}{c}{Exact Gradient} \\
  & \texttt{Greedy} & \texttt{SA} & \texttt{Tabu} & \texttt{Top10} & \texttt{WalkQAP} & \texttt{Greedy} & \texttt{SA} & \texttt{Tabu} & \texttt{Top10} & \texttt{WalkQAP} \\
\midrule
chr12a & 97.30 & 0 & 97.30 & 0 & 0 & 3.81 & 0 & 0 & 0 & 0 \\
chr15a & 199.13 & 0 & 199.13 & 0 & 0 & 0 & 0 & 0 & 0 & 0 \\
chr25a & 184.88 & 2.32 & 184.88 & 0 & 0 & 16.33 & 0 & 0 & 0 & 0 \\
esc16a & 0 & 0 & 2.94 & 0 & 0 & 0 & 0 & 0 & 0 & 0 \\
esc32e & 200.00 & 0 & 200.00 & 0 & 0 & 0 & 0 & 0 & 0 & 0 \\
had12 & 1.82 & 0.24 & 1.82 & 0 & 0 & 0 & 0 & 0 & 0 & 0 \\
had20 & 1.59 & 1.47 & 1.53 & 0 & 0 & 0 & 0 & 0 & 0 & 0 \\
lipa20a & 3.58 & 3.34 & 3.58 & 0 & 0 & 1.06 & 0 & 0 & 0 & 0 \\
lipa40a & 2.03 & 2.57 & 2.03 & 0 & 0.74 & 1.25 & 0 & 0 & 0 & 0 \\
nug20 & 4.90 & 2.49 & 4.75 & 0 & 0 & 0.54 & 0 & 0 & 0 & 0 \\
nug30 & 6.30 & 1.53 & 5.98 & 0 & 0 & 1.01 & 0 & 0 & 0 & 0 \\
sko42 & 4.45 & 2.57 & 3.45 & 0 & 0 & 1.39 & 0 & 0 & 0 & 0 \\
sko49 & 4.06 & 1.98 & 4.20 & 0 & 0.13 & 1.17 & 0 & 0 & 0.01 & 0.06 \\
tai12a & 2.08 & 0 & 2.08 & 0 & 0 & 0 & 0 & 0 & 0 & 0 \\
tai25a & 4.07 & 5.77 & 4.07 & 0.89 & 0 & 1.91 & 0 & 0 & 0 & 0.41 \\
tai40a & 4.28 & 6.79 & 4.28 & 0.81 & 1.34 & 2.16 & 0.33 & 0.12 & 0.33 & 0.84 \\
tai50a & 4.37 & 7.16 & 4.37 & 0.94 & 1.18 & 2.91 & 0.75 & 0.50 & 0.72 & 1.34 \\
tho30 & 3.36 & 1.41 & 3.36 & 0 & 0 & 1.39 & 0 & 0 & 0 & 0 \\
tho40 & 6.02 & 3.06 & 6.02 & 0 & 0.05 & 2.19 & 0 & 0 & 0 & 0 \\
wil50 & 1.69 & 2.20 & 1.10 & 0 & 0.02 & 0.62 & 0 & 0 & 0 & 0.02 \\
\bottomrule
\end{tabular}
\caption{Minimum relative optimality gaps observed over $10^5$ iterations using the five heuristics.}
\label{tab:heuristic_gaps}
\end{table}

\subsubsection{Typical Behaviour of the Greedy and Tabu Search Heuristics}

Given a feasible binary solution $x$, 
the \texttt{Greedy} heuristic iteratively selects the neighbour within $\cN(x)$ that yields the maximum immediate reduction in the objective function. While this strategy ensures the best possible improvement at each step within the incumbent full neighbourhood, if often leads to an oscillatory pattern, where the search alternates between two permutation assignments. 
As a result, the algorithm fails to make further progress towards a global optimal solution.
This pattern is evident for five of the selected QAPLIB instances, as shown in \Cref{fig:typical_alternating}.
We make observations for both settings of $\Theta$, that is, $\Delta$ and $\tilde\Delta$.
\begin{figure}[ht]
    \centering
    \includegraphics[width=0.8\linewidth]{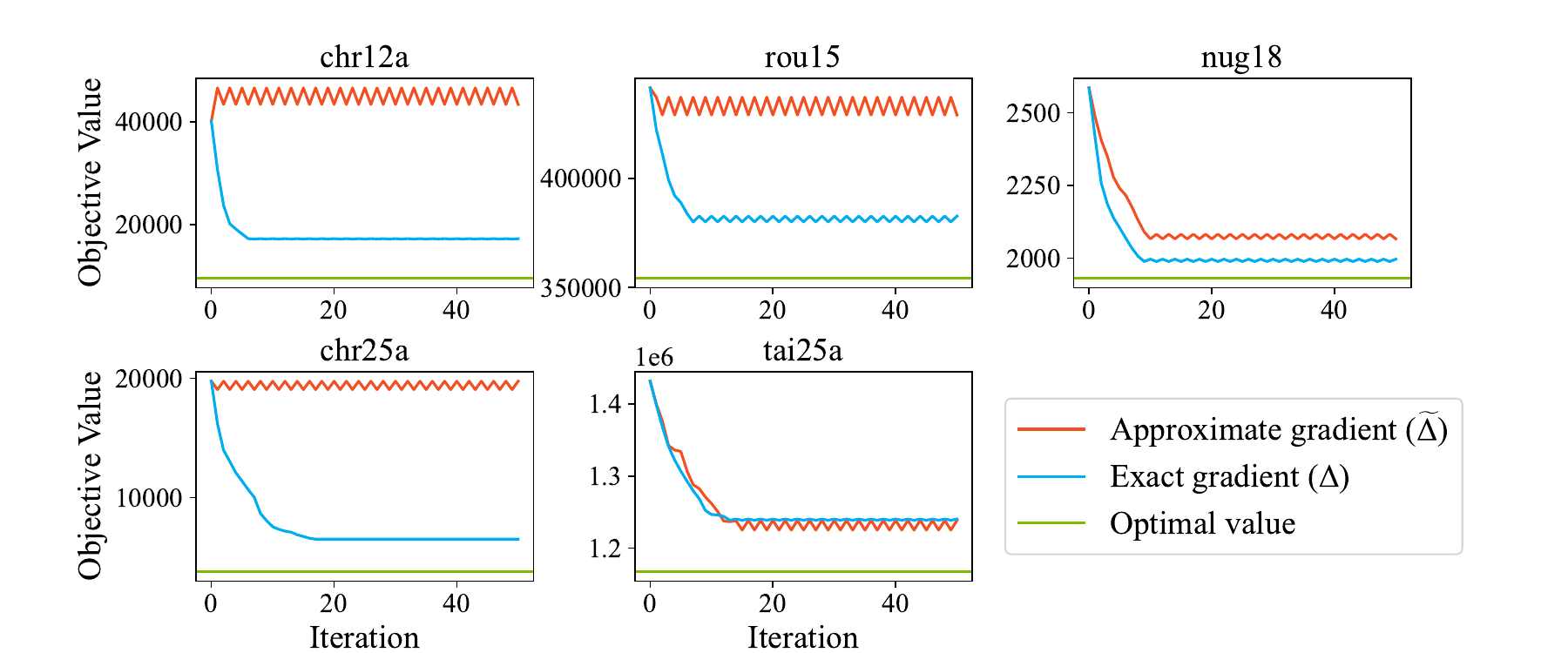}
    \caption{Objective value trajectories of the \texttt{Greedy} heuristic on five QAPLIB instances over 50 iterations of full-neighbourhood evaluations, beginning with the identity permutation $X=I$.}
    \label{fig:typical_alternating}
\end{figure}
From the sixth iteration of the QAPLIB instance ``chr12a'', when the exact gradient is used, 
the \texttt{Greedy} search heuristic alternates between the permutation assignments $[ 2, 6, 0, 1, 4, 5, 7, 11, 3, 9, 8, 10]\in \Pi_{12}$ and $[ 2, 6, 0, 1, 4, 5, 7, 11, 3, 9, 10, 8]\in \Pi_{12}$, yielding the objective values 17,210 and 17,278, respectively.
Similarly, when the approximate gradient is used, the \texttt{Greedy} heuristic alternates between 
$[ 0, 5, 2, 3, 4, 1, 6, 7, 8, 9, 10, 11]\in \Pi_{12}$ and
$[ 1, 5, 2, 3, 4, 0, 6, 7, 8, 9, 10, 11]\in \Pi_{12}$ from the first iteration, 
yielding the objective values $46,646$ and $43,420$, respectively.
Such alternating behaviour is a commonly observed pattern, which indicates that the \texttt{Greedy} search heuristic cannot escape local minima. 

Given a feasible binary solution $x$ and its gradient $\Theta$ evaluated at $x$, 
the \texttt{Tabu} heuristic begins by sorting the elements of $\Theta$. It then iteratively checks the condition $\Theta_i < f_\text{best}- \langle x,Qx\rangle$, where $\Theta_i$ corresponds to the $i$-th neighbour in $\cN(x)$. 
If this condition is satisfied, then the heuristic accepts that neighbour and proceeds to the next iteration. 
If $\Theta=\Delta$, then the condition reflects that the correct decision has been made, as $\Delta_i$ represents the true difference in the objective values. 
However, when an approximate gradient $\Theta=\tilde\Delta$ is used, discrepancies can occur due to gradient estimation errors. Specifically, if the error term is omitted from 
$\tilde\Delta_i + \mathcal{E}_i <  f_\text{best}- \langle x,Qx\rangle$, the comparison may lead to the acceptance of an incorrect neighbour. In other words, decisions based solely on $\tilde\Delta_i$ may be unreliable in the presence of approximation errors. 
This behaviour is illustrated in~\Cref{fig:typical_alternating_tabu}.
When the exact gradient is used, the \texttt{Tabu} heuristic avoids cycling between configurations. In contrast, with the approximate gradient, its behaviour resembles that of the \texttt{Greedy} heuristic for the five instances shown in the figure, alternating between suboptimal configurations.
\begin{figure}[ht]
    \centering
    \includegraphics[width=0.8\linewidth]{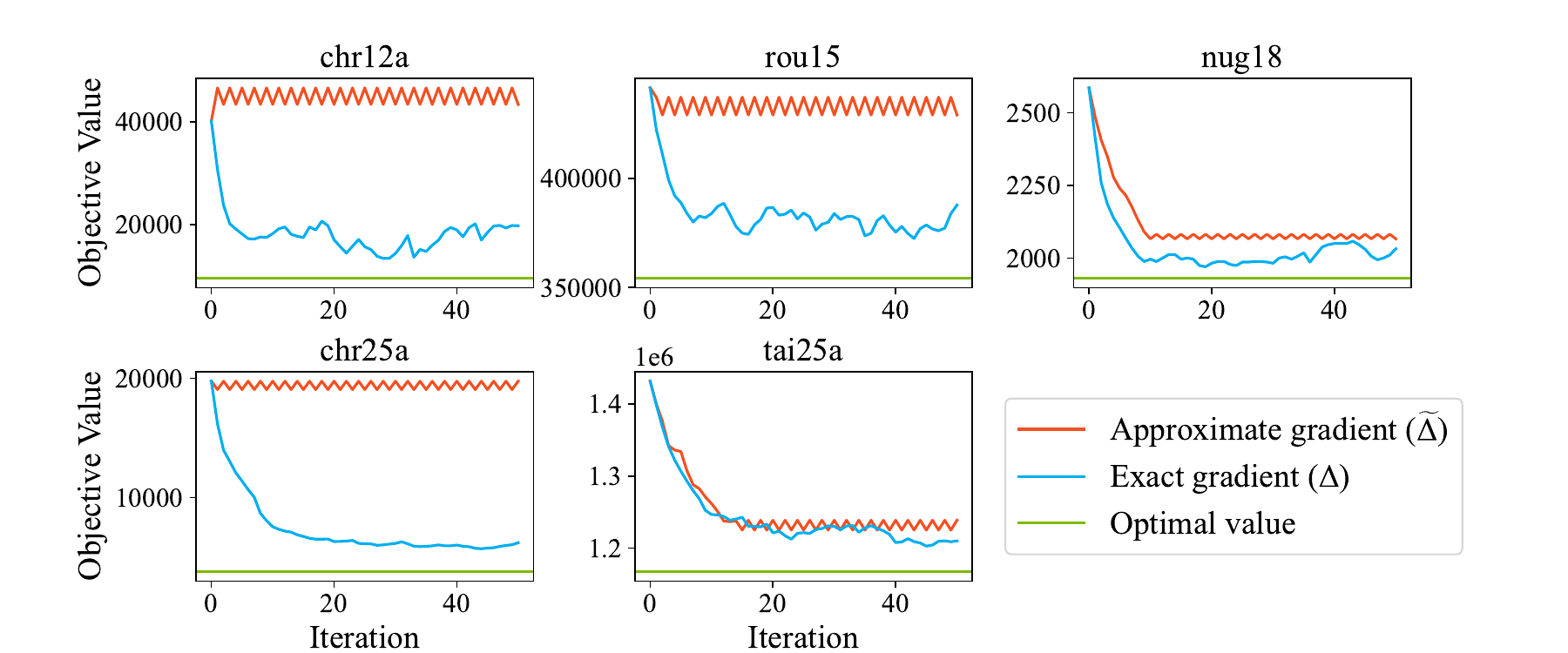}
    \caption{Objective value trajectories of the \texttt{Tabu} heuristic for five QAPLIB instances over 50 iterations of full-neighbourhood evaluations, beginning with the identity permutation $X=I$.}
    \label{fig:typical_alternating_tabu}
\end{figure}
Consequently, we do not consider the \texttt{Greedy} and \texttt{Tabu} heuristics  for further analysis.

\subsubsection{Analysis of the Top10 Heuristic}

We now describe the behaviour of the \texttt{Top10} heuristic.
\Cref{fig:rangeOptgap}(a) shows a comparison of the relative optimality gaps using the approximate gradient $\tilde\Delta$  and exact gradient  $\Delta$ observed over $10^5$ iterations, that is, 
$\left\{ {\relgap}^t(10^5) : t \in \mathcal{T} \right\}$.
\Cref{fig:rangeOptgap}(b) shows the best observed optimality gaps, $\relgapMin(10^4)$ and $\relgapMin(10^5)$, for both $\tilde\Delta$ and $\Delta$.
\begin{figure}[ht]
    \centering 
    \includegraphics[width=0.95\linewidth]{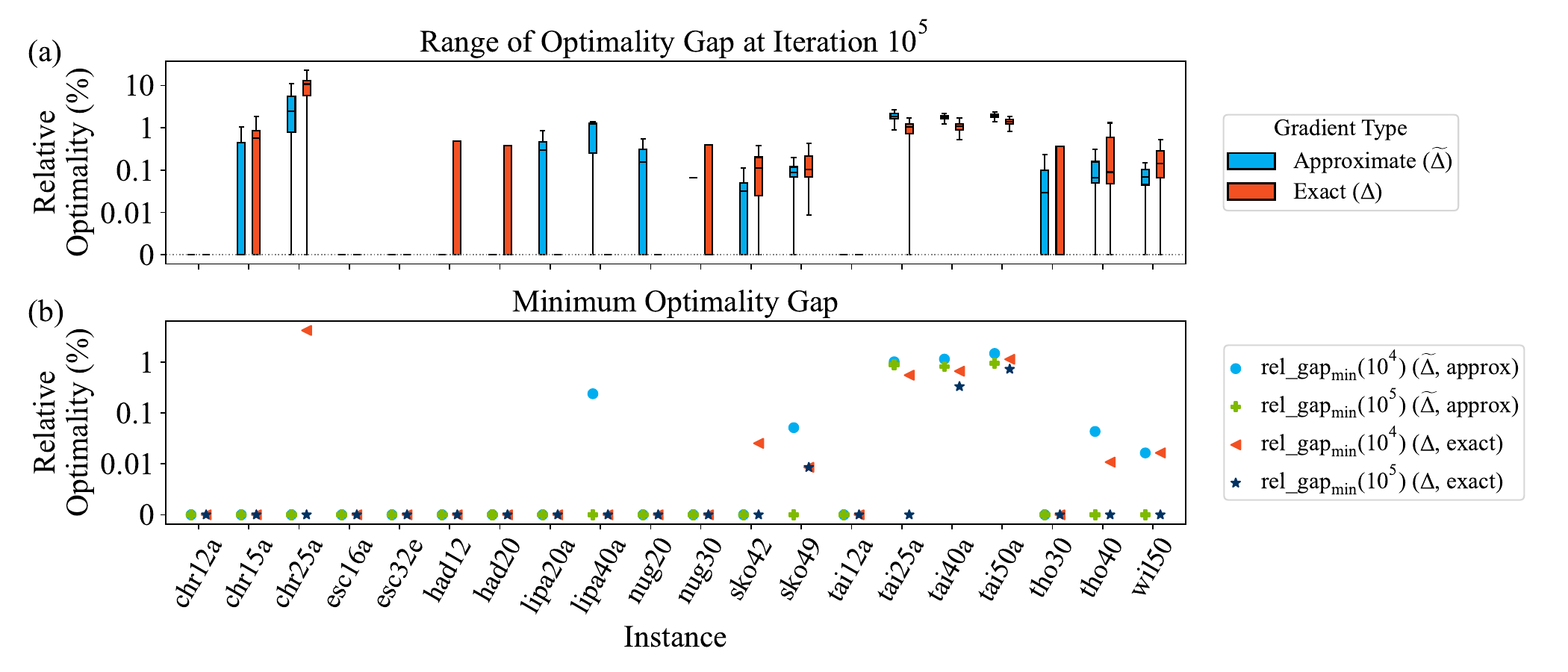}
    \caption{(a) Range of relative optimality gaps between approximate and exact gradient computations across $100$ trials with $10^5$ total iterations. (b) Best optimality gaps for approximate and exact gradient computations across $10^4$ and $10^5$ total iterations.}
    \label{fig:rangeOptgap}
\end{figure}
Both approaches yield optimality gaps that are smaller than $1\%$, demonstrating their effectiveness in reaching optimal or near-optimal solutions.
This observation supports the use of the approximate gradient as a viable alternative to the exact gradient.

\Cref{fig:which_one_wins} summarizes the number of times the approximate gradient method outperforms the exact gradient method, and vice versa, across 100 independent trials for each instance. 
\begin{figure}[ht]
    \centering
    \includegraphics[width=0.8\linewidth]{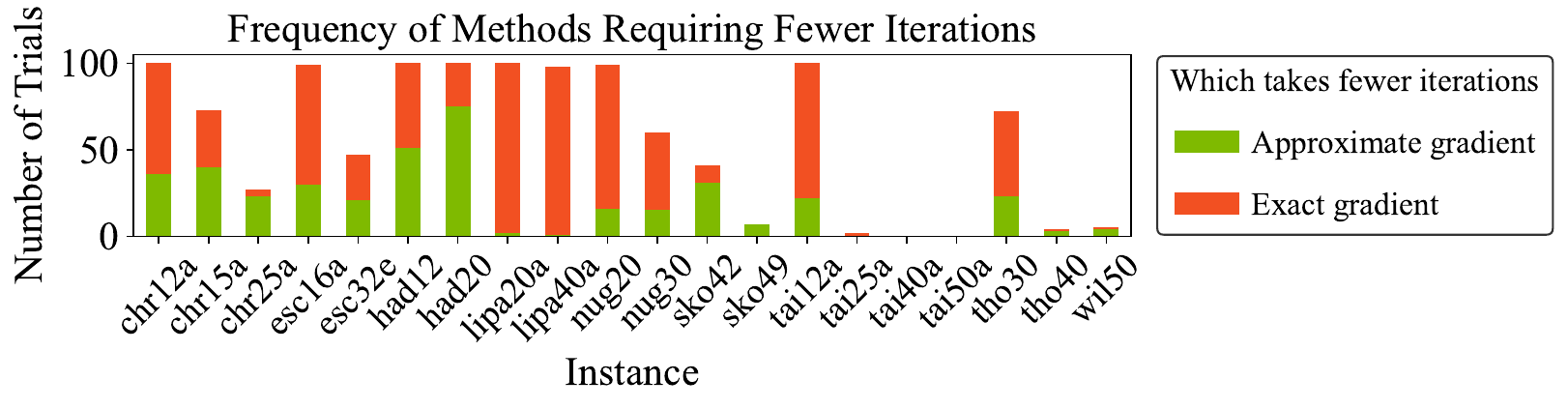}
    \caption{Proportion of 100 trials in which the approximate or exact gradient method reaches the optimal value in fewer iterations when the \texttt{Top10} heuristic is used.}
    \label{fig:which_one_wins}
\end{figure}
We record how often each method reaches the optimal value. The relative performance varies by instance, indicating that no single method is universally superior. The results highlight that using the approximate gradient can be  effective, and, in some cases, competitive with the exact method.

However, when comparing the relative runtime of computing the exact versus approximate gradient methods, the advantage of using the approximate gradient becomes evident. 
\Cref{fig:speedupApproxExact} shows the relative runtime, computed as the ratio of the time taken using 
$\Delta$ divided by the time taken using $\tilde\Delta$ for each instance, with the instances sorted by size.
The plot shows that using the exact gradient method requires significantly more computational resources and takes considerably longer than the approximate gradient method. 
Notably, the relative time increase becomes more pronounced as the instance size increases. 
\begin{figure}[ht]
    \centering
    \includegraphics[width=0.8\linewidth]{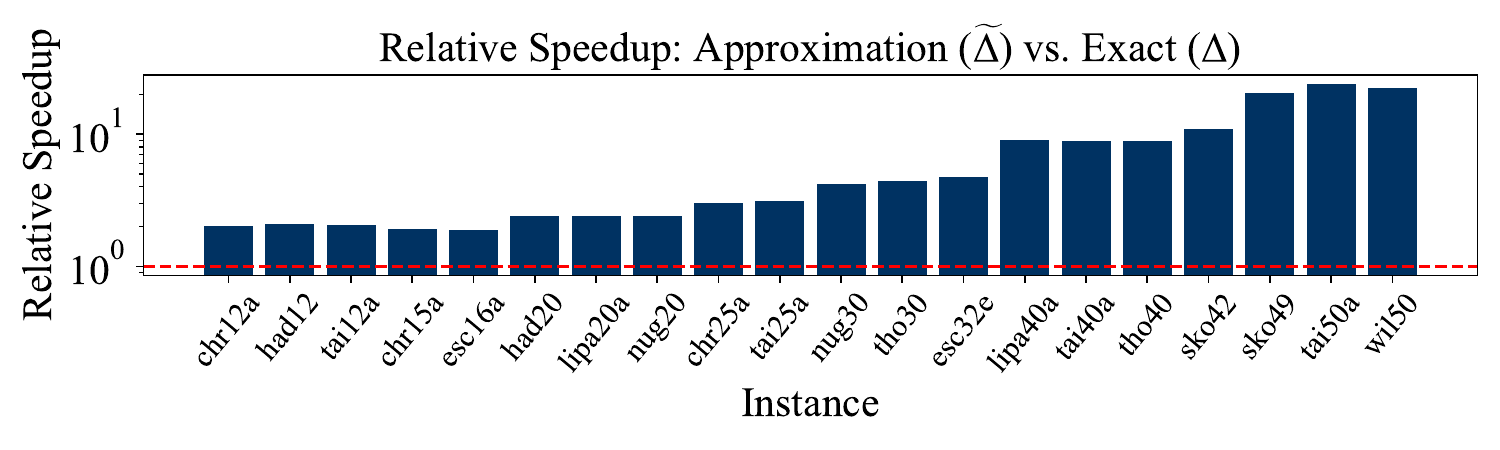} \vspace{-0.5cm}
    \caption{Ratio of the average runtime using the exact gradient method versus the approximate gradient method across $100$ trials.}
    \label{fig:speedupApproxExact}
\end{figure}

\subsubsection{Assessment of Omission of the Error Correction Terms}

We assess the impact of omitting the error terms $\mathcal{E}_i$ in the gradient computation as the problem size $n$ increases.
For each $i$-th neighbour in $\cN(x)$ at a given feasible binary solution $x$,
recall that each $\Delta_i$ involves $4n$ elements of $Q$, whereas $\mathcal{E}_i$ is determined by two elements of $Q$ regardless of the problem size.
To ensure a comprehensive evaluation, three classes from the QAPLIB that encompass a wide range of problem sizes are selected. 
The magnitude of the error terms at a given point is quantified using the following relative error metric.
Given a $i$-th neighbour $(z_1,z_2,z_3,z_4)$ in $\cN(x)$, we first decompose $\tilde\Delta=S_{34}g$ into the positive and negative components, that is, $\tilde\Delta = \tilde\Delta_i^+ - \tilde\Delta_i^-$, where 
$\tilde\Delta_i^+ = S_{34}((Qx)\circ (\mathbf{1}-x)$ and 
 $\tilde\Delta_i^- =  S_{34} \left( K  ( P_{\supp(x)} ((Qx )\circ x)) \right)$. 
We then measure the relative error:
\begin{equation}
    \label{eq:relErrorNum}
\text{rel\_error} :=
\frac{1}{\binom{n}{2}} \sum_{i=0}^{\binom{n}{2}-1} \frac{ \mathcal{E}_i  }{ \max\{ 1,\tilde\Delta_i^+  \} } .
\end{equation}
\Cref{fig:errorCorrectorNorm} shows rel\_error over the first $1000$ iterations for instances formulated by Skorin--Kapov (``sko''), Li and Pardalos (``lipa''), and Taillard (``tai''). 
We observe a consistent trend across all three classes, indicating that the contribution of $\mathcal{E}_i$ relative to $\Delta_i$ diminishes as the problem size increases, hence, the impact of omitting the error terms becomes negligible in large-scale instances. 

\begin{figure}[ht]
    \centering
    \includegraphics[width=1\linewidth]{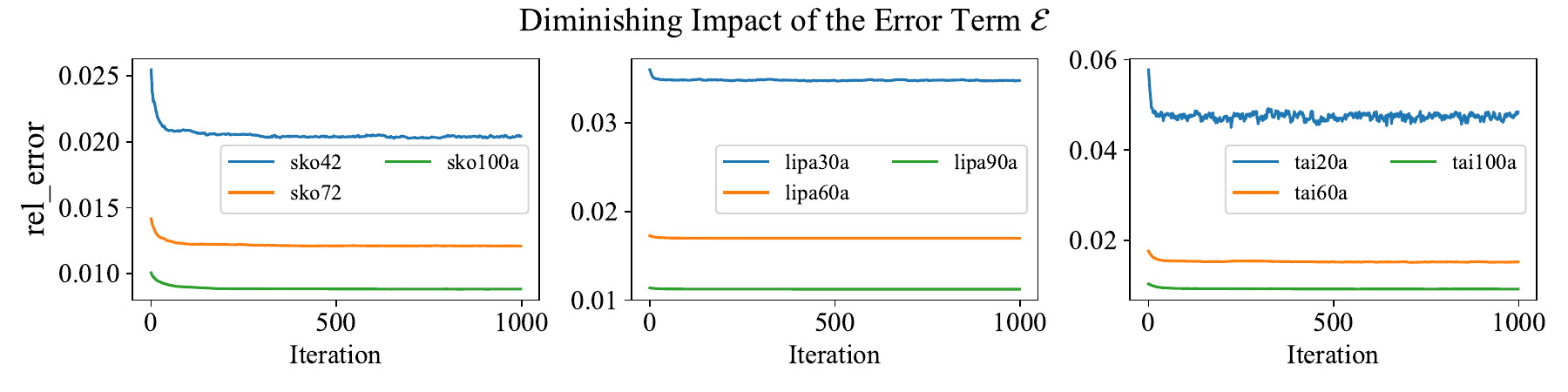}
     \vspace{-0.7cm}
    \caption{Relative errors, rel\_error, defined by the equality~\eqref{eq:relErrorNum},  observed from sko, lipa, and tai instance classes from the QAPLIB for the first $1000$ iterations of \Cref{algo:FN_All_routine}.}
    \label{fig:errorCorrectorNorm}
\end{figure}

\subsection{Performance Results for Large-Scale Instances Using Approximate Gradient Method}

We provide a comprehensive set of additional results where the approximate gradient is used for the \texttt{Top10}, \texttt{WalkQAP}, and 
\texttt{SA} heuristics on selected QAPLIB instances. The parameters described in \Cref{sec:Heuristics} are applied to these heuristics. 
An Intel Xeon E7-8890 v3 CPU running at 2.50 GHz with 144 cores was used for the experiment.
\Cref{fig:optgapThreesolvers} presents the relative optimality gaps $\relgapMin (10^5)$
for selected instances from the QAPLIB library.
The results demonstrate that the \texttt{Top10} and \texttt{WalkQAP} heuristics consistently outperform the \texttt{SA} heuristic, yielding significantly smaller gaps across most instances.
\begin{figure}[ht]
    \centering
    \includegraphics[width=0.99\linewidth]{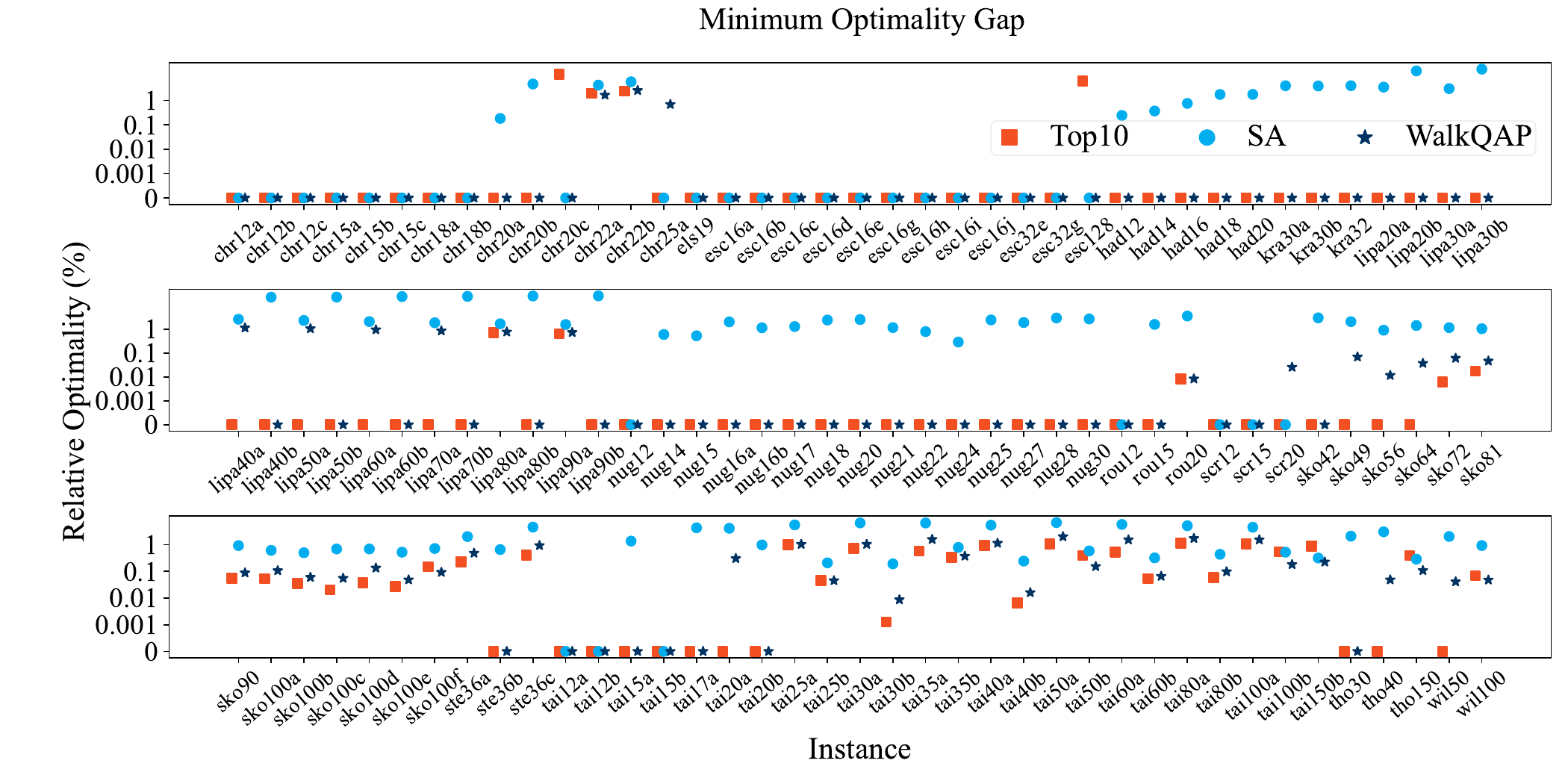} 
    \vspace{-0.3cm}
    \caption{Best optimality gaps observed across three heuristics: \texttt{Top10}; \texttt{SA}; and \texttt{WalkQAP}. The approximate gradient $\tilde{\Delta}$ is used for all heuristics. Optimality gaps that lie outside of the range [0$\%$, 5$\%$] have been omitted.}
    \label{fig:optgapThreesolvers}
\end{figure}

We now present the relative optimality gaps for instances not included in the QAPLIB library, which have a size of $n=27,45,75,125,175,$ and $343$.
These instances were kindly shared by one of the authors of Ref.~\cite{WANG2023109220}; as mentioned in Ref.~\cite{Mihic2018RDSolver},
these instances have been constructed such that they contain local optima that are relatively far apart and the difficulty increases with problem size.
We employ the approximate gradient method for each of the \texttt{Top30},  \texttt{WalkQAP}, and  \texttt{SA} heuristics.
The \texttt{Top30} heuristic corresponds to Eq.~\eqref{eq:top10Heu}, where
the members of a full neighbourhood are sorted and the best 30 are selected.  The \texttt{WalkQAP} heuristic corresponds to Eq.~\eqref{eq:walkSATHeu}, where the top 30 members are also selected, that is, \texttt{Top30} is used instead of \texttt{Top10}. The parameters described in \Cref{sec:Heuristics} are used for the \texttt{SA} heuristic. 
Relative gaps are computed using the definition \eqref {eq:relgap_formula} and the best known objective values, $f^*$, from Refs.~\cite{WANG2023109220, Mihic2018RDSolver}.
\Cref{tab:taillard_class_e} summarizes the relative optimality gaps observed across 100 trials obtained using these heuristics. An Intel Xeon 6767P CPU running at 3.90 GHz with 256 logical cores was used to generate the results.

\begin{table}[ht]
\centering
\begin{tabular}{lrrr}
\toprule
Instance & \texttt{Top30} & \texttt{WalkQAP} & \texttt{SA} \\
\midrule
tai27e01 & 0.08 & 7.74 & 0.00 \\
tai45e01 & 22.46 & 9.11 & 1.81 \\
tai75e01 & 26.37 & 9.43 & 2.79 \\
tai125e01 & 37.15 & 3.18 & 7.40 \\
tai175e01 & 25.77 & 23.04 & 9.64 \\
tai343e01 & 28.16 & 20.47 & 15.42 \\
\bottomrule
\end{tabular}
\caption{Minimum relative optimality gaps for instances from Refs.~\cite{WANG2023109220, Mihic2018RDSolver} observed over $10^5$ iterations using the \texttt{Top30}, \texttt{WalkQAP}, and \texttt{SA} heuristics with approximate gradients.}  
\label{tab:taillard_class_e}
\end{table}
We observe, from \Cref{tab:taillard_class_e}, that simple heuristics such as \texttt{Top30} and \texttt{WalkQAP}, that are well-suited for performing an intense local search, do not attain competitive objective values compared to the best objective values reported in Refs.~\cite{WANG2023109220, Mihic2018RDSolver}. 
\Cref{fig:tiallard_evoluation} illustrates the evolution of the best found objective values across 10 independent trials performed using the ``tai125e01'' instance. 
The \texttt{Top30} heuristic often does not effectively escape local minima, as it is unable to improve after a certain number of iterations. 
\begin{figure}[ht]
    \centering
    \includegraphics[width=0.95\linewidth]{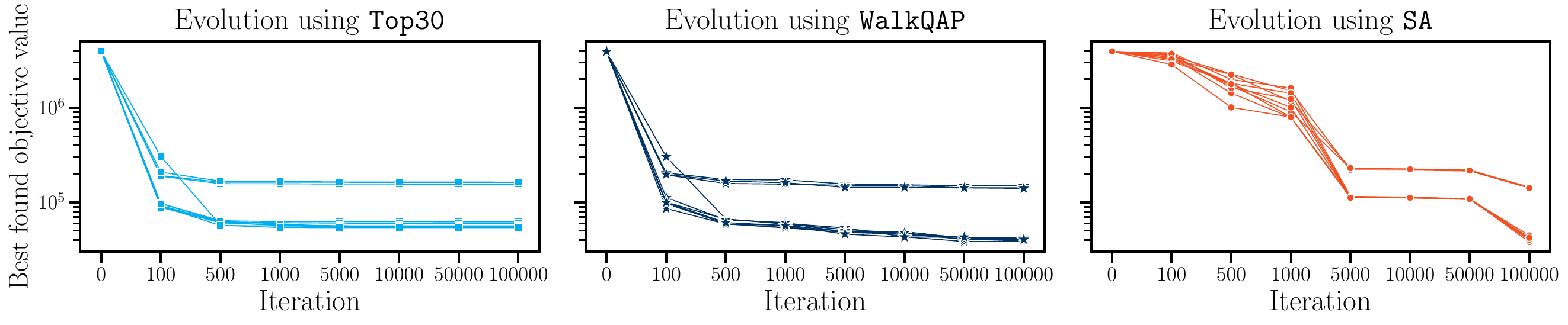}
    \caption{Evolution of the best found objective values for 10 trials of the ``tai125e01'' instance using the \texttt{Top30}, \texttt{WalkQAP}, and \texttt{SA} heuristics, respectively. For each trial, the best found objective values are recorded from a set of eight intervals; the markers at each coordinate $(i,f)$ represent the best found objective value~$f$ recorded from
    the intervals ranging from iteration 0 through iteration \mbox{$i \in \{0, 100, 500, 1000, 5000, $ 10,000, 50,000, 100,000$\}$,} where
$f=f_\text{best} (i)$ in definition
\eqref{eq:fbest_notation}.
}
    \label{fig:tiallard_evoluation}
\end{figure}
The \texttt{WalkQAP} heuristic performs better than the \texttt{Top30} heuristic due to its ability to probabilistically select a random member from a full neighourhood, thereby helping to escape local minima. The \texttt{SA} heuristic performs better than both the \texttt{Top30} and \texttt{WalkQAP} heuristics, which is consistent with its known effectiveness in escaping local minima. 
From these additional experiments, we attribute the observed performance gap between our heuristics  and state-of-the-art CMOS-based results to the intentionally simple design of our heuristics versus the more-sophisticated heuristics presented in Refs.~\cite{WANG2023109220, Mihic2018RDSolver}. Importantly, the main focus of the paper is our proposal of a hardware-friendly approach suitable for analog IMC hardware accelerators, which significantly increases the speed and efficiency of matrix--vector multiplications and exploits the accelerators' inherent parallelism. The enhanced computational speed and efficiency offered by IMC hardware is where our approach demonstrates clear advantages. We also emphasize that the proposed architecture is not restricted to a specific heuristic; when combined with more-sophisticated heuristics, the accelerated full-neighbourhood evaluation enabled by IMC accelerators is expected to further improve computational performance.

\subsection{Computation Breakdown of the Parallel Search and Its Applications to IMC Hardware}

In this section, we present a computational breakdown of our Python-based CPU implementation of \Cref{algo:FN_All_routine} to identify the computationally intensive routines by varying the problem size from $n=12$ to $n=100$.
Specifically, we compare the computation load distribution of the  subroutines in \Cref{algo:FN_All_routine} relative to the total runtime. 
The instance names are given in each pie chart in \Cref{fig:pieComputation}, followed by the average runtime for $10$ repetitions in terms of the wall-clock time.
The routines for computing $g$ and $\tilde{\Delta}$ (line~\ref{line:algoUpdateS} in \Cref{algo:FN_All_routine}) are labelled
 ``Gradient computation''. 
The update of the full neighbourhood (line~\ref{line:algoFNevalLine} in \Cref{algo:FN_All_routine}) is labelled ``Neighbourhood $\cN$ update''. 
The routine for implementing the \texttt{Top10} heuristic (line~\ref{line:pickHeuristic} in \Cref{algo:FN_All_routine}) is labelled ``Selection of neighbour''
and 
all other routines are collectively labelled ``Other''. The comparison in \Cref{fig:pieComputation} was generated using the Intel Xeon E7-8890 v3 running at 2.50 GHz with 144 cores. The Python profiler named ``cProfile'' was used to provide the performance breakdown.

\begin{figure}[ht]
    \centering
    \includegraphics[trim=0 0 0 0, clip, width=1\textwidth]{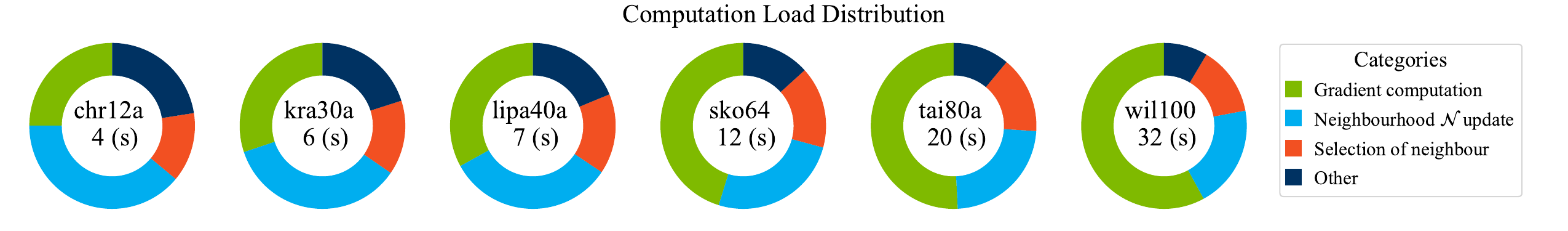}
    \caption{Computation load distribution on six QAPLIB instances. The runtimes over $10^4$ iterations appear below the instance names in the pie charts.}
    \label{fig:pieComputation}
\end{figure}

In \Cref{fig:pieComputation}, we observe a clear trend in which the fraction of the computational time required for calculating gradient increases as the problem size increases.
This increase in time is due primarily to the fact that these operations involve matrix--vector multiplications.
As the problem size increases, the matrix--vector multiplications become more demanding due to the quadratic growth of the binary formulation. 
It is worth noting, however, that the relative cost of the gradient computation can vary depending on the computational resources used.
Given that these steps are computational bottlenecks, IMC hardware presents a promising opportunity for considerable improvements in terms of computation time. 
Unlike conventional processors, which suffer from frequent memory access overhead, IMC architectures perform computations directly within memory.
As shown in \Cref{fig:gainComputationPic}, our approach is well-suited for this paradigm. 
Additionally, the energy efficiency of IMC hardware is a notable by-product, as it reduces the energy overhead of conventional digital computers. 

\section{Conclusions}
\label{sec:conclusion}

To sustain the computational performance growth rate in the post-Moore’s law era, new technologies such as IMC have emerged.  Mapping the native formulation of binary constrained optimization problems onto the unconstrained binary formulation incurs a massive overhead and remains the primary obstacle limiting the performance of special-purpose hardware. Hence, pairing special-purpose hardware with algorithms that solve optimization problems in their original, native formulation has the potential to surpass current state-of-the-art classical algorithms implemented on a CPU.
In this study, we designed an IMC-compatible co-designed framework to tackle the QAP, which is a well-known binary constrained optimization problem. The one-shot gradient evaluation framework for the QAP can be integrated with state-of-the-art local-search heuristic algorithms.  We emphasized the importance of preserving the native embedding of the problem, focusing exclusively on the native feasible space. A key strength of the proposed methodology is its ability to evaluate all gradients within a single iteration of the algorithm, specifically enabled by exploiting the parallelism afforded by IMC hardware. 
Moreover, we provided a blueprint for implementing application-driven analog hardware specifically designed for the full-neighbourhood local-search heuristic for the QAP. 
Another advantage of the approach is its flexibility, allowing it to seamlessly incorporate various heuristics available in the literature without requiring significant modifications. 
Our experimental results serve as a functional validation of the framework's logic on digital hardware, having demonstrated competitive performance and achieved close proximity to optimality, even with approximate gradient evaluation.
We showed that using the approximate gradient allows for a fast implementation, without sacrificing the quality of the results, as measured through the optimality gap. Furthermore, we highlighted the substantial speed-up potential offered by the use of IMC technology. Our results revealed an opportunity for acceleration when leveraging IMC technology, presenting a compelling case for its use across a wide range of industrially relevant problems.

\section*{Acknowledgements}
We thank our editor, Marko Bucyk, for his careful review and editing of the manuscript. This material is based upon work supported by the Defense Advanced Research Projects Agency (DARPA) through Air Force Research Laboratory Agreement No. FA8650-23-3-7313. The views, opinions, and/or findings expressed are those of the author(s) and should not be interpreted as representing the official views or policies of the Department of Defense or the U.S. Government.

\newpage

\appendix

\section{Supplementary Proof}
\begin{prop}
\label{prop:nativeDiff}
Let $x$ be a feasible solution in the binary space and 
let $y=x -e_{z_1}-e_{z_2}+e_{z_3}+e_{z_4}$ be a feasible solution  
formed by the quadruple bit-flip on $x$.
Then, the difference in the objective function value between $x$ and $y$ is 
\[
\begin{array}{ccl}
y^TQy - x^TQx &=& \left[ Q(z_3, : )  + Q(z_4, : )  \right] y + \left[ - Q(z_1, : ) - Q(z_2 , : )  \right] x \\
&&  + \left[ Q(:,z_3 )  + Q(:,z_4 ) \right]^T y + \left[ - Q(:,z_1 ) - Q(:,z_2 )  \right]^Tx \\
&& +~Q(z_1,:)\phi_x + Q(z_2,:) \phi_x - Q(z_3,:)\phi_y -Q(z_4,:)\phi_y .
\end{array}
\]
\end{prop}
\begin{proof}
Let $x$ be a feasible solution in the binary space, and let $y$ be a feasible solution after a local update, that is,  $y=x -e_{z_1}-e_{z_2}+e_{z_3}+e_{z_4}$.
Let $\phi_x$ and $\phi_y$ be as defined in the equalities \eqref{eq:diffphi_xy}.
We decompose the indices as follows:
\[
\cI := (\supp(x) \cup \supp(y) ) \setminus  \{z_1,z_2,z_3,z_4\}.
\]
Let $e_\cI= \sum_{i \in \cI} e_{i}$.
Then, 
\[
x =\phi_x  + e_{\cI} =e_{z_1}+e_{z_2} + e_{\cI}  \ \text{ and } \
y = \phi_y  + e_{\cI}=e_{z_3}+e_{z_4} + e_{\cI}.
\]
We then have
\[
\begin{array}{ccl}
y^TQy - x^TQx
&=& 
y^T Q \left( e_{z_3}+e_{z_4} + e_{\cI} \right) 
- x^T Q \left(  e_{z_1}+e_{z_2} + e_{\cI}  \right)
\\
&=& 
y^T Q e_{z_3}+ y^T Q e_{z_4} + y^T Q e_{\cI} 
- x^T Q e_{z_1} - x^TQ  e_{z_2} -x^TQ  e_{\cI} 
\\
&=& 
y^T Q e_{z_3}+ y^T Q e_{z_4} + e_{z_3}^T Q e_{\cI} + e_{z_4}^T Q e_{\cI}  + e_{\cI}^T Q e_{\cI} \\
&&
-~x^T Q e_{z_1} - x^TQ  e_{z_2} - e_{z_1}^TQ  e_{\cI} - e_{z_2}^TQ  e_{\cI} - e_{\cI}^TQ  e_{\cI} .
\end{array}
\]
After cancelling out the term $e_{\cI}^TQ  e_{\cI} $, we may write
$y^TQy - x^TQx  = D_y - D_x$, where
\[
\begin{array}{rcl}
   D_y  &:=&y^T Q e_{z_3}+ y^T Q e_{z_4} + e_{z_3}^T Q e_{\cI} + e_{z_4}^T Q e_{\cI} , \text{ and } \\
   D_x  &:=&  x^T Q e_{z_1} + x^TQ  e_{z_2} + e_{z_1}^TQ  e_{\cI} + e_{z_2}^TQ  e_{\cI}.
\end{array}
\]
Let 
\[
\begin{array}{rcl}
O_y &:=& 
e_{z_3}^T Q \phi_y- e_{z_3}^T Q \phi_y
+e_{z_4}^T Q\phi_y - e_{z_4}^T Q\phi_y , \text{ and }\\
O_x &:=& e_{z_1}^T Q \phi_x 
- e_{z_1}^T Q \phi_x 
+ e_{z_2}^T Q \phi_x
- e_{z_2}^T Q \phi_x .
\end{array}
\]
Note that $O_x=O_y=0$.
Expanding $D_y$ and $D_x$ using $O_y$ and $O_x$,  we obtain the expressions
\[
\begin{array}{rcl}
 D_y &=& D_y + O_y \\
     &=& y^T Q e_{z_3} + y^T Q e_{z_4}  + e_{z_3}^T Q y  + e_{z_4}^T   Q y 
- e_{z_3}^T Q\phi_y
- e_{z_4}^T Q\phi_y,
\end{array}
\]
and
\[
\begin{array}{rcl}
 D_x &=& D_x + O_x \\
     &=& x^T Q e_{z_1} +  x^T Q e_{z_2} + e_{z_1}^TQ x  + e_{z_2}^T  Q x
- e_{z_1}^T Q \phi_x 
- e_{z_2}^T Q \phi_x  .
\end{array}
\]
Thus, $y^TQy - x^TQx=D_y-D_x$ results in the equality in the statement.  
\end{proof}

\section{A Pathway to Hardware Implementation}
\label{sec:IMCimplementation}

While there are well-known closed-form analytical expressions for determining the change in the objective value due to performing local moves without needing to recalculate the full objective value~\cite{Taillard91}, the evaluation of a full neighbourhood of possible moves typically requires sequential evaluations of each neighbouring solution. In contrast, our method evaluates the entire neighbourhood of all possible 4-opt local moves simultaneously, performed through matrix--vector multiplications, which are a natural computational primitive in this algorithm.
Our approach is motivated by several studies that show significant gains in efficiency and speed in performing tasks involving matrix--vector multiplications using IMC hardware~\cite{10281389, 8811809, Bao2022}. 
\Cref{fig:IMCarchi} illustrates this connection where IMC hardware can support the full neighbourhood local search algorithm, \Cref{algo:FN_All_routine}, we propose.

\begin{figure}[htbp]
  \centering
  \begin{minipage}[t]{0.2\textwidth}
    \centering
    \includegraphics[height=3.5cm]{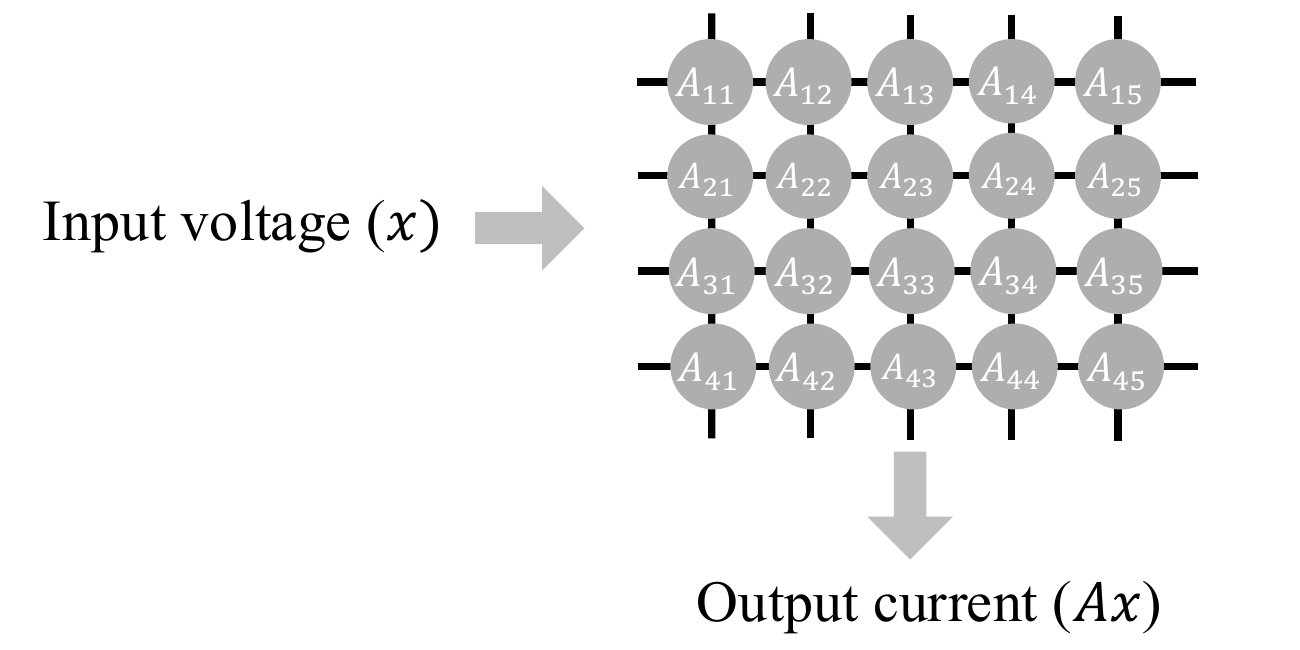}
    \caption*{\makebox[0.4\textwidth][l]{\hspace{1cm}(a) Analog crossbar array}}
\end{minipage}%
  \hfill
  \begin{minipage}[t]{0.6\textwidth}
    \centering
    \includegraphics[height=4cm]{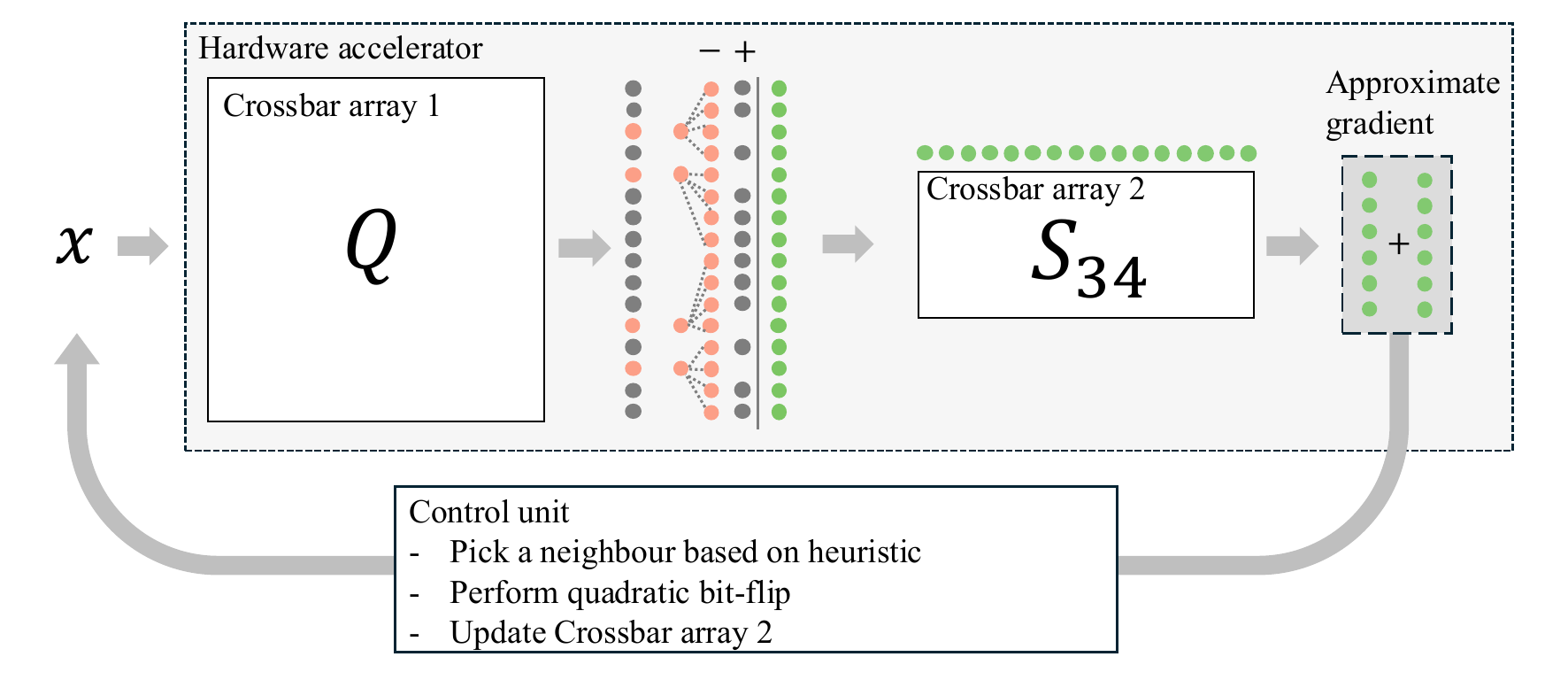}
    \caption*{(b) IMC architecture}
  \end{minipage}

  \caption{(a) Diagram of an IMC crossbar array, where each memristor’s conductance encodes an element of a matrix used for  matrix--vector multiplications. (b) High-level overview of the circuit implementation of \Cref{algo:FN_All_routine}.}
  \label{fig:IMCarchi}
\end{figure}

Our approach relies on two key hardware components that both involve matrix--vector multiplications: one IMC crossbar array used for summing the rows (or columns) of a matrix; and a second IMC crossbar array used for grouping the resulting vector in order to perform an approximate gradient computation. Performing a full evaluation of the QAP objective function using a matrix--vector multiplication requires $O(n^4)$ floating-point operations, where $n$ is the QAP problem size; however, by leveraging the capability of IMC hardware to perform operations in parallel, matrix--vector multiplication can be executed with a $O(1)$ time complexity, and with great energy efficiency, provided that a problem is able to fit within the capacity of hardware accelerators.
However, realizing an architecture that uses such hardware accelerators presents challenges in terms of supporting diverse functionalities and complex algorithms. Below, we list them, along with potential mitigation methods.

Currently, state-of-the-art single-tile crossbar arrays are small, and thus  a large-scale problem is unable to fit in a single array. The QUBO formulation of the QAP problem requires matrices $n^2$-by-$n^2$ in size, necessitating either crossbars with large arrays or a distributed architecture in order to be able to handle large-scale instances. Crossbar arrays up to 128-by-1024 in size~\cite{10067380} have been reported within the SAT community, and ones up to 512-by-2048 in size have been reported elsewhere~\cite{AmbrogioS_2023}, neither of which are sufficient for programming large-scale QAP instances (e.g., where $n>50$).   Several studies have shown how to use an architecture of relatively smaller crossbar arrays to facilitate larger matrix--vector multiplications~\cite{10.1145/3007787.3001139, 10.1145/3665314.3670851}. Another challenge is that analog IMC devices are more limited in bit  precision than CMOS-based devices. Bit precision in excess of 11 bits has been demonstrated~\cite{rao2023thousands}; however, as noted in Ref.~\cite{Bao2022}, the bit precision of IMC hardware accelerators is often closer to 6 bits. In the QAPLIB benchmarking set of instances, we observe that the dynamic ranges of the flow and distance matrices vary widely.
For example, for the ``tai12a'' and ``tai100b'' instances, the pairs representing the largest elements of matrices $D$ and $F$ are $(96,99)$ and $(1834, 8095)$, respectively. Taking the Kronecker product of these two matrices further amplifies the numeric range, which could  require employing quantization techniques~\cite{ChoiJungwook2018PPCA}, similar to those currently being used for large language models.

We envision two opportunities for integrating IMC crossbar arrays into our framework. The binary matrix $S_{34}$ from \Cref{fig:IMCarchi} is in charge of aggregating the output of the first crossbar array. Since the solution is updated at every iteration, the corresponding neighbourhood representation---and thus $S_{34}$---must also change at each iteration, which may introduce reprogramming overhead. Although $S_{34}$ is $n(n-1)/2$-by-$n^2$ in size, only $2n-3$ elements need to be modified per iteration. Such frequent reprogramming can incur costs in terms of both time and energy consumption. If the reprogramming is performed sequentially (i.e., row by row), it would require a time complexity of $O(n)$. However, efficient reprogramming techniques remain an active area of research. For example, faster and more-energy-efficient memory device technologies---such as the capacitive (DRAM-like) memories demonstrated in Ref.~\cite{LerouxNathan2025Aica}---may help mitigate this overhead.

A key feature of our algorithm is the ability to perform a one-shot update that provides a full-neighbourhood evaluation, requiring only an $O(n)$ update over the $O(n^2)$ total possible moves. Selecting and updating the $2n-3$ elements are non-trivial operations, and the most straightforward implementation may require guidance using a coprocessor. This would involve an analog-to-digital (ADC) readout, followed by a digital-to-analog  (DAC) conversion; these two steps would incur both energy and latency overheads, but we believe that the savings from the $O(n^2)$ to $O(n)$ reduction would more than compensate for the potential ADC--DAC overhead.

The noise and inaccuracies that occur during analog operations have been well-characterized~\cite{10281389, 8811809, Bao2022, Sebastian2020Memory}. For example, unstable conductance variation in memristive devices makes it difficult to store the exact resistance or charge values needed for calculations. The current--voltage relationship is linear in theory; however, in practice, non-idealities introduce errors. Another exemplary challenge is the current--resistance drop, which may be caused by a large current flowing through an array, causing a drop in the voltage; this can be exacerbated as crossbar arrays becomes larger. Noise arises from device- and circuit-level non-idealities, but it itself is not necessarily a limitation of analog computation. Recent studies demonstrate that inaccuracies that arise during analog operations can be systematically managed through error-correcting codes specifically designed for IMC systems, and even harnessed as randomness generators for use in stochastic heuristics~\cite{cai2020power}. For example, analog error-correcting schemes have been developed to detect and correct computational errors in crossbar arrays~\cite{Roth_2020_ECC, LerouxNathan2025Aica}, demonstrating that noise need not be a fundamental barrier but rather a challenge that can be mitigated through appropriate techniques, and can even enhance the effectiveness of local minima escape mechanisms~\cite{cai2020power, ding2025}.

\section{Symbols and Their Use}
\label{sec:SymbolExample}

In this section, we summarize the notation introduced in \Cref{sec:notation} and provide examples of its use.

\begin{enumerate}
\item Sets and Spaces

\begin{table}[ht]
\centering
\begin{tabular}{l l}
    \toprule
    \textbf{Notation} & \textbf{Meaning} \\
    \midrule
    $\mathbb{R}^{n}$ & The set of all real-valued column vectors with a length of $n$. \\
    $\mathbb{R}^{m\times n}$  & The set of all real-valued matrices with $m$ rows and $n$ columns. \\
    $\mathbb{S}^n$  & The set of all $n \times n$ symmetric matrices (i.e., matrices $X$ where $X = X^T$). \\
    $\Pi_n$  & The set of all permutations with a length of $n$. \\
    $P_n$  & The set of all $n \times n$ permutation matrices. \\
    \bottomrule
\end{tabular}
\end{table}

We provide examples of $P_n$ and $\Pi_n$ for the case of 
$n=3$.  The sets $P_3$ and $\Pi_3$ both have $6$ members, and are characterized as follows:
\[
\begin{small}
\text{\begin{normalsize}$P_3$\end{normalsize}}=\left\{
P^1 =
\begin{bmatrix}
1 & 0 & 0\\
0 & 1 & 0\\
0 & 0 & 1
\end{bmatrix},
P^2 =
\begin{bmatrix}
1 & 0 & 0\\
0 & 0 & 1\\
0 & 1 & 0
\end{bmatrix},
P^3 =
\begin{bmatrix}
0 & 1 & 0\\
1 & 0 & 0\\
0 & 0 & 1
\end{bmatrix},
P^4 =
\begin{bmatrix}
0 & 1 & 0\\
0 & 0 & 1\\
1 & 0 & 0
\end{bmatrix},
P^5 =
\begin{bmatrix}
0 & 0 & 1\\
1 & 0 & 0\\
0 & 1 & 0
\end{bmatrix},
P^6 =
\begin{bmatrix}
0 & 0 & 1\\
0 & 1 & 0\\
1 & 0 & 0
\end{bmatrix}
\right\},
\end{small}
\]
and
\[
\Pi_3 = \{p^1 = (1,2,3), \ p^2 = (1,3,2), \ p^3 = (2,1,3),  \ p^4 = (2,3,1), \ p^5 = (3,1,2), \ p^6 = (3,2,1)\}.
\]
Note that, for each $i=1,2,\ldots, 6$, $P^i$ is obtained by ordering the columns of the identity matrix $I_3$ with respect to the elements in $p^i$.

\item Matrix and Vector Elements

\begin{table}[ht]
\centering
\begin{tabular}{l l}
    \toprule
    \textbf{Notation} & \textbf{Meaning} \\
    \midrule
    $A(i,j), A_{i,j}$& The element located at the $i$-th row and $j$-th column of matrix $A$. \\
    $A(i,:)$ & The $i$-th row vector of matrix $A$. \\
    $A(:,i)$ & The $i$-th column vector of matrix $A$. \\
    $x_i$ & The $i$-th component of vector $x$. \\
    $\supp(x)$ & The support of vector $x$, i.e., the set of indices $i$ for which $x_i \ne 0$. \\
    \bottomrule
\end{tabular}
\end{table}

Let $A=\begin{bmatrix}
    1& 2 \\ 3&4\\ 5&6 \end{bmatrix}$, $i=1$, and $j=0$. Following the zero-indexing rule, 
we have $A(i,j) = 3$, $A(i,:)=\begin{pmatrix} 3&4\end{pmatrix}$,  and $A(:,i)=\begin{pmatrix} 2\\4\\6 \end{pmatrix}$.  
Given a vector $x=\begin{pmatrix}
    1 & 0 & 2 & -1 & 0 \end{pmatrix}^T$,  the support of $x$ is the set $\{0,2,3\}$.

\item Matrices and Vectors with Predefined Elements

\begin{table}[ht]
\centering
\begin{tabular}{l l}
    \toprule
    \textbf{Notation} & \textbf{Meaning} \\
    \midrule
    $I_n$  & The $n \times n$ identity matrix. \\
    $e_i$  & The standard unit vector in an arbitrary dimension, where the $i$-th component is equal to $1$. \\
    $\mathbf{1}_n$  & The vector with a length $n$, where all elements are $1$. \\
    \bottomrule
\end{tabular}
\end{table}

Let $n=3$. 
Then $I_3 = \begin{bmatrix}
1 & 0 & 0\\
0 & 1 & 0\\
0 & 0 & 1
\end{bmatrix}$, 
$\mathbf{1}_n = \begin{pmatrix} 1\\1\\1  \end{pmatrix}$. 
If $e_i \in \mathbb{R}^3$, then 
$e_0 = \begin{pmatrix} 1\\0\\0 \end{pmatrix}$ and $e_1 = \begin{pmatrix} 0\\1\\0 \end{pmatrix}$.

\item Operations

    \begin{table}[H]
    \centering
\begin{tabular}{l l l}
    \toprule
    \textbf{Notation} & \textbf{Operation} & \textbf{Meaning} \\
    \midrule
    $A^T$  & Transpose & The transpose of matrix $A$. If $A \in \mathbb{R}^{m\times n}$, then $A^T \in \mathbb{R}^{n\times m}$. \\
    $\kvec(X)$ & Vectorization & The vector formed by sequentially stacking the columns of matrix $X$. \\
    $\langle \cdot , \cdot \rangle$ & Inner product & The standard inner product between two vectors. \\
    $\otimes$ & Kronecker product & The block matrix product of two matrices. \\ &&The $(i,j)$-th block matrix of $(A\otimes B)$ is $A_{i,j}B$. \\
    $\circ$ & Hadamard product & The element-wise product of two vectors or matrices of the same dimension. \\
    \bottomrule
\end{tabular}
\end{table}

To illustrate the use of these operators, let $A=\begin{bmatrix}
    1&2 \\3&4 \end{bmatrix}$ and $B=\begin{bmatrix}
        -2 &1 \\ 0 & 3 \end{bmatrix}$.
Then, $\kvec(A)=\begin{pmatrix}     1\\3\\2\\4  \end{pmatrix}$ and $\kvec(B)=\begin{pmatrix}     -2 \\0 \\1\\3 \end{pmatrix}$.
The standard inner product  $\langle \kvec(A), \kvec(B) \rangle$ is equal to 12.
The Kronecker product $A\otimes B$ and the Hadamard product $A\circ B $ are 
\[ 
A\otimes B = \begin{bmatrix}
1 \cdot B & 2 \cdot B \\
3 \cdot B & 4 \cdot B
\end{bmatrix}= \begin{bmatrix}
-2 & 1 & -4 & 2 \\
0 & 3 & 0 & 6 \\
-6 & 3 & -8 & 4 \\
0 & 9 & 0 & 12
\end{bmatrix}, \text{ and }
A\circ B = \begin{bmatrix}
1\cdot(-2) & 2\cdot 1 \\
3\cdot0 & 4\cdot3
\end{bmatrix} = \begin{bmatrix}
-2 & 2 \\
0 & 12
\end{bmatrix}.
\]

\end{enumerate}

\newpage
\bibliographystyle{siam}
\bibliography{refs}

\begin{thebibliography}{10}

\bibitem{AmbrogioS_2023}
{\em An analog-ai chip for energy-efficient speech recognition and
  transcription}, Nature, 620 (2023), pp.~768--775.

\bibitem{Aramon19}
{\sc M.~Aramon, G.~Rosenberg, E.~Valiante, T.~Miyazawa, H.~Tamura, and H.~G.
  Katzgraber}, {\em Physics-inspired optimization for quadratic unconstrained
  problems using a digital annealer}, Frontiers in Physics, 7 (2019), p.~48.

\bibitem{MAQOBagherbeik_2022}
{\sc M.~Bagherbeik, W.~Xu, S.~F. Mousavi, K.~Kanda, H.~Tamura, and
  A.~Sheikholeslami}, {\em Maqo: A scalable many-core annealer for quadratic
  optimization}, in 2022 IEEE Symposium on VLSI Technology and Circuits (VLSI
  Technology and Circuits), 2022, pp.~76--77.

\bibitem{Bao2022}
{\sc H.~Bao, H.~Zhou, J.~Li, H.~Pei, J.~Tian, L.~Yang, S.~Ren, S.~Tong, Y.~Li,
  Y.~He, J.~Chen, Y.~Cai, H.~Wu, Q.~Liu, Q.~Wan, and X.~Miao}, {\em Toward
  memristive in‑memory computing: principles and applications}, Frontiers of
  Optoelectronics, 15 (2022), p.~23.

\bibitem{BashiriMahdi2012Eham}
{\sc M.~Bashiri and H.~Karimi}, {\em Effective heuristics and meta-heuristics
  for the quadratic assignment problem with tuned parameters and analytical
  comparisons}, Journal of industrial engineering international, 8 (2012),
  pp.~1--9.

\bibitem{BenjaafarSaifallah2002MaAo}
{\sc S.~Benjaafar}, {\em Modeling and analysis of congestion in the design of
  facility layouts}, Management science, 48 (2002), pp.~679--704.

\bibitem{bhattacharya2024computing}
{\sc T.~Bhattacharya, G.~H. Hutchinson, G.~Pedretti, X.~Sheng, J.~Ignowski,
  T.~Van~Vaerenbergh, R.~Beausoleil, J.~P. Strachan, and D.~B. Strukov}, {\em
  Computing high-degree polynomial gradients in memory}, Nature Communications,
  15 (2024), p.~8211.

\bibitem{Bouras2014}
{\sc A.~Bouras, M.~A. Ghaleb, U.~S. Suryahatmaja, and A.~M. Salem}, {\em The
  airport gate assignment problem: a survey.}, TheScientificWorldJournal, 2014
  (2014), p.~923859.

\bibitem{QAPLIBurl}
{\sc R.~Burkard, S.~Karisch, and F.~Rendl}, {\em Qaplib- a quadratic assignment
  problem library}, Journal of global optimization, 10 (1997), pp.~391--403.
\newblock https://coral.ise.lehigh.edu/data-sets/qaplib/.

\bibitem{burkard1998quadratic}
{\sc R.~E. Burkard, E.~Cela, P.~M. Pardalos, and L.~S. Pitsoulis}, {\em The
  quadratic assignment problem}, Springer, 1998.

\bibitem{cai2020power}
{\sc F.~Cai, S.~Kumar, T.~Van~Vaerenbergh, X.~Sheng, R.~Liu, C.~Li, Z.~Liu,
  M.~Foltin, S.~Yu, Q.~Xia, J.~J. Yang, R.~Beausoleil, W.~D. Lu, and J.~P.
  Strachan}, {\em Power-efficient combinatorial optimization using intrinsic
  noise in memristor hopfield neural networks}, Nature Electronics, 3 (2020),
  pp.~409--418.

\bibitem{Cela98}
{\sc E.~{\c C}ela}, {\em The Quadratic Assignment Problem: {T}heory and
  Algorithms}, Kluwer Academic, Dordrecht, The Netherlands, 1998.

\bibitem{CHIANG1998457}
{\sc W.-C. Chiang and C.~Chiang}, {\em Intelligent local search strategies for
  solving facility layout problems with the quadratic assignment problem
  formulation}, European Journal of Operational Research, 106 (1998),
  pp.~457--488.

\bibitem{chiew2025optimalfermionqubitmappingsquadratic}
{\sc M.~Chiew, C.~Ibrahim, I.~Safro, and S.~Strelchuk}, {\em Optimal
  fermion-qubit mappings via quadratic assignment}, 2025.

\bibitem{ChoiJungwook2018PPCA}
{\sc J.~Choi, Z.~Wang, S.~Venkataramani, P.~I.-J. Chuang, V.~Srinivasan, and
  K.~Gopalakrishnan}, {\em Pact: Parameterized clipping activation for
  quantized neural networks},  (2018).

\bibitem{Commander2005QAPThesis}
{\sc C.~W. Commander}, {\em A Survey of the Quadratic Assignment Problem, with
  Applications}, PhD thesis, University of Florida, 2005.

\bibitem{CUBUKCUOGLU2021102952}
{\sc C.~Cubukcuoglu, P.~Nourian, M.~F. Tasgetiren, I.~S. Sariyildiz, and
  S.~Azadi}, {\em Hospital layout design renovation as a quadratic assignment
  problem with geodesic distances}, Journal of Building Engineering, 44 (2021),
  p.~102952.

\bibitem{carvalho2006microarray}
{\sc S.~A. {de} Carvalho~Jr. and S.~Rahmann}, {\em Microarray layout as
  quadratic assignment problem}, in German Conference on Bioinformatics,
  Gesellschaft f{\"u}r Informatik eV, 2006.

\bibitem{Dickey_1972Cbau}
{\sc J.~Dickey and J.~Hopkins}, {\em Campus building arrangement using topaz},
  Transportation research, 6 (1972), pp.~59--68.

\bibitem{ding2025}
{\sc C.~Ding, Y.~Ren, Z.~Liu, and N.~Wong}, {\em Transforming memristor noises
  into computational innovations}, Communications Materials, 6 (2025), p.~149.

\bibitem{Dobrynin:2024iyi}
{\sc D.~Dobrynin, A.~Renaudineau, M.~Hizzani, D.~Strukov, M.~Mohseni, and J.~P.
  Strachan}, {\em {Energy landscapes of combinatorial optimization in Ising
  machines}}, Phys. Rev. E, 110 (2024), p.~045308.

\bibitem{DreznerZvi2003ANGA}
{\sc Z.~Drezner}, {\em A new genetic algorithm for the quadratic assignment
  problem}, INFORMS journal on computing, 15 (2003), pp.~320--330.

\bibitem{dury2020quboformulationqubitallocation}
{\sc B.~Dury and O.~D. Matteo}, {\em A qubo formulation for qubit allocation},
  2020.

\bibitem{Elshafei_1977_hospital}
{\sc A.~N. Elshafei}, {\em Hospital layout as a quadratic assignment problem},
  Operational Research Quarterly (1970-1977), 28 (1977), pp.~167--179.

\bibitem{EMANUEL2012525}
{\sc B.~Emanuel, S.~Wimer, and G.~Wolansky}, {\em Using well-solvable quadratic
  assignment problems for vlsi interconnect applications}, Discrete Applied
  Mathematics, 160 (2012), pp.~525--535.

\bibitem{fahimi2021combinatorial}
{\sc Z.~Fahimi, M.~Mahmoodi, H.~Nili, V.~Polishchuk, and D.~Strukov}, {\em
  Combinatorial optimization by weight annealing in memristive hopfield
  networks}, Scientific Reports, 11 (2021), p.~16383.

\bibitem{Silva2022Optimization}
{\sc A.~Fernandes Da Costa~Silva}, {\em Optimization of storage and picking
  systems in warehouses}, doctoral dissertation, Université Laval, 2022.

\bibitem{finke1987quadratic}
{\sc G.~Finke, R.~E. Burkard, and F.~Rendl}, {\em Quadratic assignment
  problems}, in North-Holland Mathematics Studies, vol.~132, Elsevier, 1987,
  pp.~61--82.

\bibitem{FU_KAKU_1997Mwam}
{\sc M.~C. FU and B.~K. KAKU}, {\em Minimizing work-in-process and material
  handling in the facilities layout problem}, IIE transactions, 29 (1997),
  pp.~29--36.

\bibitem{AntColoniesQAP99}
{\sc L.~M. Gambardella, {\'E. D}.~Taillard, and M.~Dorigo}, {\em Ant colonies
  for the quadratic assignment problem}, The Journal of the Operational
  Research Society, 50 (1999), pp.~167--176.

\bibitem{0b2b06e1-2775-327a-843f-af24a0afab1a}
{\sc A.~M. Geoffrion and G.~W. Graves}, {\em Scheduling parallel production
  lines with changeover costs: Practical application of a quadratic
  assignment/lp approach}, Operations Research, 24 (1976), pp.~595--610.

\bibitem{Glover2022}
{\sc F.~Glover, G.~Kochenberger, R.~Hennig, and Y.~Du}, {\em Quantum bridge
  analytics {I}: a tutorial on formulating and using {QUBO} models}, Annals of
  Operations Research, 314 (2022), pp.~141--183.

\bibitem{Goto19}
{\sc H.~Goto, K.~Tatsumura, and A.~R. Dixon}, {\em Combinatorial optimization
  by simulating adiabatic bifurcations in nonlinear {H}amiltonian systems},
  Science Advances, 5 (2019).

\bibitem{Hamerly18}
{\sc R.~Hamerly, T.~Inagaki, P.~L. McMahon, D.~Venturelli, A.~Marandi,
  T.~Onodera, E.~Ng, C.~Langrock, K.~Inaba, T.~Honjo, K.~Enbutsu, T.~Umeki,
  R.~Kasahara, S.~Utsunomiya, S.~Kako, K.~I. Kawarabayashi, R.~L. Byer, M.~M.
  Fejer, H.~Mabuchi, D.~Englund, E.~Rieffel, H.~Takesue, and Y.~Yamamoto}, {\em
  Experimental investigation of performance differences between coherent
  {I}sing machines and a quantum annealer}, Science Advances, 5 (2019).

\bibitem{HausmanWarrenH1976OSAi}
{\sc W.~H. Hausman, L.~B. Schwarz, and S.~C. Graves}, {\em Optimal storage
  assignment in automatic warehousing systems}, Management science, 22 (1976),
  pp.~629--638.

\bibitem{heffley1977QAPrunner}
{\sc D.~Heffley}, {\em Assigning runners to a relay team}, in Optimal
  Strategies in Sports, S.~P. Ladany and R.~E. Machol, eds., North-Holland,
  1977.

\bibitem{10.1145/3665314.3670851}
{\sc G.~H. Hutchinson, E.~Sifferman, T.~Bhattacharya, D.~Kwon, and D.~B.
  Strukov}, {\em Fpia: Field-programmable ising arrays with in-memory
  computing}, in Proceedings of the 29th ACM/IEEE International Symposium on
  Low Power Electronics and Design, ISLPED '24, New York, NY, USA, 2024,
  Association for Computing Machinery, p.~1–6.

\bibitem{Johnson11}
{\sc M.~Johnson, M.~Amin, S.~Gildert, T.~Lanting, F.~Hamze, N.~Dickson,
  R.~Harris, A.~Berkley, J.~Johansson, P.~Bunyk, E.~Chapple, C.~Enderud,
  J.~Hilton, K.~Karimi, E.~Ladizinsky, N.~Ladizinsky, T.~Oh, I.~Perminov,
  C.~Rich, and G.~Rose}, {\em Quantum annealing with manufactured spins},
  Nature, 473 (2011), pp.~194--8.

\bibitem{Koopmans57}
{\sc T.~C. Koopmans and M.~Beckmann}, {\em Assignment problems and the location
  of economic activities}, Econometrica, 25 (1957), pp.~53--76.

\bibitem{Krarup1978}
{\sc J.~Krarup and P.~M. Pruzan}, {\em Computer-aided layout design}, Springer
  Berlin Heidelberg, Berlin, Heidelberg, 1978, pp.~75--94.

\bibitem{LerouxNathan2025Aica}
{\sc N.~Leroux, P.-P. Manea, C.~Sudarshan, J.~Finkbeiner, S.~Siegel, J.~P.
  Strachan, and E.~Neftci}, {\em Analog in-memory computing attention mechanism
  for fast and energy-efficient large language models}, Nature Computational
  Science, 5 (2025), pp.~813--824.

\bibitem{li2016dynamic}
{\sc J.~Li, M.~Moghaddam, and S.~Y. Nof}, {\em Dynamic storage assignment with
  product affinity and abc classification—a case study}, The International
  Journal of Advanced Manufacturing Technology, 84 (2016), pp.~2179--2194.

\bibitem{LOIOLA2007657}
{\sc E.~M. Loiola, N.~M.~M. {de Abreu}, P.~O. Boaventura-Netto, P.~Hahn, and
  T.~Querido}, {\em A survey for the quadratic assignment problem}, European
  Journal of Operational Research, 176 (2007), pp.~657--690.

\bibitem{scaleIMC_TSP_Lu_2023}
{\sc A.~Lu, J.~Hur, Y.-C. Luo, H.~Li, D.~E. Nikonov, I.~A. Young, Y.-K. Choi,
  and S.~Yu}, {\em Scalable in-memory clustered annealer with temporal noise of
  charge trap transistor for large scale travelling salesman problems}, IEEE
  Journal on Emerging and Selected Topics in Circuits and Systems, 13 (2023),
  pp.~422--435.

\bibitem{10.3389/fphy.2014.00005}
{\sc A.~Lucas}, {\em Ising formulations of many {NP} problems}, Frontiers in
  Physics, 2 (2014).

\bibitem{Matsubara17}
{\sc S.~Matsubara, H.~Tamura, M.~Takatsu, D.~Yoo, B.~Vatankhahghadim,
  H.~Yamasaki, T.~Miyazawa, S.~Tsukamoto, Y.~Watanabe, K.~Takemoto, and
  A.~Sheikholeslami}, {\em Ising-model optimizer with parallel-trial bit-sieve
  engine}, in Complex, Intelligent, and Software Intensive Systems --
  Proceedings of the 11th International Conference on Complex, Intelligent, and
  Software Intensive Systems (CISIS), Torino, Italy, 2017, p.~432.

\bibitem{meirzada2022lightsolver}
{\sc I.~Meirzada, A.~Kalinski, D.~Furman, T.~Armon, T.~Vaknin, H.~Primack,
  C.~Tradonsky, and R.~Ben-Shlomi}, {\em Lightsolver -- a new quantum-inspired
  solver cracks the 3-regular 3-xorsat challenge}, 2022.

\bibitem{Mihic2018RDSolver}
{\sc K.~Mihić, K.~Ryan, and A.~Wood}, {\em Randomized decomposition solver
  with the quadratic assignment problem as a case study}, INFORMS Journal on
  Computing, 30 (2018), pp.~295--308.

\bibitem{MOHSENI2021}
{\sc M.~Mohseni, D.~Eppens, J.~Str{\"u}mpfer, R.~Marino, V.~S. Denchev, A.~K.
  Ho, S.~V. Isakov, S.~Boixo, F.~Ricci-Tersenghi, and H.~Neven}, {\em
  Nonequilibrium monte carlo for unfreezing variables in hard combinatorial
  optimization}, arXiv.2111.13628,  (2021).

\bibitem{10.1007/978-3-319-39636-1_4}
{\sc D.~Munera, D.~Diaz, and S.~Abreu}, {\em Hybridization as cooperative
  parallelism for the quadratic assignment problem}, in Hybrid
  Metaheuristics, M.~J. Blesa, C.~Blum, A.~Cangelosi, V.~Cutello, A.~DI~NUOVO,
  M.~Pavone, and E.-G. Talbi, eds., Cham, 2016, Springer International
  Publishing, pp.~47--61.

\bibitem{Okuyama19}
{\sc T.~Okuyama, T.~Sonobe, K.~Kawarabayashi, and M.~Yamaoka}, {\em Binary
  optimization by momentum annealing}, Physics Review E, 100 (2019), p.~012111.

\bibitem{Okuyama16}
{\sc T.~Okuyama, C.~{Yoshimura}, M.~{Hayashi}, and M.~{Yamaoka}}, {\em
  Computing architecture to perform approximated simulated annealing for
  {I}sing models}, in IEEE International Conference on Rebooting Computing
  (ICRC), 2016, pp.~1--8.

\bibitem{pedretti2025solvingSAT}
{\sc G.~Pedretti, F.~Böhm, T.~Bhattacharya, A.~Heittman, X.~Zhang, M.~Hizzani,
  G.~Hutchinson, D.~Kwon, J.~Moon, E.~Valiante, I.~Rozada, C.~E. Graves,
  J.~Ignowski, M.~Mohseni, J.~P. Strachan, D.~Strukov, R.~Beausoleil, and T.~V.
  Vaerenbergh}, {\em Solving {B}oolean satisfiability problems with resistive
  content addressable memories}, arXiv:2501.07733,  (2025).

\bibitem{Phillips1994}
{\sc A.~T. Phillips and J.~B. Rosen}, {\em A quadratic assignment formulation
  of the molecular conformation problem}, Journal of Global Optimization, 4
  (1994), pp.~229--241.

\bibitem{rao2023thousands}
{\sc M.~Rao, H.~Tang, J.~Wu, W.~Song, M.~Zhang, W.~Yin, Y.~Zhuo, F.~Kiani,
  B.~Chen, X.~Jiang, H.~Liu, H.~Y. Chen, R.~Midya, F.~Ye, H.~Jiang, Z.~Wang,
  M.~Wu, M.~Hu, H.~Wang, Q.~Xia, and J.~J. Yang}, {\em Thousands of conductance
  levels in memristors integrated on cmos}, Nature, 615 (2023), pp.~823--829.

\bibitem{10.5120/16825-6584}
{\sc A.~R. Ravi Kumar~Bhati}, {\em Quadratic assignment problem and its
  relevance to the real world: A survey}, International Journal of Computer
  Applications, 96 (2014), pp.~42--47.

\bibitem{Roth_2020_ECC}
{\sc R.~M. {Roth}}, {\em {Analog Error-Correcting Codes}}, IEEE Transactions on
  Information Theory, 66 (2020), pp.~4075--4088.

\bibitem{Sebastian2020Memory}
{\sc A.~Sebastian, M.~Le~Gallo, R.~Khaddam-Aljameh, and E.~Eleftheriou}, {\em
  Memory devices and applications for in-memory computing}, Nature
  Nanotechnology, 15 (2020), pp.~529--544.

\bibitem{walkSATBreak_KautzSelmanCohen}
{\sc B.~Selman, H.~Kautz, and B.~Cohen}, {\em Noise strategies for improving
  local search}, Proceedings of the National Conference on Artificial
  Intelligence, 1 (1999).

\bibitem{Seva2025Bed}
{\sc R.~R. Seva, D.~E. Cruz, and L.~M.~S. Tejero}, {\em Bed allocation model
  for infection control during a pandemic}, DLSU Business \& Economics Review,
  34 (2025).

\bibitem{10.1145/3007787.3001139}
{\sc A.~Shafiee, A.~Nag, N.~Muralimanohar, R.~Balasubramonian, J.~P. Strachan,
  M.~Hu, R.~S. Williams, and V.~Srikumar}, {\em Isaac: a convolutional neural
  network accelerator with in-situ analog arithmetic in crossbars}, SIGARCH
  Comput. Archit. News, 44 (2016), p.~14–26.

\bibitem{doi:10.1126/science.adi9405}
{\sc W.~Song, M.~Rao, Y.~Li, C.~Li, Y.~Zhuo, F.~Cai, M.~Wu, W.~Yin, Z.~Li,
  Q.~Wei, S.~Lee, H.~Zhu, L.~Gong, M.~Barnell, Q.~Wu, P.~A. Beerel, M.~S.-W.
  Chen, N.~Ge, M.~Hu, Q.~Xia, and J.~J. Yang}, {\em Programming memristor
  arrays with arbitrarily high precision for analog computing}, Science, 383
  (2024), pp.~903--910.

\bibitem{53bbb073-5adc-3787-97ba-b0c9dae99c60}
{\sc L.~Steinberg}, {\em The backboard wiring problem: A placement algorithm},
  SIAM Review, 3 (1961), pp.~37--50.

\bibitem{Sat2QUBO16}
{\sc J.~Su, T.~Tu, and L.~He}, {\em A quantum annealing approach for {B}oolean
  satisfiability problem}, in 2016 53nd ACM/EDAC/IEEE Design Automation
  Conference (DAC), 2016, pp.~1--6.

\bibitem{Taillard91}
{\sc E.~Taillard}, {\em Robust tabu search for the quadratic assignment
  problem}, Parallel Computing, 17 (1991), pp.~443--455.

\bibitem{10281389}
{\sc A.~Vasilopoulos, J.~Büchel, B.~Kersting, C.~Lammie, K.~Brew, S.~Choi,
  T.~Philip, N.~Saulnier, V.~Narayanan, M.~Le~Gallo, and A.~Sebastian}, {\em
  Exploiting the state dependency of conductance variations in memristive
  devices for accurate in-memory computing}, IEEE Transactions on Electron
  Devices, 70 (2023), pp.~6279--6285.

\bibitem{8811809}
{\sc N.~Verma, H.~Jia, H.~Valavi, Y.~Tang, M.~Ozatay, L.-Y. Chen, B.~Zhang, and
  P.~Deaville}, {\em In-memory computing: Advances and prospects}, IEEE
  Solid-State Circuits Magazine, 11 (2019), pp.~43--55.

\bibitem{Wang2021Scalable}
{\sc C.~Wang, S.-J. Liang, C.-Y. Wang, Z.-Z. Yang, Y.~Ge, C.~Pan, X.~Shen,
  W.~Wei, Y.~Zhao, Z.~Zhang, B.~Cheng, C.~Zhang, and F.~Miao}, {\em {Scalable
  massively parallel computing using continuous-time data representation in
  nanoscale crossbar array}}, Nature Nanotechnology, 16 (2021), pp.~1079--1085.

\bibitem{WANG2023109220}
{\sc H.~Wang and B.~Alidaee}, {\em A new hybrid-heuristic for large-scale
  combinatorial optimization: A case of quadratic assignment problem},
  Computers \& Industrial Engineering, 179 (2023), p.~109220.

\bibitem{Warren2020TST_QUBO}
{\sc R.~H. Warren}, {\em Solving the traveling salesman problem on a quantum
  annealer}, SN applied sciences, 2 (2020), p.~75.

\bibitem{10067380}
{\sc S.~Xie, M.~Yang, S.~A. Lanham, Y.~Wang, M.~Wang, S.~Oruganti, and J.~P.
  Kulkarni}, {\em 29.2 {S}nap-{SAT}: A one-shot energy-performance-aware
  all-digital compute-in-memory solver for large-scale hard boolean
  satisfiability problems}, in 2023 IEEE International Solid-State Circuits
  Conference (ISSCC), 2023, pp.~420--422.

\bibitem{Yoshimura17}
{\sc C.~Yoshimura, M.~Hayashi, T.~Okuyama, and M.~Yamaoka}, {\em Implementation
  and evaluation of {FPGA}-based annealing processor for {I}sing model by use
  of resource sharing}, International Journal of Networking and Computing, 7
  (2017), pp.~154--172.

\bibitem{zhang2024distributedbinaryoptimizationinmemory}
{\sc X.~Zhang, F.~B\"ohm, E.~Valiante, M.~Noori, T.~V. Vaerenbergh, C.-W. Yang,
  G.~Pedretti, M.~Mohseni, R.~Beausoleil, and I.~Rozada}, {\em {D}istributed
  {B}inary {O}ptimization with {I}n-{M}emory {C}omputing: {A}n {A}pplication
  for the {SAT} {P}roblem}, arXiv:2409.09152,  (2024).

\end{thebibliography}

\end{document}